\documentclass[11pt,twoside,reqno]{amsart}

\usepackage{microtype}
\usepackage{cite}
\usepackage[OT1]{fontenc}
\usepackage{type1cm}
\usepackage{amsmath}
\usepackage{amssymb}
\usepackage[shortlabels]{enumitem}
\usepackage{comment}
\usepackage{xcolor}
\usepackage{dsfont}
\usepackage{aliascnt}
\usepackage{tikz}
\usetikzlibrary{calc,patterns,angles,quotes,arrows}
\usetikzlibrary{arrows.meta}
\definecolor{darkgreen}{rgb}{0.0, 0.4, 0.0}

\usepackage{geometry}
\usepackage{hyperref}
\hypersetup{
colorlinks=true,
linkcolor=black,
anchorcolor=black,
citecolor=black,
filecolor=black,      
menucolor=black,
runcolor=black,
urlcolor=black,
}
\usepackage[capitalize,noabbrev,nameinlink]{cleveref}

\numberwithin{equation}{section}
\crefformat{equation}{(#2#1#3)}
\crefformat{figure}{#2Figure #1#3}
\crefformat{enumi}{(#2#1#3)}
\crefformat{theorem}{#2Theorem~#1#3}
\crefformat{lemma}{#2Lemma~#1#3}
\crefformat{proposition}{#2Proposition~#1#3}
\crefformat{corollary}{#2Corollary~#1#3}
\crefformat{question}{#2Question~#1#3}
\crefformat{claim}{#2Claim~#1#3}
\crefformat{example}{#2Example~#1#3}
\crefformat{section}{#2§#1#3}
\crefformat{subsection}{#2§#1#3}

\theoremstyle{plain}
\newtheorem{theorem}{Theorem}[section]

\newaliascnt{corollary}{theorem}
\newtheorem{corollary}[corollary]{Corollary}
\aliascntresetthe{corollary}
\newaliascnt{proposition}{theorem}
\newtheorem{proposition}[proposition]{Proposition}
\aliascntresetthe{proposition}

\newaliascnt{lemma}{theorem}
\newtheorem{lemma}[lemma]{Lemma}
\aliascntresetthe{lemma}
\newaliascnt{conjecture}{theorem}

\aliascntresetthe{conjecture}
\newaliascnt{question}{theorem}
\newtheorem{question}[question]{Question}
\aliascntresetthe{question}
\newtheorem{tmptheoremno}{Theorem}
\newtheorem{tmppropositionno}{Proposition}
\newtheorem{tmplemmano}{Lemma}

\theoremstyle{remark}
\newaliascnt{remark}{theorem}

\aliascntresetthe{remark}
\newaliascnt{remarks}{theorem}

\aliascntresetthe{remarks}
\newaliascnt{example}{theorem}
\newtheorem{example}[example]{Example}
\aliascntresetthe{example}

\theoremstyle{definition}
\newaliascnt{definition}{theorem}

\aliascntresetthe{definition}

\newcommand{\BB}{\mathcal{B}}
\newcommand{\CC}{\mathcal{C}}
\newcommand{\DD}{\mathcal{D}}

\newcommand{\II}{\mathcal{I}}
\newcommand{\JJ}{\mathcal{J}}
\newcommand{\KK}{\mathcal{K}}
\newcommand{\LL}{\mathcal{L}}
\newcommand{\MM}{\mathcal{M}}
\newcommand{\OO}{\mathcal{O}}
\newcommand{\PP}{\mathcal{P}}
\newcommand{\QQ}{\mathcal{Q}}

\newcommand{\VV}{\mathcal{V}}
\newcommand{\WW}{\mathcal{W}}
\newcommand{\R}{\mathbb{R}}
\newcommand{\RP}{\mathbb{RP}}

\newcommand{\C}{\mathbb{C}}
\newcommand{\Q}{\mathbb{Q}}
\newcommand{\Z}{\mathbb{Z}}
\newcommand{\N}{\mathbb{N}}
\newcommand{\hhh}{\mathtt{h}}
\newcommand{\iii}{\mathtt{i}}
\newcommand{\jjj}{\mathtt{j}}
\newcommand{\kkk}{\mathtt{k}}

\newcommand{\ppp}{\mathtt{p}}
\newcommand{\qqq}{\mathtt{q}}
\newcommand{\www}{\mathtt{w}}
\newcommand{\eps}{\varepsilon}
\newcommand{\fii}{\varphi}

\newcommand{\la}{\langle}
\newcommand{\ra}{\rangle}
\newcommand{\A}{\mathsf{A}}
\newcommand{\B}{\mathsf{B}}

\newcommand{\dd}{\,\mathrm{d}}

\expandafter\let\expandafter\ge\csname geqslant\endcsname
\expandafter\let\expandafter\le\csname leqslant\endcsname
\expandafter\let\expandafter\geq\csname geqslant\endcsname
\expandafter\let\expandafter\leq\csname leqslant\endcsname

\DeclareMathOperator{\dimloc}{dim_{loc}}
\DeclareMathOperator{\udimloc}{\overline{dim}_{loc}}
\DeclareMathOperator{\ldimloc}{\underline{dim}_{loc}}

\DeclareMathOperator{\dimm}{dim_M}
\DeclareMathOperator{\udimm}{\overline{dim}_M}
\DeclareMathOperator{\ldimm}{\underline{dim}_M}

\DeclareMathOperator{\dimh}{dim_H}
\DeclareMathOperator{\udimh}{\overline{dim}_H}
\DeclareMathOperator{\ldimh}{\underline{dim}_H}

\DeclareMathOperator{\udimp}{\overline{dim}_p}
\DeclareMathOperator{\ldimp}{\underline{dim}_p}

\DeclareMathOperator{\diml}{dim_L}
\DeclareMathOperator{\dimaff}{dim_{aff}}

\DeclareMathOperator{\dime}{dim_e}
\DeclareMathOperator{\udime}{\overline{dim}_e}
\DeclareMathOperator{\ldime}{\underline{dim}_e}

\DeclareMathOperator{\Ker}{Ker}

\DeclareMathOperator{\Mat}{Mat}
\DeclareMathOperator{\linspan}{span}
\DeclareMathOperator{\dist}{dist}
\DeclareMathOperator{\diam}{diam}
\DeclareMathOperator{\diag}{diag}
\DeclareMathOperator{\proj}{proj}

\DeclareMathOperator{\spt}{spt}

\DeclareMathOperator{\tr}{tr}

\DeclareMathOperator{\GL}{GL}
\DeclareMathOperator{\PGL}{PGL}
\DeclareMathOperator{\SL}{SL}
\DeclareMathOperator{\SO}{SO}
\DeclareMathOperator{\GO}{O}

\allowdisplaybreaks

\begin{document}

\title{Projections of self-affine sets onto lines}

\author{Bal\'azs B\'ar\'any}
\address[Bal\'azs B\'ar\'any]
        {Department of Stochastics \\
         Institute of Mathematics \\
         Budapest University of Technology and Economics \\
         M\H{u}egyetem rkp 3 \\
         H-1111 Budapest \\
         Hungary}
\email{balubsheep@gmail.com}

\author{Antti K\"aenm\"aki}
\address[Antti K\"aenm\"aki]
        {University of Eastern Finland \\
         Department of Physics and Mathematics \\
         P.O.\ Box 111 \\
         FI-80101 Joensuu \\
         Finland}
\email{antti@kaenmaki.net}

\author{Istv\'an Kolossv\'ary}
\address[Istv\'an Kolossv\'ary]
        {HUN-REN Alfr\'ed R\'enyi Institute of Mathematics \\
         Re\'altanoda u. 13–15 \\
         1053 Budapest \\
         Hungary}
\email{kolossvary.istvan@renyi.hu}

\thanks{B.\ B\'ar\'any was supported by the grants NKFI FK134251, K142169, and the grant NKFI KKP144059 ``Fractal geometry and applications''. I.\ Kolossv\'ary was supported by the European Research Council Marie Sk\l odowska-Curie Actions Postdoctoral Fellowship $\#101109013$ and the Hungarian NRDI Office grant K142169 and STARTING152587.}
\date{\today}

\begin{abstract}
  We prove an all-directions Marstrand--Mattila projection theorem for self-affine measures and sets in $\R^d$. Under exponential separation, together with proximality and strong irreducibility assumptions on the linear parts, the projection of a self-affine measure onto every line has the expected Hausdorff dimension. If the proximality assumption is strengthened to strong pinching, then the same conclusion holds for the self-affine set $X$ itself, without any separation assumption. In the plane, strong irreducibility of the linear parts alone suffices, and this is sharp. As a corollary, if $X$ additionally has upper Minkowski dimension at most one, then its Minkowski dimension exists and equals its Hausdorff dimension, giving a partial affirmative answer to the folklore question of whether the Minkowski dimension exists for every self-affine set.
\end{abstract}

\maketitle

\tableofcontents

\section{Introduction and main results} \label{sec:intro}

We advance the understanding of two significant problems related to self-affine sets, which have attracted substantial attention in recent decades. The first addresses the set of exceptional directions in the seminal Marstrand--Mattila projection theorem \cite{Marstrand1954,Mattila1975}. The second tackles the long-standing open question regarding the existence of the Minkowski dimension for self-affine sets.

The \emph{general linear group} $\GL(d,\R)$ consists of all $d \times d$ invertible matrices with real entries. An \emph{affine iterated function system} on $\R^d$ is a finite tuple $\Phi = (\fii_1,\ldots,\fii_N)$, where $N \ge 2$, and each $\fii_i \colon \R^d \to \R^d$ is a contractive affine map defined by
\begin{equation*}
  \fii_i(x) = A_ix + t_i,
\end{equation*}
with $A_i \in \GL(d,\R)$ having operator norm strictly less than one, $\|A_i\| < 1$, and $t_i \in \R^d$. A fundamental result by Hutchinson \cite{Hutchinson1981} guarantees that for any such affine iterated function system, there exists a unique non-empty compact set $X \subset \R^d$, called the \emph{self-affine set}, satisfying
\begin{equation} \label{eq:self-affine-set-def}
  X = \bigcup_{i = 1}^N \fii_i(X).
\end{equation}
To ensure non-triviality, we assume that $X$ is not a singleton, which is equivalent to the maps $\fii_i$ having no common fixed point. Throughout, we assume $d \ge 2$ for convenience. By convention, a self-affine set $X$ is always associated with the tuple $\Phi = (\fii_1,\ldots,\fii_N)$ of affine maps that defines it. The corresponding tuple of contractive matrices is denoted by $\A = (A_1,\ldots,A_N) \in \GL(d,\R)^N$.

Our main results are two projection theorems, one for self-affine measures and one for self-affine sets, together with a sharp planar refinement of the set theorem, all under conditions on the linear maps that hold for a typical tuple of matrices. The first shows that, under a suitable irreducibility assumption on the linear maps and a mild separation condition, every projection onto a line has the expected Hausdorff dimension for the corresponding self-affine measure. The second shows that the same conclusion holds for the set itself, for a generic choice of linear maps and without any separation assumption or any restriction on the translation vectors $t_1,\ldots,t_N$. The proof first establishes the measure theorem by combining the Furstenberg-theoretic input with a fixed-direction local entropy averages argument. It then turns to self-affine sets by extracting suitable smaller subsystems from strongly pinching systems and applying the measure theorem to those subsystems. A direct corollary shows that if a self-affine set $X$ has upper Minkowski dimension at most one, then the Minkowski dimension of $X$ exists and equals its Hausdorff dimension. We now state these results and place them in context.

\subsection{Marstrand--Mattila projection theorem}

The \emph{real projective space} of dimension $d-1$, denoted $\RP^{d-1}$, is the quotient space $(\R^d \setminus \{0\}) /_\sim$, where the equivalence relation $\sim$ identifies vectors $x, y \in \R^d \setminus \{0\}$ if $y = \lambda x$ for some $\lambda \neq 0$. Thus, $\RP^{d-1}$ can be regarded as the space of all lines passing through the origin in $\R^d$. The general linear group $\GL(d,\R)$ acts on $\RP^{d-1}$ via $A\linspan\{x\} = \linspan\{Ax\}$ for $A \in \GL(d, \R)$ and $x \in \R^d \setminus \{0\}$. This action maps the line through the origin spanned by $x$ to the line spanned by $Ax$. The Haar measure $\gamma_{d,1}$ on $\RP^{d-1}$ is the pushforward of the spherical measure under the quotient map.

We denote the Hausdorff dimension by $\dimh$ and refer to the books \cite{Falconer1997,Mattila1995} for its definition and basic properties. The upper Minkowski dimension $\udimm$, which enters \cref{thm:projection,thm:projection-two-dimensional}, is defined precisely in \cref{sec:IntroMinkowski}. The \emph{orthogonal projection} $\proj_V$ onto a line $V \in \RP^{d-1}$ is the linear map $\R^d \to V$ with kernel $V^\bot$ and identity restriction to $V$. Since $\proj_V$ is Lipschitz and $\proj_V(X) \subseteq V$, the Hausdorff dimension of the projection is bounded by
\begin{equation} \label{eq:hausdorff-projection-lipschitz-estimate}
  \dimh(\proj_V(X)) \le \min\{1, \dimh(X)\}
\end{equation}
for all $X \subseteq \R^d$ and $V \in \RP^{d-1}$. The Marstrand--Mattila projection theorem \cite{Marstrand1954,Mattila1975} states that for a Borel set $X \subseteq \R^d$, the projection satisfies
\begin{equation*}
  \dimh(\proj_{V}(X)) = \min\{1,\dimh(X)\}
\end{equation*}
for $\gamma_{d,1}$-almost all $V \in \RP^{d-1}$. An analogous almost-every-direction statement holds for finite Borel measures; see \cite{HuTaylor1994}. Thus, when $\dimh(X) \leq 1$, the projection typically preserves the dimension, whereas if $\dimh(X) > 1$, it achieves the maximum dimension of one on the target line. More generally, the theorem applies to orthogonal projections onto $k$-dimensional subspaces, but this paper focuses solely on projections onto lines, so the above formulation suffices. Numerous extensions and refinements of the theorem have been a vibrant area of research. A comprehensive survey of these results is beyond the scope of this paper; interested readers may refer to \cite{Kaufman1968,HuntKaloshin1997,HuTaylor1994,HochmanShmerkin2012,HochmanShmerkin2015,PeresSchlag2000,Bourgain2003,Bourgain2010,OrponenShmerkin2023,OrponenShmerkin2023-preprint,Orponen2024-preprint,Wu2025-preprint,RenWang2023-preprint,KaenmakiOrponenVenieri2025} and references therein. Here, we highlight selected results concerning the set of exceptional directions for self-affine sets. 

Falconer and Kempton \cite[Corollary 3.3]{FalconerKempton2017} demonstrated that for many dominated self-affine sets $X \subset \mathbb{R}^2$, the orthogonal projections achieve the dimension $\min\{1, \dimh(X)\}$ in all but at most one direction. B\'ar\'any and K\"aenm\"aki \cite[Theorem 2.3]{BaranyKaenmaki2017} established that projections of self-affine measures in $\mathbb{R}^d$ with a simple Lyapunov spectrum are exact-dimensional and satisfy the Ledrappier-Young formula. This result was a critical component in the work of B\'ar\'any, Hochman, and Rapaport \cite[Theorem 7.1]{BaranyHochmanRapaport2019}, who stated that, for a proximal and strongly irreducible self-affine set $X\subset\mathbb{R}^2$ satisfying the strong open set condition, the set of exceptional directions in the Marstrand--Mattila projection theorem is empty. Their proof, however, contains a slight gap and, as written, yields this only for dominated and irreducible linear parts; we explain and close the gap in \cref{sec:projections-selfaffine-measures}. It follows from Hochman and Rapaport \cite{HochmanRapaport2021} that the strong open set condition can be relaxed to the exponential separation condition. In the reducible case, B\'ar\'any, K\"aenm\"aki, Py\"or\"al\"a, and Wu \cite[Theorem 1.7]{BaranyKaenmakiPyoralaWu2023-preprint} showed that for self-affine measures on planar carpets satisfying the strong separation condition and an irrationality assumption on the entries of the linear parts, there are at most two exceptional directions. Subsequently, Py\"or\"al\"a \cite[Theorem 1.1]{Pyorala2025} refined this result by eliminating the separation assumption.

Under the irreducibility condition of \cite[Theorem 1.1(i)]{FengXie2025preprint}, recent work by Feng and Xie \cite[Theorem 1.1(iii)]{FengXie2025preprint} and Morris and Sert \cite[Theorem 2(b)]{MorrisSert2025preprint} extends the Marstrand--Mattila projection theorem to all linear subspaces of $\R^d$ for Lebesgue almost every choice of translation vectors. However, this result does not generally hold for reducible planar self-affine sets. Specifically, Feng and Xie \cite[Proposition 7.1]{FengXie2025preprint}, together with Morris and Sert \cite[Theorem 2(a)]{MorrisSert2025preprint} and Falconer \cite[Theorem 5.3]{Falconer1988}, provide sufficient conditions under which there exists $V\in\RP^{d-1}$ such that
\begin{equation*}
  \udimm(\proj_V(X)) < \min\{1, \dimh(X)\}
\end{equation*}
for Lebesgue almost every choice of translation vectors. More recently, Allen, K\"aenm\"aki, Prokaj, Simon, and Troscheit \cite{AllenKaenmakiProkajSimonTroscheit2026-preprint} showed that the projection of an irreducible planar ergodic measure that is not strongly irreducible need not be exact-dimensional.

Our first main theorem concerns self-affine measures satisfying the exponential separation condition of \cref{sec:exponential-separation} and suitable proximality and irreducibility assumptions on the linear maps, as introduced in \cref{sec:irred}. Rapaport's entropy-dimension formula \cite[Theorem~1.12]{Rapaport2024} and Feng's exact-dimensionality theorem \cite[Theorem 1.6(ii)]{Feng2023} give the expected projection formula for Furstenberg-almost all directions. The exponential separation hypothesis reflects both the available input and the geometry: our proof starts from Rapaport's projection formula, which is presently available in that setting, and some control on overlaps is unavoidable in any statement involving the Lyapunov dimension $\diml(\nu)$, since $\diml(\nu)$ depends only on the symbolic data and linear parts and does not detect dimension drops caused by overlaps. The theorem below shows that for every projection onto a line, the projected measure is exact-dimensional and has the expected dimension; the Lyapunov dimension and exact-dimensionality are defined in \cref{sec:dimensions-measures}.

\begin{theorem} \label{thm:projection-bernoulli}
  Let $X \subset \R^d$ be a self-affine set satisfying the exponential separation condition, $\nu$ be a fully supported Bernoulli measure, and $\mu$ be the associated self-affine measure such that the associated tuple of matrices is $k$-proximal and strongly $k$-irreducible for all $k \in \{1,\ldots,d-1\}$. Then
  \begin{equation*}
    \dim((\proj_V)_*\mu) = \min\{1,\diml(\nu)\}
  \end{equation*}
  for all $V \in \RP^{d-1}$.
\end{theorem}

The second main theorem is the analogous statement for self-affine sets, with the proximality assumption strengthened to strong pinching, which is defined in \cref{sec:irred}. In that setting, every projection onto a line has the expected Hausdorff dimension. No separation condition is needed because the proof passes to suitable strongly separated subsystems, applies the measure theorem there, and transfers the lower bound back to $X$ by monotonicity.

\begin{theorem} \label{thm:projection}
  Let $X \subset \R^d$ be a self-affine set such that the associated tuple of matrices is strongly pinching and strongly $k$-irreducible for all $k \in \{1,\ldots,d-1\}$. Then
  \begin{equation*}
    \dimh(\proj_V(X)) = \min\{1, \dimh(X)\} = \min\{1, \udimm(X)\}
  \end{equation*}
  for all $V \in \RP^{d-1}$.
\end{theorem}

The hypotheses of \cref{thm:projection} on the linear parts, strong pinching and strong irreducibility in every exterior power, hold for a typical tuple. Both follow from Zariski density of the tuple, which is itself, by \cref{prop:tuples-open-dense}, an open and dense condition on $\GL(d,\R)^N$. Since strong pinching implies proximality in every exterior power, the same genericity also covers the hypotheses of \cref{thm:projection-bernoulli}.

\begin{theorem} \label{thm:assumptions-typical}
  The set of tuples in $\GL(d,\R)^N$ that are strongly pinching and strongly $k$-irreducible for all $k \in \{1,\ldots,d-1\}$ contains an open and dense subset of $\GL(d,\R)^N$.
\end{theorem}

In dimension two, strong pinching is equivalent to proximality under irreducibility, so that the hypotheses of \cref{thm:projection} coincide with the proximality and strong irreducibility used by Hochman and Rapaport \cite{HochmanRapaport2021}. The proof of \cite[Lemma 2.2]{BaranyKaenmakiRossi2021} shows that the only proximal tuples that are not strongly irreducible are simultaneously upper triangular in some basis or consist only of diagonal and antidiagonal matrices. Thus \cref{thm:projection} places the planar hypothesis of \cite{HochmanRapaport2021} into a higher-dimensional framework.

\begin{theorem} \label{thm:projection-two-dimensional}
  Let $X \subset \R^2$ be a self-affine set such that the associated tuple of matrices is strongly irreducible. Then
  \begin{equation*}
    \dimh(\proj_V(X)) = \min\{1, \dimh(X)\} = \min\{1, \udimm(X)\}
  \end{equation*}
  for all $V \in \RP^1$.
\end{theorem}

This result extends the planar projection theorem of Hochman and Rapaport \cite{HochmanRapaport2021} by removing the exponential separation assumption. It also unifies two settings: for proximal tuples the conclusion follows from \cref{thm:projection}, while for non-proximal tuples the system is conjugate to a self-similar one and the conclusion is supplied by the projection theorem of Farkas \cite{Farkas2016}. In particular, \cref{thm:projection-two-dimensional} recovers the planar case of Farkas's theorem and complements it with the proximal, genuinely non-conformal, self-affine sets. Moreover, \cref{thm:projection-two-dimensional} is sharp, in the sense that strong irreducibility cannot be relaxed to irreducibility, even under the strong separation condition; see \cref{ex:sharpness}.

\subsection{Existence of Minkowski dimension}\label{sec:IntroMinkowski}

For a bounded set $A \subseteq \R^d$ and $r>0$, an \emph{$r$-packing} of $A$ is a finite collection $\{B(x_i,r) : x_i \in A\}_i$ of pairwise disjoint closed balls of radius $r>0$ centered at points in $A$. The cardinality of a maximal $r$-packing of $A$ is
\begin{equation*}
  N_r(A) = \max\{k \geq 1 : \{B_i\}_{i = 1}^k \text{ is an $r$-packing of }A\}.
\end{equation*}
The \emph{upper Minkowski dimension} of a bounded set $A \subset \R^d$ is given by
\begin{equation} \label{eq:DefUpperMinkowskiDim}
  \udimm(A) = \limsup_{r \downarrow 0} \frac{\log N_r(A)}{-\log r}.
\end{equation}
The \emph{lower Minkowski dimension} is defined analogously by replacing the $\limsup_{r \downarrow 0}$ in \cref{eq:DefUpperMinkowskiDim} with $\liminf_{r \downarrow 0}$. If the limit exists, then the \emph{Minkowski dimension} of $A$, denoted $\dimm(A)$, is defined as this limit. By convention, the notation $\dimm(A)$ is used only when this limit exists. For a compact set $A \subset \R^d$, it is well-known that $\dimh(A) \le \ldimm(A)$; see, for example, \cite[\S 5.3]{Mattila1995}. Thus, if $\udimm(A) \le \dimh(A)$, then $\dimh(A) = \dimm(A)$.

A long-standing open question in fractal geometry is the following:

\begin{question}\label{ques:DoesMinkExist}
  Does the Minkowski dimension of every self-affine set $X$ exist?
\end{question}

No counterexamples are known for self-affine sets or, more generally, for attractors of iterated function systems consisting of continuously differentiable maps, and the answer is widely conjectured to be affirmative. For self-similar sets and, more generally, sub-self-conformal sets in $\R^d$, Falconer \cite[Theorem 4]{Falconer1989} proved that the Minkowski dimension always exists and equals the Hausdorff dimension, without requiring separation conditions. Feng and Hu \cite[Theorem 8.1]{FengHu2009} extended this result to attractors of weakly conformal iterated function systems. The existence of the Minkowski dimension and its equality with the Hausdorff dimension also hold for all conformal repellers; see \cite{Barreira1996,GatzourasPeres1997,Pesin1997,PalisViana1988,Takens1988}.

In contrast, for smooth non-conformal dynamical systems, the Hausdorff and Minkowski dimensions of invariant sets do not always coincide. This is typically the case for reducible self-affine sets, such as self-affine carpets; see \cite{Baranski2007,Bedford1984,KenyonPeres1996,LalleyGatzouras1992,McMullen1984,PollicottWeiss1994}. However, certain conditions on self-affine sets ensure that the Hausdorff dimension or the lower Minkowski dimension equals the affinity dimension, a natural symbolic upper bound for the upper Minkowski dimension, and thus guarantee the existence of the Minkowski dimension. Falconer \cite[Theorem 5.4]{Falconer1988} established the existence of the Minkowski dimension for self-affine sets $X \subset \R^d$ for Lebesgue almost every choice of translation vectors. In the planar case, Falconer \cite[Corollary 5]{Falconer1992} proved existence for self-affine sets satisfying the open set condition for a connected open set, provided $X$ has a connected component not contained in any line. Hueter and Lalley \cite[Theorem 1.1]{HueterLalley1995} identified a class of planar self-affine sets with dimension less than one for which the Minkowski dimension exists for an open set of translation vectors. K\"aenm\"aki and Shmerkin \cite[Theorem 3.3]{KaenmakiShmerkin2009} introduced a similar class of planar self-affine sets with dimension greater than one. Falconer and Kempton \cite[Theorem 1.2]{FalconerKempton2018} extended the approach of \cite[Corollary 5]{Falconer1992} by proving existence for dominated self-affine sets $X \subset \R^2$ satisfying the strong separation condition, provided the projection onto Furstenberg lines has positive Lebesgue measure. B\'ar\'any, Hochman, and Rapaport \cite[Theorem 1.1]{BaranyHochmanRapaport2019} showed that for proximal and strongly irreducible self-affine sets $X \subset \R^2$ satisfying the strong open set condition, the Minkowski dimension exists. Hochman and Rapaport \cite[Theorem 1.1]{HochmanRapaport2021} further relaxed this to the exponential separation condition. More recently, Kirat \cite[Theorem 1.2]{Kirat2026-preprint} established the existence of the Minkowski dimension, without any separation condition, for the special class of integral self-affine sets, that is, sets generated by an integer expanding matrix and a finite set of integer translation vectors.

For iterated function systems with infinitely many conformal maps, Mauldin and Urbański \cite[Theorem 3.1 and Example 5.2]{MauldinUrbanski1996} demonstrated that infinitely generated self-conformal sets may have distinct Hausdorff and upper Minkowski dimensions; see also \cite[Theorem 2.11]{MauldinUrbanski1999}. Banaji and Rutar \cite[Theorem A]{BanajiRutar2024-preprint} recently showed that, in this case, the Minkowski dimension may not exist. In the infinite affine case, K\"aenm\"aki and Morris \cite[\S 1.4]{KaenmakiMorris2025-toappear} identified classes of infinitely generated self-affine sets for which the Minkowski dimension exists. Conversely, Jurga \cite[Theorem 1.2]{Jurga2023} provided an example of an infinitely generated self-affine set where the Minkowski dimension does not exist, also applicable to sub-self-affine sets. Earlier, K\"aenm\"aki and Vilppolainen \cite[Theorem 5.2]{KaenmakiVilppolainen2010} showed that the Minkowski dimension exists for sub-self-affine sets for Lebesgue almost every choice of translation vectors. Finally, Baker, Banaji, Feng, Lai, and Xiong \cite[Theorem 1.3]{BakerBanajiFengLaiXiong2025-preprint} recently constructed an iterated function system with two non-differentiable bi-Lipschitz contractions on $\R$ whose attractor has a non-existent Minkowski dimension.

As a direct corollary of \cref{thm:projection}, we obtain a partial answer to \cref{ques:DoesMinkExist}. To our knowledge, this is the first general existence result of this type that requires neither a separation condition nor any restriction on the translation vectors.

\begin{corollary} \label{thm:minkowski-dimension}
  Let $X \subset \R^d$ be a self-affine set with $\udimm(X) \le 1$ such that the associated tuple of matrices is strongly pinching and strongly $k$-irreducible for all $k \in \{1,\ldots,d-1\}$. Then the Minkowski dimension of $X$, $\dimm(X)$, exists and
  \begin{equation*}
    \dimh(X) = \dimm(X).
  \end{equation*}
\end{corollary}

\begin{proof}
  By \cref{eq:hausdorff-projection-lipschitz-estimate}, \cref{thm:projection}, and the assumption $\udimm(X) \leq 1$, we have
  \begin{equation*}
    \ldimm(X) \geq \dimh(X) \geq \dimh(\proj_V(X)) = \min\{1, \udimm(X)\} = \udimm(X),
  \end{equation*}
  which establishes the claim.
\end{proof}

Although the relationship between the Minkowski dimension and the affinity dimension remains unclear, \cref{thm:minkowski-dimension} extends the dimension result of Hochman and Rapaport \cite{HochmanRapaport2021} by removing the requirement for any separation condition.

\subsection{Outline of paper}
The paper is organized as follows. In \cref{sec:preli}, we collect the background on self-affine sets and measures, exponential separation, dimensions, local entropy averages, multilinear algebra, irreducibility, proximality, and Zariski topology. In \cref{sec:projections-selfaffine-measures}, we prove \cref{thm:projection-bernoulli}, the all-directions projection theorem for self-affine measures: we introduce the Furstenberg measure, record the projection theorem in Furstenberg-typical directions, and verify the local entropy averages hypothesis for fixed directions to obtain the lower bound in every direction. We then turn to the set theorem. In \cref{sec:zariski-density}, we prove \cref{thm:assumptions-typical}, that the hypotheses of the set theorem hold for a typical tuple: we show that Zariski dense tuples are strongly $k$-irreducible in every exterior power and strongly pinching, and that Zariski density is a generic condition. The projective quotient, a non-resonance lemma, and a criterion combining the projective quotient and determinant supply the algebraic input. In the concluding subsection \cref{sec:planar-converse} we treat the two-dimensional case, establishing the planar converse and using examples to show that its hypotheses are sharp. In \cref{sec:subsystems}, we extend the planar dominated-subsystem construction to arbitrary dimensions and then, by geometric selection, extract strongly separated subsystems that preserve the relevant dimensional and irreducibility properties. Finally, in \cref{sec:proof-main-result}, we prove \cref{thm:projection} by applying the measure theorem to these subsystems, and we deduce the planar theorem \cref{thm:projection-two-dimensional} from it together with the two-dimensional analysis of \cref{sec:planar-converse}.


\section{Preliminaries} \label{sec:preli}

This section fixes notation for self-affine sets and measures, recalls the relevant dimension-theoretic and multilinear background, and records the irreducibility, proximality, and local-entropy tools used later. Two items are original rather than recalled: \cref{lem:algebraic-esc}, which generalizes the planar exponential separation criterion of Gamburd, Jakobson, and Sarnak \cite{GamburdJakobsonSarnak1999} to higher dimensions, and \cref{prop:simultaneous-escape}, which characterizes strong irreducibility in every exterior power through the simultaneous escape property.

\subsection{Self-affine sets and measures} \label{sec:self-affine-set-measure}
Fix an integer $N \ge 2$ and let $\mathcal{S}(\A) = \la A_1,\ldots,A_N \ra$ denote the subsemigroup of $\GL(d,\R)$ generated by the tuple $\A = (A_1,\ldots,A_N) \in \GL(d,\R)^N$. To index elements of the subsemigroup, we adopt a standard convention of using a distinct alphabet. This alphabet also provides a representation for the associated self-affine set.

Let $\II = \{1,\ldots,N\}$ be a finite alphabet. The set $\II^\N$ consists of infinite words $\iii = i_1 i_2 \cdots$ with $i_k \in \II$. The set of finite words is $\II^* = \bigcup_{n \ge 0} \II^n$, where $\II^n = \{i_1 \cdots i_n : i_k \in \II\}$ for $n \geq 1$, and $\II^0 = \{\varnothing\}$ contains the empty word $\varnothing$. The concatenation of $\iii \in \II^*$ and $\jjj \in \II^* \cup \II^\N$, denoted $\iii \jjj \in \II^* \cup \II^\N$, has $\iii$ as a prefix. The longest common prefix of $\iii, \jjj \in \II^* \cup \II^\N$ is denoted $\iii \land \jjj$. The set $\II^*$ forms a free monoid under concatenation, with $\varnothing$ as the identity satisfying $\varnothing \iii = \iii \varnothing = \iii$ for all $\iii \in \II^*$.

The subsemigroup $\mathcal{S}(\A)$ is given by $\{A_\iii : \iii \in \II^* \setminus \{\varnothing\}\}$, where $A_\iii = A_{i_1} \cdots A_{i_n}$ for $\iii = i_1 \cdots i_n$ and $n \geq 1$. We also write 
\begin{equation*}
  A_{\iii}^\top = A_{i_n}^\top \cdots A_{i_1}^\top,
\end{equation*}
where $A^\top$ is the transpose of $A \in \GL(d,\R)$. If $\Phi = (\fii_1,\ldots,\fii_N)$ is an affine iterated function system, then the composition map is denoted by
\begin{equation*}
  \fii_\iii = \fii_{i_1} \circ \cdots \circ \fii_{i_n}
\end{equation*}
for $\iii = i_1 \cdots i_n$ and $n \geq 1$. We let $A_\varnothing$ and $\fii_\varnothing$ denote the identities.

The \emph{left shift} $\sigma \colon \II^\N \to \II^\N$ is defined by $\sigma(i_1i_2\cdots) = i_2i_3\cdots$ for all infinite words $\iii = i_1i_2\cdots$. For a finite word $\iii = i_1 \cdots i_n$, the left shift is $\sigma(\iii) = i_2 \cdots i_n$ if $n \geq 2$, and $\sigma(\iii) = \varnothing$ if $n = 1$. For a finite word $\iii = i_1 \cdots i_n$, its \emph{truncation} of the last symbol is $\iii^- = i_1 \cdots i_{n-1}$ if $n \geq 2$, and $\iii^- = \varnothing$ if $n = 1$. The length of a word is defined as follows: $|\iii| = n$ for $\iii \in \II^n$, $|\varnothing| = 0$ for the empty word $\varnothing$, and $|\iii| = \infty$ for infinite words $\iii \in \II^\N$. For a finite word $\iii \in \II^*$, the \emph{cylinder set} is
\begin{equation*}
  [\iii] = \{ \iii \jjj \in \II^\N : \jjj \in \II^\N \};
\end{equation*}
i.e., the set of all infinite words with prefix $\iii$. The space $\II^\N$ is a compact topological space under the topology generated by the cylinder sets as a basis. Furthermore, the cylinder sets are open and closed in this topology and they generate the Borel $\sigma$-algebra on $\II^\N$. For a word $\jjj \in \II^* \cup \II^\N$ and $0 \leq n \leq |\jjj|$, the restriction $\jjj|_n \in \II^n$ is the unique prefix $\iii \in \II^n$ of $\jjj$ such that $[\jjj] \subseteq [\iii]$ if $\jjj \in \II^*$, or $\jjj \in [\iii]$ if $\jjj \in \II^\N$.

The \emph{canonical projection} $\pi\colon \II^\N \to X$, where $X \subset \R^d$ is the self-affine set from \cref{eq:self-affine-set-def} associated with $\Phi$, is defined by
\begin{equation*}
  \pi(\iii) = \sum_{n = 1}^\infty A_{\iii|_{n-1}} t_{i_n} = \lim_{n \to \infty} \fii_{\iii|_n}(0)
\end{equation*}
for all $\iii = i_1i_2\cdots \in \II^\N$. The map $\pi$ is continuous, and its image satisfies $\pi(\II^\N) = X$. Separation conditions allow simple interplay between $\II^\N$ and $X$. The affine iterated function system $\Phi$ satisfies the \emph{strong separation condition} if $\fii_i(X) \cap \fii_j(X) = \emptyset$ for all $i \ne j$, which is equivalent to
\begin{equation*}
  \min_{i \ne j}\dist(\fii_i(X),\fii_j(X)) > 0.
\end{equation*}
As $\pi([\iii]) = \fii_\iii(X)$ for all $\iii \in \II^*$, the strong separation condition is characterized by the requirement that $\pi$ is bijective.

Let $\MM_\sigma(\II^\N)$ denote the collection of all $\sigma$-invariant Borel probability measures on $\II^\N$. A measure $\nu$ on $\II^\N$ is \emph{fully supported} if every cylinder set has positive measure, i.e., $\nu([\iii]) > 0$ for all $\iii \in \II^*$. Given a probability vector $(p_1,\ldots,p_N)$, the \emph{Bernoulli measure} $\nu \in \MM_\sigma(\II^\N)$ is defined by
\begin{equation*}
  \nu([\iii]) = p_{i_1} \cdots p_{i_n}
\end{equation*}
for all $\iii = i_1 \cdots i_n \in \II^n$ and $n \geq 1$. This measure is fully supported if and only if $p_i > 0$ for all $i$. Hutchinson \cite{Hutchinson1981} proved that for each probability vector $(p_1,\ldots,p_N)$, the pushforward measure $\mu=\pi_*\nu$ under the canonical projection $\pi \colon \II^\N \to X$ is the unique Borel probability measure on $X$ satisfying
\begin{equation*}
  \mu = \sum_{i = 1}^N p_i (\fii_i)_*\mu.
\end{equation*}
This measure $\mu$ is called a \emph{self-affine measure}.

\subsection{Exponential separation} \label{sec:exponential-separation}

We next recall the separation hypothesis used in \cref{thm:projection-bernoulli}. Let $\|\cdot\|$ be a norm on the vector space of affine maps from $\R^d$ into itself. An affine iterated function system $\Phi$ satisfies the \emph{exponential separation condition} if there exists a constant $c>0$ such that 
\begin{equation*}
  \|\fii_\iii-\fii_\jjj\| \ge c^n
\end{equation*}
for infinitely many $n \ge 1$ and distinct $\iii,\jjj \in \II^n$. The bound forces distinct words of equal length to have distinct associated maps, and in fact $\Phi$ generates a free semigroup: if $\fii_\iii = \fii_\jjj$ for some distinct $\iii,\jjj \in \II^* \setminus \{\varnothing\}$, then $\fii_{\iii\jjj} = \fii_{\jjj\iii}$, so either $\iii\jjj$ and $\jjj\iii$ are distinct words of equal length sharing the same map, which is impossible, or $\iii\jjj = \jjj\iii$, in which case $\iii = \kkk^a$ and $\jjj = \kkk^b$ are powers of a common word $\kkk$ by \cite[Proposition~1.3.2]{Lothaire1997}, with $a \ne b$ because $\iii \ne \jjj$, forcing $\fii_\kkk^{|a-b|}$ to be the identity, which is impossible since $\fii_\kkk$ is a strict contraction. 

By \cite[\S 6.2]{BaranyHochmanRapaport2019}, the strong separation condition implies exponential separation. The exponential separation condition also allows severe overlapping: if $\fii_\iii \ne \fii_\jjj$ for all distinct $\iii, \jjj \in \II^*$ and the entries of $A_i$ and $t_i$ are algebraic numbers, then $\Phi$ satisfies the exponential separation condition. The following lemma generalizes the planar result of Gamburd, Jakobson, and Sarnak \cite[Proposition 4.3]{GamburdJakobsonSarnak1999} to higher dimensions.

\begin{lemma} \label{lem:algebraic-esc}
  Let $\Phi = (\fii_1,\ldots,\fii_N)$ be an affine iterated function system on $\R^d$ such that $\fii_i(x) = A_ix+t_i$ for all $i \in \II$, the associated tuple $\A = (A_1,\ldots,A_N)$ of matrices has algebraic entries, and the translation vectors $t_1,\ldots,t_N$ have algebraic entries. If $\fii_\iii \ne \fii_\jjj$ for all distinct $\iii,\jjj \in \II^*$, then $\Phi$ satisfies the exponential separation condition.
\end{lemma}

\begin{proof}
  Let $K$ be a finite extension of $\Q$ containing every entry of every $A_i$ and $t_i$, that is, a field containing $\Q$ and having finite dimension as a vector space over $\Q$. Let $\mathcal O_K$ denote the set of algebraic integers in $K$, that is, the set of all $\alpha \in K$ that are roots of a polynomial with integer coefficients and leading coefficient one. For each $i \in \II$, define
  \begin{equation*}
    \widehat\fii_i =
    \begin{pmatrix}
      A_i & t_i \\
      0 & 1
    \end{pmatrix}
    \in \Mat_{d+1}(K),
  \end{equation*}
  where $\Mat_{d+1}(K)$ denotes the set of $(d+1) \times (d+1)$ matrices with entries in $K$. Then $\widehat\fii_\iii$ is the block matrix of the affine map $\fii_\iii$ for every word $\iii \in \II^*$. Choose an integer $q \ge 1$ such that every entry of $q\widehat\fii_i$ lies in $\mathcal O_K$ for each $i \in \II$. If $\iii = i_1\cdots i_n \in \II^n$, then
  \begin{equation} \label{eq:algebraic-esc-product}
    q^n\widehat\fii_\iii = (q\widehat\fii_{i_1})\cdots(q\widehat\fii_{i_n}),
  \end{equation}
  and therefore every entry of $q^n\widehat\fii_\iii$ lies in $\mathcal O_K$ by \cref{eq:algebraic-esc-product}.

  Regard $K$ as a subfield of $\C$ and let $\Sigma$ be the set of embeddings $\sigma \colon K \to \C$. Each such $\sigma$ is a ring homomorphism that fixes $\Q$ pointwise; being $\Q$-linear, it is determined by its values on a basis $\omega_1,\ldots,\omega_m$ of $K$ over $\Q$. Each $\omega_j$ satisfies a nonzero polynomial $p_j$ with rational coefficients, since the powers $1,\omega_j,\omega_j^2,\ldots$ cannot all be linearly independent in the finite-dimensional space $K$; applying $\sigma$ to the relation $p_j(\omega_j)=0$ yields $p_j(\sigma(\omega_j))=0$, so $\sigma(\omega_j)$ is one of the finitely many roots of $p_j$. There are therefore only finitely many possibilities for the tuple $(\sigma(\omega_1),\ldots,\sigma(\omega_m))$, and hence $\Sigma$ is finite. Moreover, since finite extensions of $\Q$ are separable, $\Sigma$ has exactly as many elements as the dimension of $K$ as a vector space over $\Q$. Equip $\Mat_{d+1}(\C)$ with the operator norm induced by $\|\cdot\|_\infty$ on $\C^{d+1}$, again denoted by $\|\cdot\|$. Define
  \begin{equation} \label{eq:algebraic-esc-M}
    M = \max_{\sigma \in \Sigma} \max_{i \in \II} \|\sigma(q\widehat\fii_i)\| < \infty.
  \end{equation}
  Then $M \ge q \ge 1$, and it follows from \cref{eq:algebraic-esc-product} and \cref{eq:algebraic-esc-M} that for every $\sigma \in \Sigma$ and $\iii \in \II^n$,
  \begin{equation} \label{eq:algebraic-esc-growth}
    \|\sigma(q^n\widehat\fii_\iii)\| \le M^n.
  \end{equation}
  Fix $n \ge 1$ and distinct $\iii,\jjj \in \II^n$. Since $\fii_\iii \ne \fii_\jjj$, we have $\widehat\fii_\iii \ne \widehat\fii_\jjj$, and so the matrix
  \begin{equation} \label{eq:algebraic-esc-H}
    H = q^n(\widehat\fii_\iii-\widehat\fii_\jjj)
  \end{equation}
  is nonzero and has entries in $\mathcal O_K$. Choose a nonzero entry $h$ of $H$. The last row of $H$ is zero, so $q^{-n}h$ is one of the coefficients of the affine map $\fii_\iii-\fii_\jjj$. It follows from \cref{eq:algebraic-esc-H} and \cref{eq:algebraic-esc-growth} that for every $\sigma \in \Sigma$,
  \begin{equation} \label{eq:algebraic-esc-conjugates}
    |\sigma(h)| \le \|\sigma(H)\| \le \|\sigma(q^n\widehat\fii_\iii)\| + \|\sigma(q^n\widehat\fii_\jjj)\| \le 2M^n.
  \end{equation}
  Since $h \in \mathcal O_K \setminus \{0\}$ and $\Sigma$ is the full set of embeddings of $K$ into $\C$, the algebraic norm
  \begin{equation} \label{eq:algebraic-esc-norm}
    N_{K/\Q}(h) = \prod_{\sigma \in \Sigma} \sigma(h)
  \end{equation}
  is a nonzero integer. Since the identity embedding belongs to $\Sigma$, it follows from \cref{eq:algebraic-esc-conjugates} and \cref{eq:algebraic-esc-norm} that
  \begin{equation} \label{eq:algebraic-esc-product-bound}
    1 \le |N_{K/\Q}(h)| = \prod_{\sigma \in \Sigma} |\sigma(h)| \le |h|(2M^n)^{|\Sigma|-1}.
  \end{equation}
  Therefore, by \cref{eq:algebraic-esc-product-bound},
  \begin{equation} \label{eq:algebraic-esc-h-lower}
    |h| \ge (2M)^{-(|\Sigma|-1)n}.
  \end{equation}
  For an affine map $\psi(x) = Bx+b$, let $\|\psi\|_{\max}$ be the maximum of the absolute values of the entries of $B$ and $b$. Since $q^{-n}h$ is one of the coefficients of the affine map $\fii_\iii-\fii_\jjj$, it follows from \cref{eq:algebraic-esc-h-lower} that
  \begin{equation} \label{eq:algebraic-esc-max-lower}
    \|\fii_\iii-\fii_\jjj\|_{\max} \ge q^{-n}|h| \ge (q(2M)^{|\Sigma|-1})^{-n}.
  \end{equation}
  Since all norms on the vector space of affine maps on $\R^d$ are equivalent, there is $C \ge 1$ such that $\|\psi\| \ge C^{-1}\|\psi\|_{\max}$ for every affine map $\psi$. Then \cref{eq:algebraic-esc-max-lower} implies
  \begin{equation*}
    \|\fii_\iii-\fii_\jjj\| \ge C^{-1}(q(2M)^{|\Sigma|-1})^{-n} \ge (C^{-1}q^{-1}(2M)^{-(|\Sigma|-1)})^n
  \end{equation*}
  for all $n \ge 1$ and all distinct $\iii,\jjj \in \II^n$. Thus $\Phi$ satisfies the exponential separation condition. This finishes the proof.
\end{proof}

\subsection{Affinity dimension} \label{sec:dimensions}
The singular values of $A \in \GL(d,\R)$ are the square roots of the eigenvalues of the positive definite matrix $A^\top A$. The \emph{singular value function}, introduced by Falconer \cite{Falconer1988} (with related ideas in \cite{DouadyOesterle1980,KaplanYorke1979}) is defined for $0 \leq s \leq d$ as
\begin{equation} \label{eq:svf-def}
  \varphi^s(A) = \alpha_1(A)\cdots \alpha_k(A) \alpha_{k+1}(A)^{s-k},
\end{equation}
where $\alpha_i(A)$ denotes the $i$-th singular value of $A$ in non-increasing order and $k$ is the integer part of $s$. For completeness, let $\fii^s(A) = |\det A|^{s/d}$ for $s>d$. Note that $\alpha_d(A) = \|A^{-1}\|^{-1}$, $\alpha_1(A) = \|A\|$, and $\alpha_d(A)^{s} \le \fii^s(A) \le \alpha_1(A)^s$ for all $s \ge 0$. In particular, $\fii^s(A) = \|A\|^s$ whenever $0 \leq s \leq 1$. By the submultiplicativity of the operator norm, the singular value function is submultiplicative, that is, $\fii^s(AB) \leq \fii^s(A) \fii^s(B)$ for all $A,B \in \GL(d,\R)$; see also \cite[Lemma 2.1]{Falconer1988}. For each $\A \in \GL(d,\R)^N$ and $s \ge 0$ we define the \emph{pressure} by setting
\begin{equation*}
  P(\A,s) = \lim_{n \to \infty} \frac{1}{n}\log\sum_{\iii \in \II^n} \fii^s(A_\iii).
\end{equation*}
As the singular value function $\fii^s$ is sub-multiplicative, the sequence $(\log\sum_{\iii \in \II^n} \fii^s(A_\iii))_{n \ge 1}$ is sub-additive and hence, the limit above exists by Fekete's lemma. It is also easy to see that the pressure $P(\A,s)$ is continuous and, if $\|A_i\|<1$ for all $i$, strictly decreasing in $s$ with $P(\A,0) \ge 0$ and $\lim_{s \to \infty}P(\A,s) = -\infty$. The \emph{affinity dimension} of $\A$, denoted $\dimaff(\A)$, is the unique $s \ge 0$ for which $P(\A,s)=0$. Falconer \cite[Theorem 5.4]{Falconer1988} proved that
\begin{equation*}
  \udimm(X) \le \dimaff(\A)
\end{equation*}
for any self-affine set $X \subset \R^d$.

\subsection{Dimensions of measures} \label{sec:dimensions-measures}
Let $\mu$ be a finite Borel measure on a separable metric space $M$. Then the \emph{upper and lower pointwise dimensions} of $\mu$ at $x \in M$ are defined as
\begin{align*}
  \udimloc(\mu,x) &= \limsup_{r \downarrow 0} \frac{\log \mu(B(x,r))}{\log r}, \\
  \ldimloc(\mu,x) &= \liminf_{r \downarrow 0} \frac{\log \mu(B(x,r))}{\log r},
\end{align*}
respectively. If the limit exists, the common value is denoted by $\dimloc(\mu,x)$. We use the convention that writing $\dimloc(\mu,x)$ implicitly means that the limit exists. If there exists a constant $s$ such that $\udimloc(\mu,x) = \ldimloc(\mu,x) = s$ for $\mu$-almost all $x \in M$, then we say that $\mu$ is \emph{exact-dimensional} and write $\dim(\mu) = s$. By convention, the use of $\dim(\mu)$ implies $\mu$ is exact-dimensional. The \emph{upper and lower packing dimensions} of $\mu$, denoted $\udimp(\mu)$ and $\ldimp(\mu)$, are the essential supremum and essential infimum, respectively, of the upper pointwise dimension with respect to $\mu$. Similarly, the upper and lower Hausdorff dimensions of $\mu$, denoted $\udimh(\mu)$ and $\ldimh(\mu)$, correspond to the essential supremum and essential infimum, respectively, of the lower pointwise dimension with respect to $\mu$. Equivalently, the Hausdorff dimensions can be expressed as
\begin{align*}
  \udimh(\mu) &= \inf\{\dimh(A) : A\text{ is a Borel set such that }\mu(M \setminus A) = 0\}, \\
  \ldimh(\mu) &= \inf\{\dimh(A) : A\text{ is a Borel set such that }\mu(A)>0\};
\end{align*}
see, for example, Falconer \cite[\S 10]{Falconer1997}. If these values coincide, the common value is the \emph{Hausdorff dimension} of $\mu$, denoted $\dimh(\mu)$. In particular, if $\mu$ is exact-dimensional with $\dim(\mu)=s$, then $\dimh(\mu)$ exists and equals $s$.

Let $\A = (A_1,\ldots,A_N) \in \GL(d,\R)^N$. The \emph{$s$-energy} of $\nu \in \MM_\sigma(\II^\N)$ is defined as
\begin{equation*}
  \Lambda(\A,\nu,s) = \lim_{n \to \infty} \frac{1}{n} \int_{\II^\N} \log \fii^s(A_{\iii|_n}) \dd\nu(\iii)
\end{equation*}
for all $s \ge 0$. Since the singular value function $\fii^s$ is sub-multiplicative and $\nu$ is invariant, the limit above exists by Fekete’s lemma. The \emph{Shannon entropy} of a Borel probability measure $\nu$ on a metric space $M$ with respect to a partition $\PP$ of $M$ is
\begin{equation*}
  H(\nu,\PP) = -\sum_{Q \in \PP} \nu(Q) \log \nu(Q).
\end{equation*}
The \emph{Kolmogorov-Sinai entropy} of $\nu \in \MM_\sigma(\II^\N)$ is
\begin{equation*}
  h(\nu) = \lim_{n \to \infty} \frac{1}{n} H(\nu,\PP_n),
\end{equation*}
where $\PP_n = \{[\iii] : \iii \in \II^n\}$. The limit above exists by Fekete's lemma and the invariance of $\nu$. It is easy to see that the sum $h(\nu) + \Lambda(\A,\nu,s)$ is continuous and, if $\|A_i\|<1$ for all $i$, strictly decreasing in $s$ with $h(\nu) + \Lambda(\A,\nu,0) \ge 0$ and $\lim_{s \to \infty} h(\nu) + \Lambda(\A,\nu,s) = -\infty$. The \emph{Lyapunov dimension} of $\nu$, denoted $\diml(\nu)$, is the unique $s \ge 0$ for which $h(\nu) + \Lambda(\A,\nu,s) = 0$. An application of Jensen’s inequality yields
\begin{equation*}
  h(\nu) + \Lambda(\A,\nu,s) \le P(\A,s)
\end{equation*}
for all $\nu \in \MM_\sigma(\II^\N)$ and $s \ge 0$. Consequently, $\diml(\nu) \le \dimaff(\A)$ for all $\nu \in \MM_\sigma(\II^\N)$; for details, see K\"aenm\"aki \cite{Kaenmaki2004}. Rossi \cite[Theorem 2.2]{Rossi2014} proved that
\begin{equation*}
  \udimp(\pi_*\nu) \le \min\{d, \diml(\nu)\}
\end{equation*}
for all ergodic measures $\nu \in \MM_\sigma(\II^\N)$.

\subsection{Local entropy averages} \label{sec:local-entropy-averages}

Let $\mu$ be a finite Borel measure on $\R^d$. The \emph{upper and lower entropy dimensions} of $\mu$ are defined as
\begin{align*}
  \udime(\mu) &= \limsup_{r \downarrow 0} \frac{1}{\mu(\R^d)} \int_{\R^d} \frac{\log\mu(B(y,r))}{\log r} \dd\mu(y), \\
  \ldime(\mu) &= \liminf_{r \downarrow 0} \frac{1}{\mu(\R^d)} \int_{\R^d} \frac{\log\mu(B(y,r))}{\log r} \dd\mu(y),
\end{align*}
respectively. If these values coincide, the common value is the \emph{entropy dimension} of $\mu$, denoted $\dime(\mu)$. Young \cite[Proposition 4.3]{Young1982} (see also \cite[Remark 3.1]{KaenmakiRajalaSuomala2016}) established that
\begin{equation*}
  \ldimh(\mu) \le \ldime(\mu) \le \udime(\mu) \le \udimp(\mu).
\end{equation*}
In particular, if $\mu$ is exact-dimensional, then $\dime(\mu)$ exists and equals $\dim(\mu)$. As the name suggests, the entropy dimension can be calculated via Shannon entropies of partitions. We denote the dyadic partition of $\mathbb{R}^d$ by
\begin{equation} \label{eq:dyadic-partition}
  \mathcal{D}_n = \biggl\{\prod_{i=1}^d \biggl[\frac{k_i}{2^{n}},\frac{k_i+1}{2^{n}}\biggr) : k_1,\ldots,k_d \in \mathbb{Z}\biggr\}.
\end{equation}
By Fan, Lau, and Rao \cite{FanLauRao2002}, the entropy dimension $\dime(\mu)$ of a Borel probability measure $\mu$ can be equivalently expressed through dyadic partitions as
\begin{equation} \label{eq:A2}
  \begin{split}
    \udime(\mu) &= \limsup_{n \to \infty} \frac{H(\mu,\mathcal{D}_n)}{n\log 2}, \\
    \ldime(\mu) &= \liminf_{n \to \infty} \frac{H(\mu,\mathcal{D}_n)}{n\log 2}.
  \end{split}
\end{equation}
The following three facts about Shannon entropy are standard; see, for example, \cite[Lemma~13.2.2]{BaranySimonSolomyak2023}. There exists a constant $C>0$ such that the following two estimates hold. If $S_{a,b} \colon \R \to \R$ is given by $S_{a,b}(x)=ax+b$ with $a\neq 0$ and $\mu$ is a Borel probability measure on $\R$, then
\begin{equation} \label{eq:A1}
  |H((S_{a, b})_* \mu, \mathcal{D}_n)-H(\mu, \mathcal{D}_{n+\lceil\log_2 |a|\rceil})| \le C
\end{equation}
for all $n\in\N$; here we use that the dyadic entropy $H(\cdot,\mathcal{D}_n)$ is invariant under the reflection $x\mapsto -x$ up to the additive constant $C$, since each atom of $\mathcal{D}_n$ meets at most two atoms of its reflection. Furthermore, if $M$ is a metric space, $\nu$ is a compactly supported finite Borel measure on $M$, and $f,g \colon M \to \R$ are continuous functions such that $\sup_{x\in M} |f(x)-g(x)| \leq 2^{-n}$ for some $n\in\N$, then
\begin{equation} \label{eq:A3}
  |H(f_*\nu,\mathcal{D}_n)-H(g_*\nu,\mathcal{D}_n)| \le C.
\end{equation}
Finally, a simple calculation shows that for every Borel probability measure $\nu$ on $\R$
\begin{equation}\label{eq:A3b}
  |H(\nu,\mathcal{D}_{n+k})-H(\nu,\mathcal{D}_n)|\leq k\log2
\end{equation} 
for all $n,k\in\N$. The method of local entropy averages, developed by Hochman and Shmerkin~\cite{HochmanShmerkin2012} and extended in~\cite{FalconerKempton2017, BaranyHochmanRapaport2019}, bounds the Hausdorff dimension of an image measure from below by the averaged entropies of images of conditioned measures. We state it for measures on $\II^\N$. If $\QQ \subset \II^*$ is such that $\{[\iii] : \iii \in \QQ\}$ is a partition of $\II^\N$ and $\jjj \in \II^\N$, then $\QQ(\jjj) \in \II^*$ denotes the unique element of $\QQ$ such that $\jjj \in [\QQ(\jjj)]$. Furthermore, if $\nu$ is a fully supported Borel probability measure on $\II^{\N}$ and $\iii \in \II^*$, then the measure $\nu$ conditioned on $[\iii]$ is $\nu_{[\iii]} = \nu([\iii])^{-1} \nu|_{[\iii]}$. For a detailed proof of the following result, we refer to~\cite[Theorem 13.2.1]{BaranySimonSolomyak2023}.

\begin{theorem} \label{thm:LocalEntropyAverages}
  Let $\nu$ be a Borel probability measure on $\II^{\N}$ and $f \colon \II^{\N}\to[-1,1]$ be continuous. Suppose $\{\mathcal{Q}_n^q : n,q \ge 1\}$ is a collection of subsets of $\II^*$ such that
  \begin{enumerate}
    \item\label{it:Q1} the family $\{[\iii] : \iii\in\mathcal{Q}_n^q\}$ is a partition of $\II^{\N}$ for all $n,q\ge 1$,
    \item\label{it:Q2} there exists $C>1$ such that
    \begin{equation*}
      C^{-1} 2^{-q n} \le \diam(f([\iii])) \le 2^{-q n-1}
    \end{equation*}
    for all $\iii \in \mathcal{Q}_n^q$ and $n,q\ge 1$,
    \item\label{it:Q3} for every $q \ge 1$ and $n \ge m \ge 1$, the set $\mathcal{Q}_n^q$ is a refinement of $\mathcal{Q}_m^q$, that is, for every $\iii \in \mathcal{Q}_n^q$ there exists a unique $\jjj \in \mathcal{Q}_m^q$ such that $[\iii] \subseteq [\jjj]$,
    \item\label{it:Q4} there exists $q_0 \ge 1$ such that
    \begin{equation*}
      \mathcal{Q}_n^q \cap \mathcal{Q}_m^q = \emptyset
    \end{equation*}
    for all $q \ge q_0$ and $n \ne m$.
  \end{enumerate}
  If there is $\alpha>0$ such that
  \begin{equation*}
    \limsup _{q \to \infty} \liminf _{n \to \infty} \frac{1}{q \log 2} \cdot \frac{1}{n} \sum_{i = 1}^n H(f_* \nu_{[\mathcal{Q}_i^q(\iii)]}, \mathcal{D}_{q(i+1)}) \ge \alpha
  \end{equation*}
  for $\nu$-almost all $\iii \in \II^\N$, then $\ldimh(f_*\nu) \ge \alpha$.
\end{theorem}

\subsection{Multilinear algebra} \label{sec:multilinear}
We review key properties of the exterior algebra. Let $\{ e_1,\ldots,e_d \}$ denote the standard orthonormal basis of $\R^d$. For each $k \in \{ 1,\ldots,d \}$, the $k$-th exterior power of $\R^d$ is the $\binom{d}{k}$-dimensional real vector space defined as
\begin{equation*}
  \wedge^k \R^d = \linspan\{ e_{i_1} \wedge \cdots \wedge e_{i_k} : 1 \leq i_1 < \cdots < i_k \leq d \}.
\end{equation*}
The wedge product $\wedge \colon \wedge^k \R^d \times \wedge^j \R^d \to \wedge^{k+j} \R^d$ in the exterior algebra of $\R^d$ is an associative, bilinear operator that satisfies graded anticommutativity: for $v \in \wedge^k \R^d$ and $w \in \wedge^j \R^d$ it holds that $v \wedge w = (-1)^{k j} w \wedge v$. Every $v \in \wedge^k \R^d$ can be expressed as
\begin{equation*}
  v = \sum_{1 \leq i_1 < \cdots < i_k \leq d} v_{i_1 \cdots i_k} e_{i_1} \wedge \cdots \wedge e_{i_k},
\end{equation*}
where $v_{i_1 \cdots i_k} \in \R$ are the coefficients. For vectors $v_j = (v_j^1,\ldots,v_j^d) \in \R^d$, $j \in \{1,\ldots,k\}$, the wedge product is
\begin{equation*}
  v_1 \wedge \cdots \wedge v_k = \sum_{1 \leq i_1 < \cdots < i_k \leq d} \det
  \begin{pmatrix}
    v_1^{i_1} & \cdots & v_1^{i_k} \\
    \vdots & \ddots & \vdots \\
    v_k^{i_1} & \cdots & v_k^{i_k}
  \end{pmatrix}
  e_{i_1} \wedge \cdots \wedge e_{i_k}.
\end{equation*}
An element $v \in \wedge^k \R^d$ is \emph{decomposable} if it can be written as $v = v_1 \wedge \cdots \wedge v_k$ for some $v_1,\ldots,v_k \in \R^d$. For example, $e_1 \wedge e_2 + e_3 \wedge e_4 \in \wedge^2 \R^4$ is not decomposable. The inner product on $\wedge^k \R^d$ is defined by
\begin{equation*}
  \la v, w \ra_k = \sum_{1 \leq i_1 < \cdots < i_k \leq d} v_{i_1 \cdots i_k} w_{i_1 \cdots i_k},
\end{equation*}
where $v_{i_1 \cdots i_k}$ and $w_{i_1 \cdots i_k}$ are the coefficients of $v$ and $w$, respectively. The induced norm is $|v|_k = \la v,v \ra_k^{1/2}$. For decomposable $v = v_1 \wedge \cdots \wedge v_k$, the norm $|v|_k$ equals the $k$-dimensional volume of the parallelepiped spanned by $v_1,\ldots,v_k$.

For $A \in \GL(d,\R)$, let $A^{\wedge k} \colon \wedge^k \R^d \to \wedge^k \R^d$ denote the induced linear map, defined on basis elements by
\begin{equation*}
  A^{\wedge k} (e_{i_1} \wedge \cdots \wedge e_{i_k}) = A e_{i_1} \wedge \cdots \wedge A e_{i_k},
\end{equation*}
for $1 \leq i_1 < \cdots < i_k \leq d$, and extended linearly. Equivalently, $A^{\wedge k}(v_1 \wedge \cdots \wedge v_k) = Av_1 \wedge \cdots \wedge Av_k$ for all decomposable vectors $v_1 \wedge \cdots \wedge v_k \in \wedge^k \R^d$. The map $A^{\wedge k}$ is invertible and represented by a $\binom{d}{k} \times \binom{d}{k}$ matrix whose entries are the $k \times k$ minors of $A$ corresponding to ordered index sets. Standard determinant properties imply that
\begin{equation*}
  (AB)^{\wedge k} = A^{\wedge k} B^{\wedge k}
\end{equation*}
for all $A, B \in \GL(d,\R)$. In particular, $A^{\wedge k} \in \GL(\binom{d}{k},\R)$. The operator norm of $A^{\wedge k}$ is
\begin{equation} \label{eq:knorm}
  \|A^{\wedge k}\|_k = \max\{ |A^{\wedge k} v|_k : |v|_k = 1 \} = \alpha_1(A) \cdots \alpha_k(A),
\end{equation}
where $\alpha_d(A) \le \cdots \le \alpha_1(A)$ are the singular values of $A$. In particular, the singular value function defined in \cref{eq:svf-def} can be expressed as
\begin{equation*}
  \varphi^s(A) = \|A^{\wedge k}\|_{k}^{k+1-s} \|A^{\wedge(k+1)}\|_{k+1}^{s-k}
\end{equation*}
for all $0 \leq s \le d$, where $k$ is the integer part of $s$. More generally, the singular values of $A^{\wedge k}$ are the products $\alpha_{i_1}(A) \cdots \alpha_{i_k}(A)$ for all $1 \leq i_1 < \cdots < i_k \leq d$.


\subsection{Irreducibility, proximality, and domination} \label{sec:irred}

Let $\mathcal{S}(\A) = \la A_1,\ldots,A_N \ra$ denote the subsemigroup of $\GL(d,\R)$ generated by the tuple $\A = (A_1,\ldots,A_N) \in \GL(d,\R)^N$. We say that $\A$ is \emph{reducible} if there exists a proper, non-trivial linear subspace $V \subset \R^d$ such that $AV = V$ for all $A \in \mathcal{S}(\A)$. If $\A$ is not reducible, it is \emph{irreducible}. Furthermore, $\A$ is \emph{strongly irreducible} if, for every proper, non-trivial linear subspace $V \subset \R^d$, the set $\{AV : A \in \mathcal{S}(\A)\}$ is infinite. Equivalently, $\A$ is strongly irreducible if and only if there does not exist a set $\VV$, formed by a finite union of proper, non-trivial linear subspaces of $\R^d$, such that $A_i\VV = \VV$ for all $i$. When $\VV$ consists of a single linear subspace, this condition reduces to the definition of irreducibility.

A tuple $\A = (A_1, \ldots, A_N) \in \GL(d, \R)^N$ is \emph{proximal} if there exists a sequence of finite words $\iii_1, \iii_2, \ldots \in \II^*$ and real numbers $c_1, c_2, \ldots$ such that the sequence $(c_n A_{\iii_n})_{n \geq 1}$ converges to a rank-one linear transformation. For an irreducible tuple $\A$, this is equivalent to the existence of a matrix $A \in \mathcal{S}(\A)$ with a simple dominant eigenvalue; see \cite[Lemma~4.1]{BenoistQuint2016}. A $d \times d$ matrix is \emph{pinching} if it is diagonalizable and all its eigenvalues are distinct in absolute value. The tuple $\A$ is \emph{pinching} if its generated subsemigroup $\mathcal{S}(\A)$ contains a pinching matrix. As a result, any pinching irreducible tuple is proximal, and in the case of $\GL(2, \R)$, the converse also holds; that is, a proximal irreducible tuple is pinching. Furthermore, if $A \in \mathcal{S}(\A)$ is pinching and $\A$ is \emph{twisting}, i.e., there exists $A' \in \mathcal{S}(\A)$ such that $A'V \cap W = \{0\}$ for every pair of $A$-invariant linear subspaces $V$ and $W$ with $\dim(V) + \dim(W) \leq d$, then $\A$ is strongly irreducible; see \cite[Remark~3.8]{MartinezRamos2026}.

Let $\A = (A_1, \ldots, A_N) \in \GL(d, \R)^N$ be a tuple of invertible matrices. For each $k \in \{1, \ldots, d-1\}$, define $\A^{\wedge k} = (A_1^{\wedge k}, \ldots, A_N^{\wedge k}) \in \GL(\binom{d}{k}, \R)^N$, where $A_i^{\wedge k}$ denotes the induced map on the $k$-th exterior power $\wedge^k \R^d$. We say that $\A$ is \emph{$k$-reducible}, \emph{$k$-irreducible}, \emph{strongly $k$-irreducible}, \emph{$k$-proximal}, or \emph{$k$-pinching} if the tuple $\A^{\wedge k}$ is reducible, irreducible, strongly irreducible, proximal, or pinching, respectively. Additionally, $\A$ is \emph{strongly pinching} if there exists $A \in \mathcal{S}(\A)$ such that $A^{\wedge k}$ is $1$-pinching for each $k \in \{1, \ldots, d-1\}$. A strongly pinching tuple is clearly $k$-pinching for all $k \in \{1,\ldots,d-1\}$. Observe also that if $\A \in \GL(2,\R)^N$ is $1$-proximal and irreducible, or $\A \in \GL(3,\R)^N$ is $1$-pinching, then $\A$ is strongly pinching.

The norm on $\wedge^2\R^d$ induces a metric on the real projective space $\RP^{d-1}$. For $V,W \in \RP^{d-1}$, choose unit vectors $v,w \in \R^d$ in the directions $V,W$, respectively, and define
\begin{equation} \label{eq:RP-metric}
  d(V,W) = |v \wedge w|_2 = \|\proj_V-\proj_W\| = |\sin(\sphericalangle(V,W))|.
\end{equation}
Since the map $V \mapsto \proj_V$ is a continuous injection from the compact space $\RP^{d-1}$ into the finite-dimensional space of linear maps on $\R^d$, it is a homeomorphism onto its image. Therefore $d$ is a metric on $\RP^{d-1}$ inducing the usual topology. For $V,W \in \RP^{d-1}$, we also write
\begin{equation*}
  |\la V,W\ra| = |\cos(\sphericalangle(V,W))| = \sqrt{1-d(V,W)^2}
\end{equation*}
for the absolute value of the inner product of unit vectors in the directions $V$ and $W$.

Domination is a sufficient condition for proximality. A tuple $\mathsf{A} = (A_1,\ldots,A_N) \in \GL(d,\R)^N$ is \emph{$k$-dominated} if there are constants $C>0$ and $0<\tau<1$ such that
\begin{equation*}
  \alpha_{k+1}(A_\iii) \le C\tau^{|\iii|}\alpha_k(A_\iii)
\end{equation*}
for all $\iii \in \II^*$, where $\alpha_k(A_\iii)$ is the $k$-th largest singular value of $A_\iii$. By \cref{eq:knorm}, $\A$ is $k$-dominated if and only if $\A^{\land k}$ is $1$-dominated. Bochi and Gourmelon \cite[Theorem~B]{BochiGourmelon2009} showed that $\A$ is $1$-dominated if and only if there is a nonempty proper closed set $\CC \subset \RP^{d-1}$ with finitely many connected components such that
\begin{equation*}
  A_i\CC\subset\CC^o
\end{equation*}
for all $i \in \{1,\ldots,N\}$, where $\CC^o$ is the interior of $\CC$ in $\RP^{d-1}$, and there is a $(d-1)$-dimensional linear subspace transverse to all elements of $\CC$. Such a set $\CC$ is called a \emph{strongly invariant multicone} for $\A$. For a fixed finite tuple $\A$, this is equivalent to the existence of a nonempty compact set $\CC_0 \subset \CC^o$ such that $A_i\CC \subset \CC_0$ for all $i \in \{1,\ldots,N\}$, for instance by taking $\CC_0 = \bigcup_{i = 1}^N A_i\CC$. A fixed-point argument applied to the strongly invariant multicone shows that a $k$-dominated tuple is $k$-proximal.

Domination gives a lower bound for the pressure through a strongly invariant multicone.

\begin{lemma} \label{lem:domination-pressure}
  Let $\CC \subset \RP^{d-1}$ be a nonempty proper closed set with finitely many connected components, $\CC_0 \subset \CC^o$ be a nonempty compact set, and suppose that there is a hyperplane transverse to every element of $\CC$. Then there exists a constant $0<\kappa\le 1$, depending only on $\CC$ and $\CC_0$, with the following property: if $\B = (B_1,\ldots,B_M) \in \GL(d,\R)^M$ is a finite tuple satisfying $B_i\CC \subset \CC_0$ for all $i \in \{1,\ldots,M\}$, then
  \begin{equation*}
    \|B_{\iii\jjj}\| \ge \kappa^2\|B_\iii\|\|B_\jjj\|
  \end{equation*}
  for all finite words $\iii,\jjj \in \{1,\ldots,M\}^*$, and
  \begin{equation*}
    P(\B,s) \ge \log\kappa^{2s} + \log\sum_{i \in \{1,\ldots,M\}} \|B_i\|^s
  \end{equation*}
  for all $0 \le s \le 1$. In particular, if $\A \in \GL(d,\R)^N$ is $1$-dominated, then the same conclusions hold.
\end{lemma}

\begin{proof}
  Let $\widehat{\CC} = \{v \in \R^d \setminus \{0\} : \linspan\{v\} \in \CC\}$ and $\widehat{\CC}_0 = \{v \in \R^d \setminus \{0\} : \linspan\{v\} \in \CC_0\}$, and note that $\widehat{\CC}$ and $\widehat{\CC}_0$ are invariant under multiplication by non-zero scalars. Their projectivizations are $\CC$ and $\CC_0$, and the given hyperplane is disjoint from $\widehat{\CC}$. Since $\CC_0 \subset \CC^o$, \cite[Lemma 2.2]{BochiMorris2015} gives $\kappa > 0$ such that
  \begin{equation*}
    |Bv| \ge \kappa\|B\||v|
  \end{equation*}
  for all $v \in \widehat{\CC}_0$ and every $B \in \GL(d,\R)$ satisfying $B\CC \subset \CC_0$. Since $|Bv| \le \|B\||v|$ by the definition of the operator norm, the inequality above forces $\kappa \le 1$. If $V \in \CC_0$ and $v \in V$ satisfies $|v| = 1$, then $v \in \widehat{\CC}_0$, and therefore
  \begin{equation*}
    \|B|V\| = |Bv| \ge \kappa\|B\|
  \end{equation*}
  for all $V \in \CC_0$ and every $B \in \GL(d,\R)$ satisfying $B\CC \subset \CC_0$. Since $\CC_0 \subset \CC$, we have $B_i\CC_0 \subset B_i\CC \subset \CC_0$ for all $i \in \{1,\ldots,M\}$. An induction on the word length gives
  \begin{equation*}
    B_\iii\CC \subset \CC_0 \qquad \text{and} \qquad B_\iii\CC_0 \subset \CC_0
  \end{equation*}
  for all non-empty finite words $\iii \in \{1,\ldots,M\}^*$. Fix $V \in \CC_0$ and non-empty finite words $\iii,\jjj \in \{1,\ldots,M\}^*$. Since $V$ and $B_\jjj V$ are one-dimensional subspaces and $B_\jjj V \in \CC_0$, the operator norm is multiplicative for the composition $V \to B_\jjj V \to B_\iii B_\jjj V$, and hence
  \begin{equation*}
    \|B_\iii B_\jjj\| \ge \|B_\iii B_\jjj | V\| = \|B_\iii | B_\jjj V\| \|B_\jjj | V\| \ge \kappa^2 \|B_\iii\| \|B_\jjj\|.
  \end{equation*}
  If either $\iii = \varnothing$ or $\jjj = \varnothing$, then the same inequality follows from $B_\varnothing = I$ and $\kappa \le 1$. Therefore
  \begin{equation*}
    \sum_{\iii \in \{1,\ldots,M\}^{m+n}} \|B_\iii\|^s \ge \kappa^{2s} \sum_{\iii \in \{1,\ldots,M\}^{m}} \|B_\iii\|^s \sum_{\jjj \in \{1,\ldots,M\}^{n}} \|B_\jjj\|^s
  \end{equation*}
  for all $m,n \ge 1$, so the sequence
  \begin{equation*}
    \biggl(\kappa^{2s}\sum_{\iii \in \{1,\ldots,M\}^{n}} \|B_\iii\|^s\biggr)_{n \ge 1}
  \end{equation*}
  is supermultiplicative. Since $0 \le s \le 1$, we have $\fii^s(A) = \|A\|^s$ for all $A \in \GL(d,\R)$, and therefore Fekete's lemma gives
  \begin{align*}
    P(\B,s) &= \lim_{n \to \infty} \frac{1}{n}\log\sum_{\iii \in \{1,\ldots,M\}^n} \|B_\iii\|^s = \lim_{n \to \infty} \frac{1}{n}\log\biggl(\kappa^{2s}\sum_{\iii \in \{1,\ldots,M\}^n} \|B_\iii\|^s\biggr) \\
    &= \sup_{n \ge 1} \frac{1}{n}\log\biggl(\kappa^{2s}\sum_{\iii \in \{1,\ldots,M\}^n} \|B_\iii\|^s\biggr) \ge \log\kappa^{2s} + \log\sum_{i \in \{1,\ldots,M\}} \|B_i\|^s
  \end{align*}
  as claimed. The final claim follows from the characterization of $1$-domination by a strongly invariant multicone recalled above.
\end{proof}

\subsection{Zariski topology and simultaneous escape property} \label{sec:zariski-escape}

We recall the basic properties of the Zariski topology on $\GL(d,\R)$. For a more detailed discussion, we refer the reader to the books of Benoist and Quint \cite{BenoistQuint2016} and Humphreys \cite{Humphreys1975}. We say that a set $Z \subset \GL(d,\R)$ is \emph{Zariski closed} if it is the common zero set of finitely many polynomials on $\GL(d,\R)$. A set is \emph{Zariski open} if its complement is Zariski closed. For a set $E \subset \GL(d,\R)$, its \emph{Zariski closure} is the smallest Zariski closed subset of $\GL(d,\R)$ containing $E$, equivalently the intersection of all Zariski closed subsets of $\GL(d,\R)$ containing $E$; when the ambient topology is clear, we denote it by $\overline{E}$. If $X$ and $Y$ are endowed with Zariski topologies, a map $f \colon X \to Y$ is \emph{Zariski continuous} if $f^{-1}(F)$ is Zariski closed in $X$ for every Zariski closed set $F \subset Y$. A tuple $\A = (A_1,\ldots,A_N) \in \GL(d,\R)^N$ is \emph{Zariski dense} if for every polynomial $p \colon \GL(d,\R) \to \R$ the condition $p(A_\iii) = 0$ for all $\iii \in \II^*$ implies that $p(A) = 0$ for all $A \in \GL(d,\R)$. Since $\det$ is a nowhere-vanishing regular function on $\GL(d,\R)$, allowing $p$ to be a regular function $q/\det^k$ yields the same notion, as $q/\det^k$ and $q$ have the same zero set; this agrees with the standard convention for Zariski density in algebraic groups. The Zariski topology satisfies the \emph{descending chain condition}: every descending chain of Zariski closed sets stabilizes. Equivalently, every non-empty family of Zariski closed sets contains a minimal element with respect to inclusion.

For every $B \in \GL(d,\R)$, the left translation $C \mapsto BC$, the right translation $C \mapsto CB$, and the inversion $C \mapsto C^{-1}$ are Zariski homeomorphisms of $\GL(d,\R)$; each is a bijection whose inverse is a map of the same form. For the translations this holds because the entries of $BC$ and $CB$ are linear in the entries of $C$, so $p(BC)$ and $p(CB)$ are polynomials whenever $p$ is, and the preimage of a common zero set is again a common zero set. For the inversion, Cramer's rule expresses the entries of $C^{-1}$ as polynomials in the entries of $C$ divided by $\det(C)$; if $p$ has total degree at most $m$, then $C \mapsto \det(C)^mp(C^{-1})$ is a polynomial with the same zero set on $\GL(d,\R)$ as $C \mapsto p(C^{-1})$, since $\det$ is nowhere vanishing there, so the preimage under inversion of a Zariski closed set is Zariski closed. In particular, each of these maps carries Zariski closed sets to Zariski closed sets and commutes with taking Zariski closures.

A non-empty set $X \subset \GL(d,\R)$ is \emph{Zariski irreducible} if, whenever $X \subset Z_1 \cup Z_2$ with $Z_1$ and $Z_2$ Zariski closed, then $X \subset Z_1$ or $X \subset Z_2$; by induction on $r$, a Zariski irreducible set contained in a finite union $Z_1 \cup \cdots \cup Z_r$ of Zariski closed sets is contained in one of the $Z_s$. Every singleton is Zariski irreducible, the image of a Zariski irreducible set under a Zariski continuous map is Zariski irreducible, and the product $XY = \{BC : B \in X \text{ and } C \in Y\}$ of two Zariski irreducible sets is Zariski irreducible. Using the descending chain condition, every non-empty Zariski closed set $Z$ can be written as a finite union of maximal Zariski irreducible closed subsets; these are the \emph{Zariski irreducible components} of $Z$, and every Zariski irreducible subset of $Z$, in particular every point of $Z$, lies in one of them.

These facts have two consequences for the subsemigroup $\mathcal{S}(\A) \subset \GL(d,\R)$ generated by $\A$, which we use below. First, the Zariski closure $H = \overline{\mathcal{S}(\A)}$ is a Zariski closed subgroup of $\GL(d,\R)$; in particular $I \in H$. Indeed, since $\mathcal{S}(\A)$ is closed under multiplication and the translations are Zariski homeomorphisms that commute with closures, $H$ is closed under multiplication as well. To see this explicitly, fix $B \in H$ and $C \in \mathcal{S}(\A)$. The right-translation identity $\overline{\mathcal{S}(\A)}C = \overline{\mathcal{S}(\A)C}$ gives $BC \in H$, because $\mathcal{S}(\A)C \subset \mathcal{S}(\A)$. Thus $B\mathcal{S}(\A) \subset H$, and the left-translation identity $B\overline{\mathcal{S}(\A)} = \overline{B\mathcal{S}(\A)}$ gives $BH \subset H$. Since $B \in H$ was arbitrary, $H$ is closed under multiplication. Moreover, for every $B \in H$ the sets $B^nH$ are Zariski closed and form a descending chain $H \supseteq BH \supseteq B^2H \supseteq \cdots$, which stabilizes by the descending chain condition and thereby forces $BH = H$, placing both $I$ and $B^{-1}$ in $H$.

Second, among the Zariski irreducible components of a Zariski closed subgroup $H$ there is a unique one, denoted $H^\circ$, that contains $I$: if two components contained $I$, their product would be a Zariski irreducible subset of $H$ containing both and, lying in a single component, would force the two to coincide. This component $H^\circ$ is a Zariski closed subgroup, because a product of Zariski irreducible sets and the image of a Zariski irreducible set under inversion are again Zariski irreducible, so that $H^\circ H^\circ$ and $(H^\circ)^{-1}$ are Zariski irreducible subsets of $H$ containing $I$ and hence, by this uniqueness, lie in $H^\circ$. Finally, for each $B \in H$ the coset $BH^\circ$ is a Zariski irreducible subset of $H$ containing $B$; if $C$ is a component of $H$ containing $B$, then $B^{-1}C$ is a Zariski irreducible subset of $H$ containing $I$, hence lies in $H^\circ$, giving $C \subset BH^\circ$, and since $BH^\circ$ in turn lies in a component, the maximality of $C$ forces $C = BH^\circ$. Thus the Zariski irreducible components of $H$ are exactly the cosets $BH^\circ$ with $B \in H$, and there are finitely many of them, say $B_1H^\circ,\ldots,B_rH^\circ$.

With the Zariski topology at our disposal, we return to the irreducibility of matrix tuples. Strong $k$-irreducibility for every $k \in \{1,\ldots,d-1\}$ is a standing assumption both in our main theorems and in the entropy-dimension formula of Rapaport~\cite{Rapaport2024} on which they rely, but the proofs below, most notably the subsystem construction in \cref{sec:strongly-separated-subsystems}, invoke it through a seemingly stronger escape property, which \cref{prop:simultaneous-escape} shows to be equivalent. We say that $\A = (A_1,\ldots,A_N) \in \GL(d,\R)^N$ has the \emph{simultaneous escape property} if for every $n \in \N$, all choices of $k_j \in \{1,\ldots,d-1\}$, proper and non-trivial subspaces $V_j \subset \wedge^{k_j}\R^d$, and proper subspaces $W_{j,1},\ldots,W_{j,m_j} \subset \wedge^{k_j}\R^d$ for $j \in \{1,\ldots,n\}$, there exists a word $\iii \in \II^*$ such that
\begin{equation*}
  A_\iii^{\wedge k_j}V_j \notin \{W_{j,1},\ldots,W_{j,m_j}\}
\end{equation*}
for all $j \in \{1,\ldots,n\}$. In other words, a single word $\iii$ must make each subspace $A_\iii^{\wedge k_j}V_j$ escape its prescribed finite set of forbidden subspaces $W_{j,1},\ldots,W_{j,m_j}$, and do so simultaneously for all $j$, across the various exterior powers $\wedge^{k_j}\R^d$; such a word brings finitely many prescribed subspaces into general position at once.

\begin{proposition} \label{prop:simultaneous-escape}
  Let $\A = (A_1,\ldots,A_N) \in \GL(d,\R)^N$. Then $\A$ has the simultaneous escape property if and only if $\A$ is strongly $k$-irreducible for every $k \in \{1,\ldots,d-1\}$.
\end{proposition}

\begin{proof}
  Suppose that $\A$ is not strongly $k$-irreducible for some $k \in \{1,\ldots,d-1\}$. Then there exists a proper and non-trivial subspace $U \subset \wedge^k\R^d$ whose orbit $\OO = \{A_\iii^{\wedge k}U : \iii \in \II^* \setminus \{\varnothing\}\}$ is finite, say $\OO = \{U_1,\ldots,U_m\}$. Each $U_\ell$ is the image of $U$ under an invertible linear map, so each $U_\ell$ is a proper and non-trivial subspace of $\wedge^k\R^d$. Fix a non-empty word $\kkk$ with $U_1 = A_\kkk^{\wedge k}U$. For every $\iii \in \II^*$, the concatenation $\iii\kkk$ is non-empty and
  \begin{equation*}
    A_\iii^{\wedge k}U_1 = A_\iii^{\wedge k}A_\kkk^{\wedge k}U = A_{\iii\kkk}^{\wedge k}U \in \OO.
  \end{equation*}
  Hence no word $\iii \in \II^*$ satisfies $A_\iii^{\wedge k}U_1 \notin \{U_1,\ldots,U_m\}$, and the simultaneous escape property fails with $n = 1$, $k_1 = k$, $V_1 = U_1$, and $W_{1,\ell} = U_\ell$ for $\ell \in \{1,\ldots,m\}$.

  Let us then show the converse. Suppose to the contrary that $\A$ is strongly $k$-irreducible for every $k \in \{1,\ldots,d-1\}$, but there are $n \in \N$ and, for each $j \in \{1,\ldots,n\}$, a choice of $k_j \in \{1,\ldots,d-1\}$, a proper and non-trivial subspace $V_j \subset \wedge^{k_j}\R^d$, and proper subspaces $W_{j,1},\ldots,W_{j,m_j} \subset \wedge^{k_j}\R^d$ such that for every $\iii \in \II^*$ there exist $j$ and $\ell$ with $A_\iii^{\wedge k_j}V_j = W_{j,\ell}$. Writing $Z(k,V,W) = \{B \in \GL(d,\R) : B^{\wedge k}V = W\}$ for the coincidence set, this means that
  \begin{equation*}
    A_\iii \in Z = \bigcup_{j=1}^n \bigcup_{\ell=1}^{m_j} Z(k_j,V_j,W_{j,\ell})
  \end{equation*}
  for every $\iii \in \II^*$. Each coincidence set $Z(k_j,V_j,W_{j,\ell})$ is Zariski closed. Since $B^{\wedge k_j}$ is invertible, $B^{\wedge k_j}V_j$ has the same dimension as $V_j$, so $Z(k_j,V_j,W_{j,\ell})$ is empty, and hence Zariski closed, unless $\dim(V_j) = \dim(W_{j,\ell})$. Suppose therefore that $\dim(V_j) = \dim(W_{j,\ell}) = t$ and write $E_j = \wedge^{k_j}\R^d$. Choose nonzero vectors $v_j,w_{j,\ell} \in \wedge^t E_j$ spanning the lines $\wedge^t V_j$ and $\wedge^t W_{j,\ell}$, respectively, and fix a basis of $\wedge^t E_j$, writing $u_p$ for the coordinates of $u \in \wedge^t E_j$ in this basis. A $t$-dimensional subspace $V' \subseteq E_j$ is recovered from the line $\wedge^t V'$ as $\{x \in E_j : x \wedge v' = 0\}$ for any nonzero $v' \in \wedge^t V'$, so $B^{\wedge k_j}V_j = W_{j,\ell}$ if and only if the one-dimensional subspaces $\wedge^t(B^{\wedge k_j}V_j)$ and $\wedge^t W_{j,\ell}$ coincide, that is, if and only if the nonzero vectors $\wedge^t(B^{\wedge k_j})v_j$ and $w_{j,\ell}$ are proportional, which amounts to the identities
  \begin{equation*}
    (\wedge^t(B^{\wedge k_j})v_j)_p(w_{j,\ell})_q - (\wedge^t(B^{\wedge k_j})v_j)_q(w_{j,\ell})_p = 0
  \end{equation*}
  for all coordinate indices $p$ and $q$. Since $B \mapsto \wedge^t(B^{\wedge k_j})$ depends polynomially on the entries of $B$, each of these expressions is a polynomial in the entries of $B$, so $Z(k_j,V_j,W_{j,\ell})$ is Zariski closed. As a finite union of Zariski closed sets, $Z$ is Zariski closed. In particular, $\mathcal{S}(\A) \subset Z$, and hence $H = \overline{\mathcal{S}(\A)} \subset Z$. By the two consequences recorded above, $H$ is a Zariski closed subgroup of $\GL(d,\R)$, its Zariski irreducible component $H^\circ$ containing $I$ is a subgroup, and $H$ is the union of finitely many cosets $B_1H^\circ,\ldots,B_rH^\circ$ with $B_1,\ldots,B_r \in H$.

  Since $H^\circ$ is Zariski irreducible and $H^\circ \subset H \subset Z$, where $Z$ is the finite union of the Zariski closed sets $Z(k_j,V_j,W_{j,\ell})$, the component $H^\circ$ is contained in one of them, so there are indices $j_0$ and $\ell_0$ with $H^\circ \subset Z(k_{j_0},V_{j_0},W_{j_0,\ell_0})$. Since $I \in H^\circ$ and $I^{\wedge k_{j_0}}$ is the identity, evaluating the coincidence condition at $I$ gives $V_{j_0} = W_{j_0,\ell_0}$, and therefore
  \begin{equation*}
    h^{\wedge k_{j_0}}V_{j_0} = V_{j_0}
  \end{equation*}
  for all $h \in H^\circ$. If $\iii \in \II^* \setminus \{\varnothing\}$, then $A_\iii \in \mathcal{S}(\A) \subset H$, so $A_\iii = B_q h$ for some $q \in \{1,\ldots,r\}$ and $h \in H^\circ$, and consequently
  \begin{equation*}
    A_\iii^{\wedge k_{j_0}}V_{j_0} = B_q^{\wedge k_{j_0}}h^{\wedge k_{j_0}}V_{j_0} = B_q^{\wedge k_{j_0}}V_{j_0}.
  \end{equation*}
  Hence the orbit $\{A_\iii^{\wedge k_{j_0}}V_{j_0} : \iii \in \II^* \setminus \{\varnothing\}\}$ is contained in the set $\{B_q^{\wedge k_{j_0}}V_{j_0} : q \in \{1,\ldots,r\}\}$ of at most $r$ subspaces. Since $V_{j_0}$ is a proper and non-trivial subspace of $\wedge^{k_{j_0}}\R^d$, this contradicts the strong $k_{j_0}$-irreducibility of $\A$ and completes the proof.
\end{proof}

\section{Projections of self-affine measures} \label{sec:projections-selfaffine-measures}

In this section, we prove \cref{thm:projection-bernoulli}, the all-directions projection theorem for self-affine measures. In \cref{sec:furstenberg-measure}, we define the Furstenberg measure $\mu_F$ and combine results of Rapaport \cite{Rapaport2024} and Feng \cite{Feng2023} to show that the projection formula holds for $\mu_F$-almost every direction. The remaining step is to pass from $\mu_F$-almost every direction to every direction. For a fixed $V$, one still has the entropy dimension identity $\dime((\proj_V)_*\mu) = \min\{1, \diml(\nu)\}$, but this does not by itself yield the required Hausdorff dimension lower bound. We prove that lower bound by verifying the local entropy averages hypothesis of \cref{thm:LocalEntropyAverages} for each fixed direction.

\begin{proposition} \label{prop:fixed-direction-lea}
  Let $X \subset \R^d$ be a self-affine set satisfying the exponential separation condition, $\nu$ be a fully supported Bernoulli measure, and $\mu$ be the associated self-affine measure such that the associated tuple of matrices is $k$-proximal and strongly $k$-irreducible for all $k \in \{1,\ldots,d-1\}$. Then
  \begin{equation*}
    \ldimh((\proj_V)_*\mu) \ge \min \{ 1, \diml(\nu) \}
  \end{equation*}
  for all $V \in \RP^{d-1}$.
\end{proposition}

In the proof of \cref{prop:fixed-direction-lea}, which can be found in \cref{sec:proof-all-directions}, we follow and adapt the strategy of \cite[Theorem~7.1]{BaranyHochmanRapaport2019} and \cite[Theorem~3.1]{FalconerKempton2017}. We note a slight gap in the proof of B\'ar\'any, Hochman, and Rapaport: in \cite[Lemma~7.3]{BaranyHochmanRapaport2019}, they claim that the corresponding time-$q$ map $T_q$ is ergodic by appealing to Falconer and Kempton \cite[Lemma~5.3]{FalconerKempton2017}. The proof of \cite[Lemma~5.3]{FalconerKempton2017} relies on a result of Quas and Soo \cite[Proposition~5]{QuasSoo2012}, which is applicable only when the roof function of the induced suspension flow is H\"older continuous over the symbolic space; this is the case, in particular, when the matrices are dominated. However, domination is not assumed in the setting of \cite[Theorem~7.1]{BaranyHochmanRapaport2019}. We close this gap by not assuming ergodicity and working instead with the ergodic decomposition.

The previous proposition supplies the lower bound in every direction. The reverse inequality is the standard upper bound for Lipschitz images, so with \cref{prop:fixed-direction-lea} in hand we can now complete the proof of \cref{thm:projection-bernoulli}.

\begin{theorem} \label{thm:proj-dim}
  Let $\Phi$ be an affine iterated function system satisfying the exponential separation condition such that the associated tuple $\A \in \GL(d,\R)^N$ is $k$-proximal and strongly $k$-irreducible for all $k \in \{1,\ldots,d-1\}$, and let $X \subset \R^d$ be the self-affine set associated with $\Phi$. If $\nu$ is a fully supported Bernoulli measure and $\mu = \pi_*\nu$ is the associated self-affine measure, then
  \begin{equation*}
    \dim((\proj_V)_*\mu) = \min\{1,\diml(\nu)\}
  \end{equation*}
  for all $V \in \RP^{d-1}$.
\end{theorem}

\begin{proof}
  Fix $V \in \RP^{d-1}$. By \cref{prop:fixed-direction-lea},
  \begin{equation*}
    \ldimh((\proj_V)_*\mu) \ge \min \{ 1, \diml(\nu) \}.
  \end{equation*}
  Since $(\proj_V)_*\mu$ is a Borel measure on the line $V$ and $\proj_V$ is Lipschitz, we have $\udimp((\proj_V)_*\mu) \le \min\{1, \udimp(\mu)\}$. Since $\mu = \pi_*\nu$ and $\nu$ is Bernoulli, hence ergodic, \cite[Theorem 2.2]{Rossi2014} gives
  \begin{equation*}
    \udimp(\mu) \le \min \{ d, \diml(\nu) \},
  \end{equation*}
  and therefore
  \begin{equation*}
    \udimp((\proj_V)_*\mu) \le \min \{ 1, \diml(\nu) \}.
  \end{equation*}
  Hence the lower Hausdorff and upper packing dimensions of $(\proj_V)_*\mu$ coincide and equal $\min \{1,\diml(\nu)\}$, as claimed.
\end{proof}

The remainder of the section is devoted to the proof of \cref{prop:fixed-direction-lea}. We first introduce the Furstenberg measure and prove the projection theorem in Furstenberg-typical directions, then construct a suspension flow yielding the entropy input, and finally verify the local entropy averages hypothesis in each fixed direction.

\subsection{The Furstenberg measure}\label{sec:furstenberg-measure}

We introduce the stationary measure on projective space associated to a proximal strongly irreducible tuple, then record the projection theorem that holds in its typical directions. For a proximal tuple $\A = (A_1,\ldots,A_N) \in \GL(d,\R)^N$, the set of \emph{Furstenberg directions} is defined as
\begin{equation*}
  X_F = \{A\R^d \in \RP^{d-1} : A \in \overline{\{cA_{\iii}^\top : c\in\R\text{ and }\iii\in\II^*\}}\text{ and $A$ has rank one}\}.
\end{equation*}
In the planar papers \cite{BaranyKaenmakiYu2021-preprint,AnttilaBaranyKaenmaki2024,BaranyKaenmakiRossi2021}, the relevant Furstenberg directions are the orthogonal complements of the elements of our $X_F$.

If $\A$ is $1$-proximal and strongly $1$-irreducible and $\nu$ is the Bernoulli measure obtained from a probability vector $(p_1,\ldots,p_N)$ with $p_i = \nu([i]) > 0$ for all $i$, then a Borel probability measure $\mu_F$ on $\RP^{d-1}$ satisfying
\begin{equation*}
  \mu_F = \sum_{i = 1}^N p_i (A_i^\top)_*\mu_F
\end{equation*}
is called the \emph{Furstenberg measure} associated to $\nu$. We first prove existence and uniqueness of this measure.

\begin{proposition} \label{prop:FurstenbergMeasure}
  Let $\A = (A_1, \ldots, A_N) \in \GL(d,\R)^N$ be $1$-proximal and strongly $1$-irreducible and $\nu \in \MM_\sigma(\II^\N)$ be a fully supported Bernoulli measure. Then there exists a unique Borel probability measure $\mu_F$ on $\RP^{d-1}$ such that
  \begin{equation} \label{eq:Furstenberg-invariance}
    \mu_F = \sum_{i = 1}^N p_i (A_i^\top)_*\mu_F,
  \end{equation}
  and this measure satisfies $\ldimh(\mu_F) > 0$, has projective hyperplane non-concentration in the sense that $\mu_F(U^\bot) = 0$ for every $U \in \RP^{d-1}$, and has support $\spt(\mu_F) = X_F$.
\end{proposition}

\begin{proof}
  Set
  \begin{equation*}
    \eta = \sum_{i \in \II} p_i\delta_{A_i^\top}.
  \end{equation*}
  If $(c_nA_{\iii_n})_n$ converges to a rank-one map, then $(c_nA_{\iii_n}^\top)_n$ converges to a rank-one map as well. Hence $\A^\top = (A_1^\top,\ldots,A_N^\top)$ is $1$-proximal. Since each $A_{\iii_n}^\top$ lies in the semigroup generated by $\A^\top$ and the normalizations $\|A_{\iii_n}^\top\|^{-1}A_{\iii_n}^\top$ have a rank-one limit point, the measure $\eta$ is contracting in the sense of \cite[Definition III.1.3]{BougerolLacroix1985}. By~\cite[Lemma III.3.3]{BougerolLacroix1985}, $\A^\top$ is also strongly $1$-irreducible. Strong irreducibility and contraction are the standing hypotheses of the results of \cite{BougerolLacroix1985} applied below. Applying~\cite[Theorem III.3.1]{BougerolLacroix1985} to $\eta$ now gives the existence and uniqueness of the Borel probability measure $\mu_F$ satisfying \cref{eq:Furstenberg-invariance}. Since $\eta$ has finite support, it has an exponential moment, so the hypotheses of~\cite[Proposition VI.4.1 and Corollary VI.4.2]{BougerolLacroix1985} are satisfied. Hence there exist constants $\gamma,C > 0$ such that
  \begin{equation*}
    \mu_F(\{U \in \RP^{d-1} : d(U,W) < r\}) \le Cr^\gamma
  \end{equation*}
  for all $W \in \RP^{d-1}$ and all $r > 0$. Thus $\mu_F$ is a $\gamma$-Frostman measure and, by the definition of $\ldimh(\mu_F)$, this gives $\ldimh(\mu_F) \ge \gamma > 0$.

  We next prove projective hyperplane non-concentration. Fix $U \in \RP^{d-1}$, choose a unit vector $u \in U$ and a unit vector $w \in \R^d$, and define the rank-one operator $B \colon \R^d \to \R^d$ by
  \begin{equation*}
    Bx = \la u,x \ra w.
  \end{equation*}
  Then $B|L = 0$ if and only if $L \subseteq U^\bot$, that is, if and only if $L \in U^\bot$. Since $\mu_F$ is stationary by the invariance relation \cref{eq:Furstenberg-invariance} and $\A^\top$ is strongly $1$-irreducible, \cite[Proposition III.2.3]{BougerolLacroix1985} gives
  \begin{equation*}
    \mu_F(U^\bot) = \mu_F(\{L \in \RP^{d-1} : B|L = 0\}) = 0.
  \end{equation*}
  It remains to identify $\spt(\mu_F)$. Let
  \begin{equation*}
    \mathcal C = \overline{\{cA_{\iii}^\top : c \in \R\text{ and }\iii \in \II^*\}}.
  \end{equation*}
  Since $\mathcal C$ is closed, the set of matrices of rank at most one is closed, and the unit sphere in the finite-dimensional space of linear maps on $\R^d$ is compact, the set
  \begin{equation*}
    \mathcal{K} = \{B \in \mathcal C : \|B\| = 1\text{ and $B$ has rank at most one}\}
  \end{equation*}
  is compact. As every element of $\mathcal{K}$ has norm $1$, every element of $\mathcal{K}$ has rank one. For each $B \in \mathcal{K}$, the line $B\R^d$ belongs to $\RP^{d-1}$ and
  \begin{equation*}
    \proj_{B\R^d} = \|BB^\top\|^{-1}BB^\top.
  \end{equation*}
  Therefore the map $B \mapsto B\R^d$ is continuous on $\mathcal{K}$, and its image is $X_F$. Hence $X_F$ is compact. In particular, $X_F$ is closed. If $W \in X_F$ and $i \in \II$, then there exist words $\jjj_n \in \II^*$ and scalars $c_n \in \R$ such that $c_nA_{\jjj_n}^\top$ converges to a rank-one matrix $B$ with range $W$. Since $A_i^\top B$ has rank one and range $A_i^\top W$, and since
  \begin{equation*}
    A_i^\top(c_nA_{\jjj_n}^\top) = c_n(A_{\jjj_n}A_i)^\top,
  \end{equation*}
  we obtain $A_i^\top W \in X_F$. Thus $X_F$ is forward invariant under each map $A_i^\top$. Define $T \colon \mathcal P(X_F) \to \mathcal P(X_F)$ by
  \begin{equation*}
    T\lambda = \sum_{i = 1}^N \nu([i]) (A_i^\top)_*\lambda.
  \end{equation*}
  Since $\A^\top$ is $1$-proximal, some sequence $c_nA_{\iii_n}^\top$ converges to a rank-one matrix lying in $\mathcal C$, whose normalization to unit norm belongs to $\mathcal K$; hence $X_F \ne \emptyset$. Since $X_F$ is compact, the space $\mathcal P(X_F)$ of Borel probability measures on $X_F$ is weak$^*$ compact and convex, and $T$ is continuous. Schauder's fixed point theorem therefore yields a Borel probability measure $\lambda$ on $X_F$ satisfying
  \begin{equation*}
    \lambda = \sum_{i = 1}^N \nu([i]) (A_i^\top)_*\lambda.
  \end{equation*}
  Since $\mu_F$ is the unique Borel probability measure satisfying this invariance relation, we have $\lambda = \mu_F$, and hence $\spt(\mu_F) \subset X_F$. Conversely, fix $W \in X_F$. Then there exist words $\jjj_n \in \II^*$ and scalars $c_n \in \R$ such that $c_nA_{\jjj_n}^\top$ converges to a rank-one matrix $B$ with range $W$. Since $B$ has rank one, $\ker(B)$ is a hyperplane. Let $U = \ker(B)^\bot$. Then $\{V \in \RP^{d-1} : B|V = 0\} = U^\bot$, so the projective hyperplane non-concentration just proved gives
  \begin{equation*}
    \mu_F(\{V \in \RP^{d-1} : B|V = 0\}) = 0.
  \end{equation*}
  Since $\mu_F(\RP^{d-1} \setminus \spt(\mu_F)) = 0$, we may choose $V \in \spt(\mu_F)$ such that $B|V \ne 0$. As each map $A_i^\top$ is a homeomorphism of $\RP^{d-1}$ and all weights $\nu([i])$ are positive, the invariance relation \cref{eq:Furstenberg-invariance} implies
  \begin{equation*}
    \spt(\mu_F) = \bigcup_{i \in \II} A_i^\top \spt(\mu_F).
  \end{equation*}
  Therefore $A_{\jjj_n}^\top V \in \spt(\mu_F)$ for every $n$, and since $c_nA_{\jjj_n}^\top|V$ converges to $B|V \ne 0$, the directions $A_{\jjj_n}^\top V$ converge to the range of $B|V$, namely $W$. Since $\spt(\mu_F)$ is closed, we obtain $W \in \spt(\mu_F)$. This proves the claim on the support, completing the proof of the proposition.
\end{proof}

The next proposition records the almost-sure asymptotic behaviour of the projective action of the cocycle $(A_{\iii|_n}^\top)_{n \ge 1}$ on $\RP^{d-1}$ needed later.

\begin{proposition} \label{prop:FurstenbergMeasure2}
  Let $\A = (A_1, \ldots, A_N) \in \GL(d,\R)^N$ be $1$-proximal and strongly $1$-irreducible and $\nu \in \MM_\sigma(\II^\N)$ be a fully supported Bernoulli measure. Then there exist a set $\Omega \subset \II^\N$ of full $\nu$-measure and a measurable map $Z \colon \Omega \to \RP^{d-1}$ such that
  \begin{equation} \label{eq:Furstenberg4}
    \lim_{n\to\infty} \frac{\|A_{\iii|_n}^{\top}|V\|}{\|A_{\iii|_n}^{\top}|W\|} = \frac{|\la Z(\iii),V\ra|}{|\la Z(\iii),W\ra|}
  \end{equation}
  for all $\iii \in \Omega$ and all $V,W \in \RP^{d-1}$ with $W \not\perp Z(\iii)$. Furthermore,
  \begin{equation} \label{eq:Furstenberg3}
    \lim_{n \to \infty} d(A_{\iii|_n}^{\top} V, A_{\iii|_n}^{\top} W) = 0
  \end{equation}
  for all $\iii \in \Omega$ and all $V,W \in \RP^{d-1}$ with $V \not\perp Z(\iii)$ and $W \not\perp Z(\iii)$. Finally, for every $V \in \RP^{d-1}$,
  \begin{equation*}
    \nu(\{\iii : Z(\iii) \perp V\}) = 0.
  \end{equation*}
\end{proposition}

\begin{proof}
  Since $\nu$ is a fully supported Bernoulli measure, the matrices $A_{i_n}^\top$, $n \ge 1$, are independent and identically distributed, and $A_{\iii|_n}^\top = A_{i_n}^\top\cdots A_{i_1}^\top$ is the corresponding product. Both $\A$ and $\A^\top$ are $1$-proximal and strongly $1$-irreducible, the latter by the proof of \cref{prop:FurstenbergMeasure}. By \cite[Theorem III.3.1 and Proposition III.3.2(a)]{BougerolLacroix1985}, there exist a set $\Omega \subset \II^\N$ of full $\nu$-measure and a measurable map $Z \colon \Omega \to \RP^{d-1}$ such that $Z(\iii)=\lim_{n\to\infty}\linspan\{z_1(\iii|_n)\}$, where $\{z_i(\iii|_n)\}_{i=1}^d$ is the right singular orthonormal basis of $\R^d$ generated by $A_{\iii|_n}^\top$. That is, $A_{\iii|_n}^\top z_i(\iii|_n)=\alpha_i(A_{\iii|_n}^\top)w_i(\iii|_n)$, where $\{w_i(\iii|_n)\}_{i=1}^d$ is the left singular orthonormal basis of $\R^d$ generated by $A_{\iii|_n}^\top$. Moreover, for every $V \in \RP^{d-1}$,
  \begin{equation*}
    \nu(\{\iii : Z(\iii) \perp V\}) = 0.
  \end{equation*}
  Recall that $\alpha_1(A_{\iii|_n}^\top) = \|A_{\iii|_n}^\top\|$. Intersecting $\Omega$ with a further set of full $\nu$-measure if necessary, we may assume that part~(b) of \cite[Proposition III.3.2]{BougerolLacroix1985} holds for every $\iii \in \Omega$, that is,
  \begin{equation*}
    \lim_{n\to\infty}\frac{\alpha_2(A_{\iii|_n}^\top)}{\alpha_1(A_{\iii|_n}^\top)} = 0.
  \end{equation*}
  Combining the convergence $Z(\iii)=\lim_{n\to\infty}\linspan\{z_1(\iii|_n)\}$ with this singular-value gap, the polar decomposition of $A_{\iii|_n}^\top$ gives
  \begin{equation*}
    \lim_{n\to\infty}\frac{\|A_{\iii|_n}^\top|V\|}{\|A_{\iii|_n}^\top\|} = |\la Z(\iii),V\ra|
  \end{equation*}
  for every $\iii \in \Omega$ and every $V \in \RP^{d-1}$, which implies \cref{eq:Furstenberg4}. To see \cref{eq:Furstenberg3}, since $d(V,W) \le 1$, \cite[Lemma III.4.2]{BougerolLacroix1985}, applied to unit representatives of $V$ and $W$, gives
  \begin{equation*}
    d(A_{\iii|_n}^{\top} V, A_{\iii|_n}^{\top} W) \le \frac{\alpha_1(A_{\iii|_n}^\top)\alpha_2(A_{\iii|_n}^\top)}{\|A_{\iii|_n}^\top|V\|\|A_{\iii|_n}^\top|W\|} = \frac{\alpha_2(A_{\iii|_n}^\top)}{\alpha_1(A_{\iii|_n}^\top)}\cdot\frac{\|A_{\iii|_n}^\top\|}{\|A_{\iii|_n}^\top|V\|}\cdot\frac{\|A_{\iii|_n}^\top\|}{\|A_{\iii|_n}^\top|W\|},
  \end{equation*}
  which converges to zero as $n\to\infty$ for every $\iii\in\Omega$ and all $V,W \in \RP^{d-1}$ with $V \not\perp Z(\iii)$ and $W \not\perp Z(\iii)$, since $\alpha_2(A_{\iii|_n}^\top)/\alpha_1(A_{\iii|_n}^\top) \to 0$ while $\|A_{\iii|_n}^\top\|/\|A_{\iii|_n}^\top|V\| \to |\la Z(\iii),V\ra|^{-1}$ and $\|A_{\iii|_n}^\top\|/\|A_{\iii|_n}^\top|W\| \to |\la Z(\iii),W\ra|^{-1}$.
\end{proof}

We conclude the subsection by recording the Furstenberg-typical projection theorem that serves as the starting point for the all-directions argument. It combines Rapaport's entropy-dimension formula for projected measures with Feng's exact-dimensionality theorem. 

\begin{theorem} \label{thm:rapaport}
  Let $\Phi$ be an affine iterated function system satisfying the exponential separation condition such that the associated tuple $\A \in \GL(d,\R)^N$ is $k$-proximal and strongly $k$-irreducible for all $k \in \{1,\ldots,d-1\}$, and let $X \subset \R^d$ be the self-affine set associated with $\Phi$. If $\nu$ is a fully supported Bernoulli measure and $\mu = \pi_*\nu$ is the associated self-affine measure, then
  \begin{equation*}
    \dim((\proj_V)_*\mu) = \min\{1,\diml(\nu)\}
  \end{equation*}
  for $\mu_F$-almost all $V \in \RP^{d-1}$.
\end{theorem}

\begin{proof}
  By Rapaport \cite[Theorem 1.12]{Rapaport2024},
  \begin{equation*}
    \dime((\proj_V)_*\mu) = \min\{1,\diml(\nu)\}
  \end{equation*}
  for all $V \in \RP^{d-1}$. By \cref{sec:dimensions-measures}, the Hausdorff dimension of an exact-dimensional measure equals its entropy dimension, so it suffices to prove that $(\proj_V)_*\mu$ is exact-dimensional for $\mu_F$-almost every $V \in \RP^{d-1}$.

  This exact-dimensionality is provided by Feng \cite[Theorem 1.6(ii)]{Feng2023}, whose hypotheses require an ergodic quasi-Bernoulli measure on the two-sided shift, an average contracting affine system, and at least two distinct Lyapunov exponents. Let $m$ be the two-sided Bernoulli measure with the same weights as $\nu$. The theorem holds in one-sided form through the identity $(\pi^+)_*(m^+) = \pi_*m$ between the one- and two-sided coding push-forwards; see \cite[\S 1]{Feng2023}. Equivalently, taking $m^+ = \nu$, our fully supported Bernoulli measure meets these hypotheses, being ergodic and quasi-Bernoulli with contracting linear parts and, since $\A$ is $1$-proximal and strongly $1$-irreducible, having a simple leading Lyapunov exponent by \cite[Theorem IV.1.2]{BougerolLacroix1985}, so that at least two distinct exponents occur. In this case Feng's Furstenberg--Oseledets measure is the law of the leading Oseledets subspace $V_x^1$, defined for a two-sided sequence $x = (x_n)_{n \in \Z} \in \II^\Z$ distributed according to $m$: the subspace $V_x^1$ is cut out by the backward products $A_{x_{-n}} \cdots A_{x_{-1}}$ of the linear parts and is a hyperplane because the leading Lyapunov exponent is simple. By \cite[Theorem 1.6(ii)]{Feng2023}, the projection of $\mu$ onto the line $(V_x^1)^\perp$ is exact-dimensional for $m$-almost every such $x$.

  It remains to identify this Furstenberg--Oseledets measure on the space of hyperplanes in $\R^d$ with $\mu_F$ through the homeomorphism $V \mapsto V^\perp$ of $\RP^{d-1}$ onto that space. Since $(A_i^\top V)^\perp = A_i^{-1}V^\perp$ for every $i \in \II$ and every $V \in \RP^{d-1}$, this homeomorphism conjugates the cocycle $V \mapsto A_i^\top V$ on $\RP^{d-1}$ to the dual cocycle $W \mapsto A_i^{-1}W$ on the space of hyperplanes, and hence carries the stationary measures of the former bijectively onto those of the latter. By \cref{prop:FurstenbergMeasure}, $\mu_F$ is the unique stationary measure of $V \mapsto A_i^\top V$, so its image under $V \mapsto V^\perp$ is the unique stationary measure of the dual cocycle. Feng's Furstenberg--Oseledets measure is stationary for the dual cocycle by the Oseledets equivariance of the subspaces $V_x^1$, and therefore coincides with the image of $\mu_F$. Hence the exact-dimensionality holds for $\mu_F$-almost every $V \in \RP^{d-1}$, which completes the proof.
\end{proof}

\subsection{Suspension flow and entropy input}

We begin the proof of \cref{prop:fixed-direction-lea} with a suspension semiflow over $\II^\N \times \RP^{d-1}$ whose roof function records the contraction rate in the current projective direction. Along this flow the direction coordinate converges to the Furstenberg direction, independently of the initial direction. Rapaport's entropy identity therefore yields entropy bounds on ergodic components of the time-$q$ maps, and a stopping-time comparison transfers these bounds from the evolving direction to the prescribed direction.

Throughout the remainder of the section, let $\Phi = (\fii_1,\ldots,\fii_N)$ denote the affine iterated function system, $\A = (A_1,\ldots,A_N) \in \GL(d,\R)^N$ the associated matrix tuple, $\nu$ the Bernoulli measure associated to a probability vector $(p_1,\ldots,p_N)$ with $p_i > 0$ for all $i$, $\mu = \pi_*\nu$ the self-affine measure, and $X \subset \R^d$ the self-affine set, all as in \cref{prop:fixed-direction-lea}. For each $V \in \RP^{d-1}$, let $u_V \in V$ be the unit vector whose first non-zero coordinate is positive, and define $P_V \colon \R^d \to \R$ by $P_V(x) = \la u_V,x \ra$. Thus $P_V$ identifies the line $V$ isometrically with $\R$. For every affine map $f(x) = Ax+t$ and every $V \in \RP^{d-1}$, there exists $\delta \in \{-1,1\}$ such that
\begin{equation*}
  P_V(f(x)) = \delta\|A^\top|V\|P_{A^\top V}(x) + P_V(f(0))
\end{equation*}
for all $x \in \R^d$. Therefore, by self-affinity, for every $\iii \in \II^*$ there exists $\delta_{\iii,V} \in \{-1,1\}$ such that
\begin{equation} \label{eq:ProjSelfAffinity}
  (P_V\circ \pi)_* \nu_{[\iii]} = (S_{\delta_{\iii,V}\|A_{\iii}^{\top}|V\|, P_V(\fii_\iii(0))} \circ P_{A_{\iii}^{\top}V})_*\mu.
\end{equation}
Here $S_{a,b} \colon \R \to \R$ is the affine map given by $S_{a,b}(x) = ax+b$, and $\nu_{[\iii]} = \nu([\iii])^{-1}\nu|_{[\iii]}$ is the measure $\nu$ conditioned on the cylinder $[\iii]$. For $(\mathtt{i},V) \in \II^\N \times \RP^{d-1}$, define
\begin{equation*}
  F(\mathtt{i},V) = (\sigma\mathtt{i},A_{i_1}^{\top}V) \qquad \text{and} \qquad r(\mathtt{i},V) = -\log_2\|A_{i_1}^{\top}|V\|.
\end{equation*}
Since each $A_i$ is contractive and invertible, the roof function $r$ is bounded above and away from zero. Define
\begin{equation*}
  \widehat{\mathcal{X}} = \{(\mathtt{i},V,t)\in\II^\N\times\RP^{d-1}\times\R : 0\leq t\leq r(\mathtt{i},V)\},
\end{equation*}
and let $\sim$ be the equivalence relation on $\widehat{\mathcal{X}}$ defined by
\begin{equation*}
  (\mathtt{i},V,r(\mathtt{i},V)) \sim (F(\mathtt{i},V),0).
\end{equation*}
We define the suspension space by $\mathcal{X} = \widehat{\mathcal{X}}/_\sim$. Each equivalence class has a unique representative with $0\leq t<r(\mathtt{i},V)$, so we may identify $\mathcal{X}$ with
\begin{equation*}
  \{(\mathtt{i},V,t)\in\II^\N\times\RP^{d-1}\times\R : 0\leq t<r(\mathtt{i},V)\}.
\end{equation*}
Let $(T_u)_{u \geq 0}$ denote the associated suspension semiflow, and write $\mathbf{x}$ for a point of $\mathcal{X}$. For $u \ge 0$, the map $T_u$ is obtained by increasing the third coordinate by $u$ and using the identification whenever the roof is crossed. Let $\mu_F$ be the Furstenberg measure associated to $\nu$. Since $\nu$ is Bernoulli and $\mu_F$ satisfies
\begin{equation*}
  \mu_F = \sum_{i=1}^N p_i(A_i^\top)_*\mu_F,
\end{equation*}
the measure $\nu\times\mu_F$ is $F$-invariant. Hence
\begin{equation*}
  m = \frac{(\nu\times\mu_F\times \mathcal{L}^1)|_{\mathcal{X}}}{(\nu\times\mu_F\times \mathcal{L}^1)(\mathcal{X})}
\end{equation*}
is a $T_u$-invariant Borel probability measure on $\mathcal{X}$. For each integer $q \in \N$, we use the ergodic decomposition of $m$ with respect to the time-$q$ map $T_q$: for $m$-almost every $\mathbf{x}\in \mathcal{X}$ there exists a $T_q$-invariant, ergodic Borel probability measure $\nu_{q,\mathbf{x}}$ such that for every $f\in L^1(\mathcal{X},m)$,
\begin{equation} \label{eq:ErgodicDecomp}
  \int_{\mathcal{X}} f(\mathbf{x}) \dd m(\mathbf{x}) = \int_{\mathcal{X}} \biggl( \int_{\mathcal{X}} f(\mathbf{y}) \dd \nu_{q,\mathbf{x}}(\mathbf{y}) \biggr) \dd m(\mathbf{x}),
\end{equation}
where the map $\mathbf{x}\mapsto \int_{\mathcal{X}} f \dd \nu_{q,\mathbf{x}}$ is measurable and $T_q$-invariant. For integers $q,k\in\N$, define the function $f_{q,k} \colon \mathcal{X}\to \R$ by
\begin{equation*}
  f_{q,k}(\iii,V,t) = \min_{\ell \in \{0,\ldots,k\}} H((P_{A_{\iii|_\ell}^\top V}\circ \pi)_*\nu,\mathcal{D}_q),
\end{equation*}
where $\mathcal{D}_q$ is the dyadic partition of $\R$ defined in \cref{eq:dyadic-partition}. The map $V \mapsto u_V$ is Borel, being continuous on each set of the partition of $\RP^{d-1}$ according to the index of the first non-zero coordinate; hence, for each $q,k\in\N$, the function $f_{q,k}$ is a bounded Borel function on $\mathcal{X}$.

The first lemma is the suspension-flow form of Rapaport's all-directions entropy-dimension identity \cite[Theorem~1.12]{Rapaport2024}, recalled in the proof of \cref{thm:rapaport}. It gives the lower bound $\min \{ 1, \diml(\nu) \}$ for the normalized entropy averages on almost every ergodic component of the time-$q$ map $T_q$. This is the ingredient that feeds the entropy-dimension identity into the fixed-direction local entropy averages argument.

\begin{lemma} \label{lem:ErgDecompValue}
  Let $m$ be the $T_u$-invariant Borel probability measure on $\mathcal{X}$ defined above. For each $q \in \N$, let $\nu_{q,\mathbf{x}}$ be the ergodic decomposition measure of $m$ with respect to $T_q$. Then for $m$-almost every $\mathbf{x}\in \mathcal{X}$ and every $k\in\N$,
  \begin{equation*}
    \limsup_{q\to \infty} \frac{1}{q\log 2} \int f_{q,k}(\mathbf{y}) \dd \nu_{q,\mathbf{x}}(\mathbf{y}) \ge \min \{ 1, \diml(\nu) \}.
  \end{equation*}
\end{lemma}

\begin{proof}
  Set $s = \min \{ 1, \diml(\nu) \}$. We use Rapaport's all-directions entropy-dimension identity, in the form recalled in the proof of \cref{thm:rapaport}. Since $(P_V)_*\mu$ is an isometric image of $(\proj_V)_*\mu$, Rapaport \cite[Theorem~1.12]{Rapaport2024} implies that
  \begin{equation*}
    \dime((P_V)_*\mu) = s
  \end{equation*}
  for every $V \in \RP^{d-1}$. Therefore \cref{eq:A2} implies that for every $k\in\N$,
  \begin{equation*}
    \frac{f_{q,k}(\mathbf{x})}{q\log 2} \to s
  \end{equation*}
  for every $\mathbf{x}\in\mathcal{X}$ as $q\to\infty$. Since $P_V(X)$ has diameter at most $\diam(X)$ for every $V \in \RP^{d-1}$, the partition $\mathcal{D}_q$ has at most $2^q\diam(X) + 2$ atoms meeting the support of $(P_V)_*\mu$. Hence there exists $C>0$ such that
  \begin{equation*} 
    0 \le \frac{f_{q,k}(\mathbf{x})}{q\log 2} \le 1 + \frac{C}{q\log 2}
  \end{equation*}
  for all $q,k \in \N$ and $\mathbf{x}\in \mathcal{X}$. Hence,
  \begin{equation*}
    \frac{f_{q,k}}{q\log 2} \to s
  \end{equation*}
  in $L^1(\mathcal{X},m)$. For each $q \in \N$, the ergodic decomposition kernel $\mathbf{x} \mapsto \nu_{q,\mathbf{x}}$ is defined for $m$-almost every $\mathbf{x}$, and \cref{eq:ErgodicDecomp} then holds for every $f \in L^1(\mathcal{X},m)$. Since only countably many integers $q$ are used below, discarding the countable union of the corresponding $m$-null sets leaves one fixed $m$-null set outside which \cref{eq:ErgodicDecomp} is available for all of them. Define for every $q,k\in\N$,
  \begin{equation*}
    h_{q,k}(\mathbf{x}) = \frac{1}{q\log 2} \int f_{q,k}(\mathbf{y}) \dd \nu_{q,\mathbf{x}}(\mathbf{y}).
  \end{equation*}
  Using \cref{eq:ErgodicDecomp} and Jensen's inequality, we obtain
  \begin{align*}
    \int_{\mathcal{X}} |h_{q,k}(\mathbf{x}) - s| \dd m(\mathbf{x}) &\le \int_{\mathcal{X}} \biggl( \int_{\mathcal{X}} \biggl| \frac{f_{q,k}(\mathbf{y})}{q\log 2} - s \biggr| \dd \nu_{q,\mathbf{x}}(\mathbf{y}) \biggr) \dd m(\mathbf{x}) \\
    &= \int_{\mathcal{X}} \biggl| \frac{f_{q,k}(\mathbf{x})}{q\log 2} - s \biggr| \dd m(\mathbf{x}).
  \end{align*}
  Hence $h_{q,k}\to s$ in $L^1(\mathcal{X},m)$ as $q\to\infty$. Convergence in $L^1(\mathcal{X},m)$ implies convergence in measure, so there exists a strictly increasing sequence $(q_n)_{n\in\N}$ such that $h_{q_n,k}(\mathbf{x})\to s$ for $m$-almost every $\mathbf{x}\in \mathcal{X}$. Thus, for each $k\in\N$,
  \begin{equation*}
    \limsup_{q\to \infty} \frac{1}{q\log 2} \int f_{q,k}(\mathbf{y}) \dd \nu_{q,\mathbf{x}}(\mathbf{y}) \ge s
  \end{equation*}
  for $m$-almost every $\mathbf{x}\in \mathcal{X}$. Since $\N$ is countable, the exceptional set can be chosen independently of $k$, so the inequality holds for $m$-almost every $\mathbf{x}\in\mathcal{X}$ simultaneously for all $k\in\N$, as claimed.
\end{proof}

\subsection{Proof of the all-directions lower bound} \label{sec:proof-all-directions}

The suspension flow constructed above, the entropy estimate of \cref{lem:ErgDecompValue}, and the Furstenberg-direction results of \cref{prop:FurstenbergMeasure2} provide everything needed to apply local entropy averages and prove \cref{prop:fixed-direction-lea}. 

Before giving the details, we outline the argument. We first attach to the fixed direction $W$ a family of stopping-time cutsets and verify that they form the nested, diameter-controlled cylinder partitions required by \cref{thm:LocalEntropyAverages}. Using \cref{prop:FurstenbergMeasure2} and Fubini's theorem, we then pass to a full-measure set of sequences along which the evolving projective direction never becomes orthogonal to the Furstenberg direction $Z(\iii)$, so that the projected entropy in the evolving direction agrees with the entropy in $W$ up to a bounded error. Identifying the time-$q$ map $T_q$ of the suspension flow with these cutsets expresses the resulting entropy averages, via Birkhoff's theorem, as integrals of $f_{q,k}$ against the ergodic components, which \cref{lem:ErgDecompValue} bounds below by $\min\{1,\diml(\nu)\}$. Finally, these estimates verify the local entropy averages hypothesis of \cref{thm:LocalEntropyAverages} in the direction $W$, and the proposition follows.

\begin{proof}[Proof of \cref{prop:fixed-direction-lea}]
  Fix $W \in \RP^{d-1}$. Rescaling and translating $\Phi$ if necessary, we assume that $\diam(X)\le 1$ and $0\in X$, so that $P_W\circ\pi$ maps $\II^\N$ into $[-1,1]$. Let $L = \linspan(X-X)$. Since $X$ is not a singleton by convention, $L$ is non-trivial. Moreover, $A_iL\subset L$ for every $i\in\II$, because $\fii_i(X)\subset X$, and hence $A_iL = L$ since $A_i$ is invertible. Strong $1$-irreducibility therefore gives $L=\R^d$. Thus $X$ is not contained in any affine hyperplane, so $\diam(P_V(X))>0$ for every $V\in\RP^{d-1}$. By compactness of $\RP^{d-1}$ and continuity of $V\mapsto\diam(P_V(X))$, we have $\delta_0 = \inf_{V\in\RP^{d-1}}\diam(P_V(X))>0$. For every $\iii \in \II^\N$ and $i,q \in \N$, let $\eta = \eta_{i,q}(\iii,W)$ be the least integer such that
  \begin{equation*} 
    \|A_{\iii|_{\eta}}^\top|W\| \le 2^{-iq-1}\min_{j\in\II}\alpha_d(A_j).
  \end{equation*}
  The supermultiplicativity bound $\|A_{\kkk}^\top|W\|\ge\alpha_d(A_{k_{|\kkk|}})\|A_{\kkk^-}^\top|W\|$ and the minimality of $\eta$ give
  \begin{equation} \label{eq:partbound}
    2^{-iq-1}\Bigl(\min_{j\in\II}\alpha_d(A_j)\Bigr)^2 < \|A_{\iii|_{\eta}}^\top|W\| \le 2^{-iq-1}\min_{j\in\II}\alpha_d(A_j).
  \end{equation}
  Let $Q_i^q(\iii,W) = \iii|_{\eta_{i,q}(\iii,W)}$ and write $\mathcal Q_i^q(W)=\{Q_i^q(\iii,W) : \iii\in\II^\N\}$. Since all matrices are strict contractions, the stopping time $\eta_{i,q}(\iii,W)$ is finite. The membership of a word in $\mathcal Q_i^q(W)$ depends only on the word itself, and
  \begin{equation*}
    \diam((P_W\circ\pi)([Q_i^q(\iii,W)]))=\|A_{Q_i^q(\iii,W)}^\top|W\|\diam(P_{A_{Q_i^q(\iii,W)}^\top W}(X)),
  \end{equation*}
  which by \cref{eq:partbound} lies in the interval $(\delta_0 2^{-iq-1}(\min_{j\in\II}\alpha_d(A_j))^2,2^{-iq-1}]$ because $\delta_0\le\diam(P_V(X))\le\diam(X)\le 1$ for every $V$. Thus $\mathcal Q_i^q(W)$ is a partition of $\II^\N$ by cylinders, and the displayed bounds give \cref{it:Q1} and \cref{it:Q2}. If $i\ge m$, then the stopping word at level $i$ extends the stopping word at level $m$, giving \cref{it:Q3}. Finally, choose $q_0$ such that $2^{-q}<\delta_0(\min_{j\in\II}\alpha_d(A_j))^2$ for all $q\ge q_0$. If a word belonged to $\mathcal Q_i^q(W)\cap\mathcal Q_m^q(W)$ with $i>m$ and $q\ge q_0$, then the same projected cylinder would have diameter both larger than $\delta_0 2^{-mq-1}(\min_{j\in\II}\alpha_d(A_j))^2$ and at most $2^{-iq-1}$, contradicting
  \begin{equation*}
    2^{-iq-1} \le 2^{-(m+1)q-1} < \delta_0 2^{-mq-1}\Bigl(\min_{j\in\II}\alpha_d(A_j)\Bigr)^2.
  \end{equation*}
  This proves \cref{it:Q4}, so the family $\mathcal Q_i^q(W)$ satisfies the assumptions of \cref{thm:LocalEntropyAverages}.

  Let $\Omega \subseteq \II^\N$ and $Z \colon \Omega \to \RP^{d-1}$ be as in \cref{prop:FurstenbergMeasure2}, so that $\nu(\{\iii : Z(\iii) \perp V\}) = 0$ for every $V \in \RP^{d-1}$. By Fubini's theorem, $\nu\times\mu_F(\{(\iii,V) : Z(\iii)\not\perp V\})=1$, and in particular, there exists $\mathcal{G}_2\subseteq\Omega$ with $\nu(\mathcal{G}_2)=1$ and for every $\iii\in\mathcal{G}_2$ a measurable set $\widehat{\Xi}_{\iii}\subset\RP^{d-1}$ with $\mu_F(\widehat{\Xi}_{\iii})=1$ such that $\bigcup_{\iii\in\mathcal{G}_2}\{\iii\}\times\widehat{\Xi}_{\iii}\subseteq\{(\iii,V) : Z(\iii)\not\perp V\}$. Using the invariance of the Furstenberg measure,
  \begin{equation*}
    1 = \mu_F(\widehat{\Xi}_\iii) = \sum_{i\in\II}p_i\mu_F((A_{i}^\top)^{-1}\widehat{\Xi}_{\iii}),
  \end{equation*}
  and so $\mu_F((A_{i}^\top)^{-1}\widehat{\Xi}_{\iii})=1$ for every $i\in\II$. In particular, for every finite word $\jjj\in\II^*$, $\mu_F((A_{\jjj}^\top)^{-1}\widehat{\Xi}_{\iii})=1$. Let $\Xi_\iii = \bigcap_{\jjj\in\II^*}(A_{\jjj}^\top)^{-1}\widehat{\Xi}_{\iii}$. Then clearly $\mu_F(\Xi_\iii)=1$. Moreover, by \cref{prop:FurstenbergMeasure2}, the set $\mathcal{G}_3=\{\iii\in\Omega : Z(\iii)\not\perp W\}$ satisfies $\nu(\mathcal{G}_3)=1$.

  Using Birkhoff's ergodic theorem and \cref{lem:ErgDecompValue}, let $\mathcal{X}_1\subset\mathcal{X}$ be such that $m(\mathcal{X}_1)=1$, and for every $\mathbf{x}\in\mathcal{X}_1$ and every $q,k\in\N$,
  \begin{equation} \label{eq:birkus}
    \begin{split}
      \lim_{n\to\infty}\frac{1}{n}\sum_{i=0}^{n-1}f_{q,k}(T_q^i(\mathbf{x})) &= \int f_{q,k}(\mathbf{y}) \dd \nu_{q,\mathbf{x}}(\mathbf{y}), \\
      \limsup_{q\to\infty}\frac{1}{q\log 2}\int f_{q,k}(\mathbf{y}) \dd \nu_{q,\mathbf{x}}(\mathbf{y}) &\ge \min\{1,\diml(\nu)\}.
    \end{split}
  \end{equation}
  By Fubini's theorem, there exists $\mathcal{G}_1\subset\II^\N$ with $\nu(\mathcal{G}_1)=1$ and for every $\iii\in\mathcal{G}_1$ a set $\Theta_{\iii}\subset\{(V,t) : V\in\RP^{d-1}\text{ and }t\in(0,r(\iii,V))\}$ such that $m_{\iii}(\Theta_{\iii}^c)=0$ and $\bigcup_{\iii\in\mathcal{G}_1}\{\iii\}\times\Theta_{\iii}\subset\mathcal{X}_1$, where $\dd m(\iii,V,t) = \dd m_{\iii}(V,t) \dd \nu(\iii)$. Set $r_{\min} = -\log_2\max_{j\in\II}\alpha_1(A_j)$ and $r_{\max} = -\log_2\min_{j\in\II}\alpha_d(A_j)$, and let $\mathcal{G} = \{\iii\in\II^\N : \jjj\iii\in\mathcal{G}_1\cap\mathcal{G}_2\cap\mathcal{G}_3\text{ for all }\jjj\in\II^*\}$. Since $\nu$ is Bernoulli and fully supported, $\nu(E\cap[\jjj]) = \nu([\jjj])\nu(\{\iii : \jjj\iii\in E\})$ for every Borel set $E\subseteq\II^\N$ and every $\jjj\in\II^*$, so each of the countably many sets in this intersection has full $\nu$-measure, and hence $\nu(\mathcal{G})=1$. Since for every $\iii\in\mathcal{G}$ and $\jjj\in\II^*$ the concatenation satisfies $\jjj\iii\in\mathcal{G}_1$, and since $r(\iii,V) \ge r_{\min}$ for every $(\iii,V)\in\II^\N\times\RP^{d-1}$, the set $\widehat{\Theta}_{\iii}=\bigcap_{\jjj\in\II^*}\Theta_{\jjj\iii}$ has full $\mu_F\times\mathcal{L}^1|_{(0,r_{\min})}$-measure for every $\iii\in\mathcal{G}$.
 
  Our goal is now to show that
  \begin{equation} \label{eq:goal}
    \limsup_{q \to \infty} \liminf_{n \to \infty} \frac{1}{q\log 2} \cdot \frac{1}{n} \sum_{i=1}^{n} H((P_W\circ\pi)_*\nu_{[Q_i^q(\iii,W)]}, \mathcal{D}_{q(i+1)}) \ge \min \{ 1, \diml(\nu) \}
  \end{equation}
  for all $\iii\in\mathcal{G}$, which implies the claim by \cref{thm:LocalEntropyAverages}. Let $\iii\in\mathcal{G}$ be arbitrary but fixed, and let $(V,t)\in\widehat{\Theta}_\iii\cap(\Xi_{\iii}\times(0,r_{\min}))$ also be arbitrary but fixed. By the construction of the set $\mathcal{G}$,
  \begin{enumerate}
    \item\label{it:good} for every $\jjj\in\II^*$, $(\jjj\iii,V,t)\in\mathcal{X}_1$, and so \cref{eq:birkus} holds for $\mathbf{x}=(\jjj\iii,V,t)$;
    \item\label{it:limits} $\iii\in\Omega$ and, for every $\jjj\in\II^*$, $A_{\jjj}^\top V\not\perp Z(\iii)$ and $W\not\perp Z(\iii)$, so \cref{eq:Furstenberg4} and \cref{eq:Furstenberg3}, applied with $A_{\jjj}^\top V$ in place of $V$, give
    \begin{equation*}
      \lim_{n\to\infty}\frac{\|A_{\iii|_n}^\top|A_{\jjj}^\top V\|}{\|A_{\iii|_n}^\top|W\|}=\frac{|\la Z(\iii),A_{\jjj}^\top V\ra|}{|\la Z(\iii),W\ra|} \qquad \text{and} \qquad \lim_{n\to\infty}d(A_{\iii|_n}^\top A_{\jjj}^\top V,A_{\iii|_n}^\top W)=0.
    \end{equation*}
  \end{enumerate}
  Now fix $\ell \ge 0$ and $\jjj\in\II^\ell$ such that
  \begin{equation} \label{eq:bound1}
    \Bigl(\min_{j\in\II}\alpha_d(A_j)\Bigr)^{\ell} \le \|A_{\jjj}^\top|V\| \le \frac{|\la Z(\iii),W\ra|}{|\la Z(\iii),A_{\jjj}^\top V\ra|}.
  \end{equation}
  The first inequality in \cref{eq:bound1} holds for every $\jjj\in\II^\ell$ since $\alpha_d$ is supermultiplicative, and the second can be arranged since $|\la Z(\iii),W\ra|>0$ and, for every $\jjj\in\II^\ell$, the product
  \begin{equation*}
    \|A_{\jjj}^\top|V\| |\la Z(\iii),A_{\jjj}^\top V\ra|
  \end{equation*}
  tends to zero as $\ell\to\infty$.
 
  Now choose $N=N(\jjj,q,\iii,V,t)$ such that for every $n \ge N$,
  \begin{equation*}
    2^{-1}\frac{|\la Z(\iii),A_{\jjj}^\top V\ra|}{|\la Z(\iii),W\ra|}\le\frac{\|A_{\iii|_n}^\top|A_{\jjj}^\top V\|}{\|A_{\iii|_n}^\top|W\|}\le 2\frac{|\la Z(\iii),A_{\jjj}^\top V\ra|}{|\la Z(\iii),W\ra|}
  \end{equation*}
  and $\min_{\delta\in\{-1,1\}}\sup_{x\in\II^\N}|P_{A_{\iii|_n}^\top A_{\jjj}^\top V}(\pi(x))-\delta P_{A_{\iii|_n}^\top W}(\pi(x))|\le 2^{-q}$; such a choice is possible since $d(A_{\iii|_n}^\top A_{\jjj}^\top V,A_{\iii|_n}^\top W)\to0$ as $n\to\infty$ and $\pi(\II^\N)$ is bounded, the sign $\delta$ accounting for the orientation convention in the definition of $P_V$. Hence, \cref{eq:A3} together with the reflection invariance of dyadic entropy implies
  \begin{equation} \label{eq:compent}
    |H((P_{A_{\iii|_n}^\top A_{\jjj}^\top V}\circ\pi)_*\nu,\mathcal{D}_q)-H((P_{A_{\iii|_n}^\top W}\circ\pi)_*\nu,\mathcal{D}_q)|\le 2C
  \end{equation}
  for all $n \ge N$.

  Observe that, for every $q$ and $i$ with $\eta_{i,q}(\iii,W) \ge N$,
  \begin{align*}
    \|A_{\iii|_{\eta_{i,q}(\iii,W)}}^\top A_{\jjj}^\top|V\|&=\|A_{\jjj}^\top|V\|\frac{\|A_{\iii|_{\eta_{i,q}(\iii,W)}}^\top|A_{\jjj}^\top V\|}{\|A_{\iii|_{\eta_{i,q}(\iii,W)}}^\top|W\|}\|A_{\iii|_{\eta_{i,q}(\iii,W)}}^\top|W\|\\
    &\le \|A_{\jjj}^\top|V\| 2\frac{|\la Z(\iii),A_{\jjj}^\top V\ra|}{|\la Z(\iii),W\ra|}2^{-iq-1}\min_{j\in\II}\alpha_d(A_j)\\
    &\le \min_{j\in\II}\alpha_d(A_j)2^{-iq}\le 2^{-iq-r_{\min}}\le 2^{-iq-t}.
  \end{align*}
  On the other hand, choose $p \in \N$, depending on $\jjj,\iii,V,t$ but not on $i$ or $q$, such that $2^{p r_{\min}-(\ell+2) r_{\max}-2}|\la Z(\iii),A_{\jjj}^\top V\ra|>2^{-t}$. Then
  \begin{align*}
    \|A_{\iii|_{\eta_{i,q}(\iii,W)-p}}^\top A_{\jjj}^\top|V\|&\ge 2^{pr_{\min}}\frac{\|A_{\iii|_{\eta_{i,q}(\iii,W)}}^\top |A_{\jjj}^\top V\|}{\|A_{\iii|_{\eta_{i,q}(\iii,W)}}^\top |W\|}\|A_{\iii|_{\eta_{i,q}(\iii,W)}}^\top |W\|\|A_{\jjj}^\top|V\|\\
    &\ge 2^{pr_{\min}}2^{-1}\frac{|\la Z(\iii),A_{\jjj}^\top V\ra|}{|\la Z(\iii),W\ra|}2^{-iq-1}\Bigl(\min_{j\in\II}\alpha_d(A_j)\Bigr)^{\ell+2}\\
    &\ge 2^{-iq}2^{p r_{\min}-(\ell+2) r_{\max}-2}|\la Z(\iii),A_{\jjj}^\top V\ra|>2^{-iq-t}.
  \end{align*}
  The last two estimates say that the suspension orbit crosses the level $iq+t$ after the prefix $\eta_{i,q}(\iii,W)-p$ has been read and no later than after the prefix $\eta_{i,q}(\iii,W)$ has been read. Hence, for every $q$ and $i$ such that $\eta_{i,q}(\iii,W)-p \ge N$ there exist $\xi\in\{0,1,\ldots,p\}$ and $t'\in[0,r(\sigma^{\eta_{i,q}(\iii,W)-\xi}\iii,A_{\iii|_{\eta_{i,q}(\iii,W)-\xi}}^\top A_{\jjj}^\top V))$ such that
  \begin{equation*}
    T_{q}^i(\jjj\iii,V,t)=(\sigma^{\eta_{i,q}(\iii,W)-\xi}\iii,A_{\iii|_{\eta_{i,q}(\iii,W)-\xi}}^\top A_{\jjj}^\top V,t').
  \end{equation*}
  Since $f_{q,p}$ does not depend on the third coordinate, the value $f_{q,p}(T_q^i(\jjj\iii,V,t))$ is unaffected by $t'$. Moreover, taking $\ell'=\xi\le p$ in the minimum defining $f_{q,p}$ and using $A_{(\sigma^{\eta_{i,q}(\iii,W)-\xi}\iii)|_{\xi}}^\top A_{\iii|_{\eta_{i,q}(\iii,W)-\xi}}^\top = A_{Q_i^q(\iii,W)}^\top$, we have 
  \begin{equation*}
    f_{q,p}(T_q^i(\jjj\iii,V,t))\le H((P_{A_{Q_i^q(\iii,W)}^\top A_{\jjj}^\top V}\circ\pi)_*\nu,\mathcal{D}_q).
  \end{equation*}
  By \cref{eq:partbound}, the integer $\lceil\log_2\|A_{Q_i^q(\iii,W)}^\top|W\|\rceil$ differs from $-iq$ by a bounded amount independent of $i$ and $q$. Therefore, by the self-affinity identity \cref{eq:ProjSelfAffinity}, \cref{eq:A1}, \cref{eq:A3b}, and \cref{eq:compent}, there exists a constant $C'>0$ independent of $i$ and $q$ such that
  \begin{align*}
    H((P_W)_*\pi_*\nu_{[Q_i^q(\iii,W)]}, \mathcal{D}_{q(i+1)})&\ge H((P_{A_{Q_i^q(\iii,W)}^\top W})_*\pi_*\nu, \mathcal{D}_{q})-C'\\
    &\ge H((P_{A_{Q_i^q(\iii,W)}^\top A_{\jjj}^\top V})_*\pi_*\nu, \mathcal{D}_{q})-2C'\\
    &\ge f_{q,p}(T_q^i(\jjj\iii,V,t))-2C'
  \end{align*}
  for all sufficiently large $q$. For each fixed $q$, the inequality $\eta_{i,q}(\iii,W)-p \ge N$ holds for all but finitely many $i$, since $\eta_{i,q}(\iii,W)$ is non-decreasing in $i$ and tends to infinity. We may therefore sum the previous display over the indices for which this inequality holds and then pass to the normalized limit inferior: the omitted entropy terms $H((P_W\circ\pi)_*\nu_{[Q_i^q(\iii,W)]}, \mathcal{D}_{q(i+1)})$ are non-negative, while the omitted terms $f_{q,p}(T_q^i(\jjj\iii,V,t))$ are finitely many fixed terms and hence disappear after division by $n$. Thus, by \cref{eq:birkus} and \cref{it:good},
  \begin{align*}
    \limsup_{q\to\infty}\liminf_{n \to \infty} &\frac{1}{q\log 2} \cdot \frac{1}{n} \sum_{i=1}^{n} H((P_W\circ\pi)_*\nu_{[Q_i^q(\iii,W)]}, \mathcal{D}_{q(i+1)})\\
    &\ge \limsup_{q\to\infty}\liminf_{n \to \infty} \frac{1}{n} \sum_{i=1}^{n} \frac{1}{q\log 2}f_{q,p}(T_q^i(\jjj\iii,V,t))-\frac{2C'}{q\log 2}\\
    &= \limsup_{q\to\infty}\frac{1}{q\log 2}\int f_{q,p}(\mathbf{y}) \dd \nu_{q,(\jjj\iii,V,t)}(\mathbf{y})-\frac{2C'}{q\log 2}\\
    &\ge \min\{1,\diml(\nu)\},
  \end{align*}
  which proves \cref{eq:goal}.
\end{proof}

\section{Zariski dense tuples} \label{sec:zariski-density}

The set projection theorem, \cref{thm:projection}, assumes that the linear parts are strongly pinching and strongly irreducible in every exterior power. This section establishes \cref{thm:assumptions-typical}, that these two hypotheses hold for a typical tuple. The centerpiece of the section is the following genericity statement, whose proof, resting on a criterion that recovers $\GL(d,\R)$ from projective and determinant data, is given in \cref{sec:zariski-dense-generic}.

\begin{proposition} \label{prop:tuples-open-dense}
  The set of Zariski dense tuples in $\GL(d,\R)^N$ is a non-empty Zariski open subset of $\GL(d,\R)^N$. In particular, it is dense in the Zariski topology and open and dense in the Euclidean topology.
\end{proposition}

In \cref{sec:zariski-consequences}, we show in \cref{lem:zariski-irreducible,lem:zariski-strongly-pinching} that every Zariski dense tuple is strongly $k$-irreducible for all $k \in \{1,\ldots,d-1\}$ and strongly pinching. Together with \cref{prop:tuples-open-dense}, these lemmas prove \cref{thm:assumptions-typical}.

\begin{theorem} \label{thm:typical-tuples}
  The set of tuples in $\GL(d,\R)^N$ that are strongly pinching and strongly $k$-irreducible for all $k \in \{1,\ldots,d-1\}$ contains the set of Zariski dense tuples, which is open and dense in $\GL(d,\R)^N$.
\end{theorem}

\begin{proof}
  By \cref{lem:zariski-irreducible,lem:zariski-strongly-pinching}, every Zariski dense tuple in $\GL(d,\R)^N$ is strongly pinching and strongly $k$-irreducible for all $k \in \{1,\ldots,d-1\}$. By \cref{prop:tuples-open-dense}, the set of Zariski dense tuples is open and dense in $\GL(d,\R)^N$.
\end{proof}

The containment in \cref{thm:typical-tuples} can be strict: in \cref{ex:zariski-examples}, we exhibit a strongly pinching and strongly $1$-irreducible pair in $\SL(2,\R)$, necessarily non-contractive, and, in dimension three, a contractive strongly pinching tuple, strongly $k$-irreducible for all $k \in \{1,2\}$, whose generated semigroup preserves a quadratic form up to scalar; neither tuple is Zariski dense. Finally, in \cref{sec:planar-converse}, we treat the two-dimensional case and prove that, for contractive tuples, strong irreducibility and proximality already characterise Zariski density; since proximality and strong pinching coincide for irreducible planar tuples, the containment in \cref{thm:typical-tuples} then becomes an equality for contractive tuples in the plane.

\subsection{Consequences of Zariski density} \label{sec:zariski-consequences}

For each $k \in \{1,\ldots,d-1\}$, the restriction of $A \mapsto A^{\wedge k}$ to $\SL(d,\R) = \{B \in \GL(d,\R) : \det(B) = 1\}$ is the \emph{natural action} of $\SL(d,\R)$ on $\wedge^k \R^d$. The natural action is \emph{irreducible} if the only subspaces of $\wedge^k \R^d$ that are invariant under every map $B^{\wedge k}$ with $B \in \SL(d,\R)$ are $\{0\}$ and $\wedge^k \R^d$.

\begin{lemma} \label{lem:natural-action-irreducible}
  The natural action of $\SL(d,\R)$ on $\wedge^k \R^d$ is irreducible for all $k \in \{1,\ldots,d-1\}$.
\end{lemma}

\begin{proof}
  Fix $k \in \{1,\ldots,d-1\}$. Let $e_1,\ldots,e_d$ be the standard basis of $\R^d$ and write $e_I = e_{i_1} \wedge \cdots \wedge e_{i_k}$ for $I = \{i_1,\ldots,i_k\} \subset \{1,\ldots,d\}$ with $i_1 < \cdots < i_k$. Then $(e_I)_I$ is a basis of $\wedge^k \R^d$. Suppose that $V \subset \wedge^k \R^d$ is a nonzero $\SL(d,\R)$-invariant subspace. To see that $V$ contains some $e_I$, choose real numbers $s_1,\ldots,s_d$ for which the sums $\sum_{i \in I} s_i$ over $k$-element subsets $I \subset \{1,\ldots,d\}$ are pairwise distinct, for instance $s_i = 4^i$, and let $D$ be the diagonal matrix with entries $t_i = \exp(s_i - \tfrac{1}{d}\sum_{j=1}^d s_j)$. Then $\prod_{i=1}^d t_i = 1$, so $D \in \SL(d,\R)$, and the eigenvalue of $D^{\wedge k}$ on $e_I$ is $\prod_{i \in I} t_i = \exp(\sum_{i \in I} s_i - \tfrac{k}{d}\sum_{j=1}^d s_j)$. Since every $I$ has cardinality $k$, the logarithms of these eigenvalues differ from the subset sums $\sum_{i \in I} s_i$ by the common shift $-\tfrac{k}{d}\sum_{j=1}^d s_j$; the logarithms are therefore pairwise distinct, and hence so are the eigenvalues. Hence the basis $(e_I)_I$ diagonalizes $D^{\wedge k}$ with pairwise distinct eigenvalues, so every $D^{\wedge k}$-invariant subspace is spanned by a subset of $(e_I)_I$; as $V$ is nonzero and invariant under $D^{\wedge k}$, it contains some $e_I$. Since $\SL(d,\R)$ acts transitively on the Grassmannian of $k$-dimensional subspaces of $\R^d$, for every $J$ there exists $B \in \SL(d,\R)$ such that $B^{\wedge k}e_I$ is a nonzero scalar multiple of $e_J$. As $V$ is invariant, $B^{\wedge k}e_I \in V$, and since $V$ is a subspace it follows that $e_J \in V$. As $J$ was arbitrary, every basis vector belongs to $V$, and hence $V = \wedge^k \R^d$.
\end{proof}

By combining the polynomial structure of the induced maps $A \mapsto A^{\wedge k}$ with the irreducibility established in \cref{lem:natural-action-irreducible}, one obtains the following consequence of Zariski density.

\begin{lemma} \label{lem:zariski-irreducible}
  Let $\A = (A_1,\ldots,A_N) \in \GL(d,\R)^N$ be Zariski dense. Then $\A$ is strongly $k$-irreducible for all $k \in \{1,\ldots,d-1\}$.
\end{lemma}

\begin{proof}
  Suppose to the contrary that $\A$ does not have the simultaneous escape property. Then there are $n \in \N$ and, for each $j \in \{1,\ldots,n\}$, a choice of $k_j \in \{1,\ldots,d-1\}$, a proper and non-trivial subspace $V_j \subset \wedge^{k_j}\R^d$, and proper subspaces $W_{j,1},\ldots,W_{j,m_j} \subset \wedge^{k_j}\R^d$ such that for every $\iii \in \II^*$ there exist $j$ and $\ell$ with $A_\iii^{\wedge k_j}V_j = W_{j,\ell}$. Writing $Z(k_j,V_j,W_{j,\ell}) = \{B \in \GL(d,\R) : B^{\wedge k_j}V_j = W_{j,\ell}\}$ for the coincidence set, this means that
  \begin{equation*}
    A_\iii \in Z = \bigcup_{j=1}^n \bigcup_{\ell=1}^{m_j} Z(k_j,V_j,W_{j,\ell})
  \end{equation*}
  for every $\iii \in \II^*$. Recall from the proof of \cref{prop:simultaneous-escape} that each $Z(k_j,V_j,W_{j,\ell})$ is Zariski closed, and hence so is $Z$.

  To see that $Z(k_j,V_j,W_{j,\ell})$ is proper, note that if $Z(k_j,V_j,W_{j,\ell}) = \GL(d,\R)$, then $B^{\wedge k_j}V_j = W_{j,\ell}$ for all $B \in \GL(d,\R)$. Evaluating at the identity gives $W_{j,\ell} = V_j$, and therefore $C^{\wedge k_j}V_j = V_j$ for every $C \in \SL(d,\R)$. Since $V_j$ is a proper and non-trivial subspace of $\wedge^{k_j}\R^d$, this contradicts \cref{lem:natural-action-irreducible}. Thus $Z(k_j,V_j,W_{j,\ell})$ is proper, so for each pair $(j,\ell)$ we may choose a polynomial $p_{j,\ell}$ in the entries of $B$ that vanishes on $Z(k_j,V_j,W_{j,\ell})$ but is not identically zero on $\GL(d,\R)$. The union $Z$ is proper as well: if $Z = \GL(d,\R)$, then $\prod_{j=1}^n \prod_{\ell=1}^{m_j} p_{j,\ell}$ vanishes on all of $\GL(d,\R)$, and since $\GL(d,\R)$ is a non-empty Euclidean open subset of $\R^{d^2}$, this product is the zero polynomial, which is impossible because each factor is non-zero. Therefore $Z$ is a proper Zariski closed subset of $\GL(d,\R)$. 
  
  Since $Z$ is a proper Zariski closed subset of $\GL(d,\R)$, there exists a polynomial $p \colon \GL(d,\R) \to \R$ that vanishes on $Z$ but is not identically zero on $\GL(d,\R)$. As $A_\iii \in Z$ for every $\iii \in \II^*$, it follows that $p(A_\iii) = 0$ for all $\iii \in \II^*$, contradicting the Zariski density of $\A$. This contradiction proves that $\A$ has the simultaneous escape property and the claim follows from \cref{prop:simultaneous-escape}.
\end{proof}

We turn to the second consequence, strong pinching, which requires a single element of the semigroup generated by $\A$ that is pinching in every exterior power. Such an element is found by passing to the projective quotient of $\GL(d,\R)$, where the Jordan projection records the eigenvalue moduli, and the Jordan decomposition then yields the required pinching. We fix notation for the quotient and the decomposition before stating the two lemmas. For $d \geq 2$, the \emph{projective linear group} is
\begin{equation*}
  \PGL(d,\R) = \GL(d,\R)/\{tI : t \in \R \setminus \{0\}\},
\end{equation*}
and let $\Pi \colon \GL(d,\R) \to \PGL(d,\R)$ be the quotient map. We equip $\PGL(d,\R)$ with the quotient Zariski topology, in which a subset $Z \subseteq \PGL(d,\R)$ is \emph{Zariski closed} if and only if $\Pi^{-1}(Z)$ is Zariski closed in $\GL(d,\R)$; in particular $\Pi$ is Zariski continuous, and a subset of $\PGL(d,\R)$ is \emph{Zariski dense} if it is contained in no proper Zariski closed subset. For $B \in \GL(d,\R)$, we write
\begin{equation*}
  B = B_eB_hB_u = B_uB_hB_e
\end{equation*}
for its \emph{Jordan decomposition}, where the three factors commute pairwise, $B_e$ is semisimple with all eigenvalues of modulus $1$, $B_h$ is diagonalizable over $\R$ with positive eigenvalues, and $B_u$ is unipotent; for the existence and uniqueness of this decomposition, see \cite[\S 2.4]{Benoist1997} or \cite[\S 6.7.6]{BenoistQuint2016}. If $g \in \PGL(d,\R)$ and $B' \in \GL(d,\R)$ is any lift of $g$, then the \emph{Jordan projection} $\lambda(g)$ is the ordered $d$-tuple whose coordinates are the logarithms of the moduli of the eigenvalues of $B_h'$, arranged in non-increasing order and shifted by a common constant so that their sum is zero. The tuple $\lambda(g)$ does not depend on the choice of lift: two lifts of $g$ differ by a nonzero scalar multiple, which scales the eigenvalues of the hyperbolic part by a common positive factor and hence shifts the coordinates of $\lambda(g)$ by a common constant that the normalization removes. Moreover, $\lambda(g)$ agrees with the intrinsic Jordan projection of $g$, because the Jordan decomposition is compatible with the quotient morphism $\Pi$; see \cite[\S 6.7.6]{BenoistQuint2016}. Benoist's description of the Jordan projections of a Zariski dense semigroup yields an element whose hyperbolic part is non-resonant in every exterior power.

\begin{lemma} \label{lem:zariski-nonresonant}
  Let $\A = (A_1,\ldots,A_N) \in \GL(d,\R)^N$ be Zariski dense. Then there exists $B \in \mathcal{S}(\A)$ such that if $B_h$ denotes the hyperbolic part of $B$ in its Jordan decomposition, then for every $k \in \{1,\ldots,d-1\}$ the eigenvalues of $B_h^{\wedge k}$ have pairwise distinct moduli.
\end{lemma}

\begin{proof}
  We pass to the projective group, where Jordan projections are naturally defined, and then use Benoist's description \cite{Benoist1997} of Jordan projections for Zariski dense semigroups to find an element whose exterior powers have no repeated eigenvalue moduli. The quotient map $\Pi \colon \GL(d,\R) \to \PGL(d,\R)$ is a morphism of real algebraic groups and hence Zariski continuous. We claim that the semigroup $\Gamma = \Pi(\mathcal{S}(\A))$ is Zariski dense in $\PGL(d,\R)$. Indeed, if $\Gamma$ were contained in a proper Zariski closed subset $Z \subset \PGL(d,\R)$, then $\Pi^{-1}(Z)$ would be a proper Zariski closed subset of $\GL(d,\R)$ containing $\mathcal{S}(\A)$. Being Zariski closed, it would then contain the Zariski closure $\overline{\mathcal{S}(\A)}$, which is a subgroup of $\GL(d,\R)$ and in particular contains $I = A_\varnothing$; thus $A_\iii \in \Pi^{-1}(Z)$ for every $\iii \in \II^*$, contradicting the definition of Zariski density.

  Let $\sigma \colon \PGL(d,\R) \to \PGL(d,\R)$ be given by $\sigma(g) = g^2$, and let $\Gamma_2$ be the subsemigroup of $\Gamma$ generated by the squares of elements of $\Gamma$. We square only to reduce to the identity component $\PGL(d,\R)^\circ$, that is, the connected component of $\PGL(d,\R)$ containing the identity element. This is the connected semisimple group needed for Benoist's theorem: it is connected by definition, and semisimple because its Lie algebra is $\mathfrak{pgl}_d(\R) \cong \mathfrak{sl}_d(\R)$, which is semisimple by \cite[Theorem 5.49]{Kirillov2008}. Since $\PGL(d,\R)/\PGL(d,\R)^\circ$ has order at most $2$, every square in $\PGL(d,\R)$ belongs to $\PGL(d,\R)^\circ$, and therefore $\sigma(\PGL(d,\R)) \subset \PGL(d,\R)^\circ$ and $\Gamma_2 \subset \PGL(d,\R)^\circ$. Moreover, $\sigma$ is a morphism of real algebraic varieties, so it is Zariski continuous. The differential of $\sigma$ at the identity is multiplication by $2$ on the Lie algebra of $\PGL(d,\R)$, so the inverse function theorem implies that $\sigma(\PGL(d,\R))$ contains a non-empty Euclidean open subset of $\PGL(d,\R)^\circ$. Since a proper Zariski closed subset of $\PGL(d,\R)^\circ$ cannot contain a non-empty Euclidean open subset, the image $\sigma(\PGL(d,\R))$ is Zariski dense in $\PGL(d,\R)^\circ$. Since $\Gamma$ is Zariski dense in $\PGL(d,\R)$ and $\sigma$ is Zariski continuous, the set $\sigma(\Gamma)$ is Zariski dense in $\sigma(\PGL(d,\R))$. Therefore $\Gamma_2 \supset \sigma(\Gamma)$ is Zariski dense in $\PGL(d,\R)^\circ$.

  Let
  \begin{align*}
    C^+ &= \biggl\{x \in \R^d : x_1 \ge \cdots \ge x_d \text{ and } \sum_{i=1}^d x_i = 0\biggr\}, \\
    C^{++} &= \{x \in C^+ : x_1 > \cdots > x_d\}.
  \end{align*}
  For each $k \in \{1,\ldots,d-1\}$ and each pair of distinct subsets $I,J \subset \{1,\ldots,d\}$ with $\# I = \# J = k$, define
  \begin{equation*}
    H_{k,I,J} = \{x \in C^+ : \sum_{i \in I} x_i = \sum_{j \in J} x_j\},
  \end{equation*}
  which is the intersection of $C^+$ with a proper hyperplane: since $I \ne J$ and $\# I = \# J$, both $I \setminus J$ and $J \setminus I$ are non-empty, so the coefficient vector of $\sum_{i \in I} x_i - \sum_{j \in J} x_j$, equal to $1$ on $I \setminus J$ and $-1$ on $J \setminus I$, is nonzero and not proportional to $(1,\ldots,1)$, and therefore defines a proper hyperplane of the subspace $\{x \in \R^d : \sum_{i=1}^d x_i = 0\}$. For each $k \in \{1,\ldots,d-1\}$, let
  \begin{equation*}
    H_k = \bigcup \{H_{k,I,J} : I,J \subset \{1,\ldots,d\},\; \# I = \# J = k, \text{ and } I \ne J\},
  \end{equation*}
  and define
  \begin{equation*}
    H = \bigcup_{k=1}^{d-1} H_k.
  \end{equation*}
  The set $H$ records exactly the coincidences among subset sums that would produce repeated moduli in some exterior power. Thus it is enough to find a Jordan projection in $C^{++} \setminus H$. We now use Benoist \cite{Benoist1997} in two steps. Both results apply to Zariski dense sub-semigroups, not only subgroups, of a connected semisimple linear real Lie group, which is exactly the setting of the sub-semigroup $\Gamma_2 \subset \PGL(d,\R)^\circ$. First, with the coordinate convention fixed above, Jordan projections in $\PGL(d,\R)^\circ$ are vectors in $C^+$. The theorem stated in \cite[\S 1.2]{Benoist1997} and proved in \cite[\S 7.4]{Benoist1997} therefore shows that the limit cone of the Zariski dense semigroup $\Gamma_2$ has non-empty interior in $C^+$. Since $C^{++}$ is the relative interior of $C^+$ and $H$ is a finite union of such proper hyperplane sections, the set $C^{++} \setminus H$ is a dense open cone in $C^+$. Hence the interior of the limit cone meets $C^{++} \setminus H$, so we may fix an open cone $\Omega \subset C^{++} \setminus H$ contained in this intersection. Second, the lemma of \cite[\S 4.2]{Benoist1997} says that once an open cone meets the limit cone, the elements whose Jordan projections lie in that open cone still form a Zariski dense subset. Therefore the set
  \begin{equation*}
    Y = \{g \in \Gamma_2 : \lambda(g) \in \Omega\}
  \end{equation*}
  is Zariski dense in $\PGL(d,\R)^\circ$. In particular, $Y$ is non-empty, so we may choose $g \in Y$. Write $\lambda(g) = (x_1,\ldots,x_d)$. Since $\Gamma_2 \subset \Gamma = \Pi(\mathcal{S}(\A))$, there exists $\iii \in \II^* \setminus \{\varnothing\}$ such that
  \begin{equation*}
    \Pi(A_\iii) = g.
  \end{equation*}
  Write $B = A_\iii \in \mathcal{S}(\A)$, which is well defined because $\iii$ is nonempty. Then $\lambda(\Pi(B)) = \lambda(g) = (x_1,\ldots,x_d) \in \Omega$. Since $(x_1,\ldots,x_d) \in C^{++}$, there exist $t \in \R$ and $P \in \GL(d,\R)$ such that
  \begin{equation*}
    B_h = P\diag(e^{t+x_1},\ldots,e^{t+x_d})P^{-1}.
  \end{equation*}
  Thus the eigenvalues of $B_h$ have pairwise distinct moduli. Fix $k \in \{1,\ldots,d-1\}$. The eigenvalues of $B_h^{\wedge k}$ are
  \begin{equation*}
    e^{kt+\sum_{i \in I} x_i},
  \end{equation*}
  where $I$ ranges over the $k$-element subsets of $\{1,\ldots,d\}$. Since $(x_1,\ldots,x_d) \notin H$, the numbers $\sum_{i \in I} x_i$ are pairwise distinct for fixed $k$. Therefore the eigenvalues of $B_h^{\wedge k}$ have pairwise distinct moduli for every $k \in \{1,\ldots,d-1\}$. This proves the lemma.
\end{proof}

The Jordan decomposition shows that the element produced in \cref{lem:zariski-nonresonant} is itself pinching in every exterior power, and hence that $\A$ is strongly pinching.

\begin{lemma} \label{lem:zariski-strongly-pinching}
  Let $\A = (A_1,\ldots,A_N) \in \GL(d,\R)^N$ be Zariski dense. Then $\A$ is strongly pinching.
\end{lemma}

\begin{proof}
  By \cref{lem:zariski-nonresonant}, there exists $B \in \mathcal{S}(\A)$ such that if $B = B_eB_hB_u = B_uB_hB_e$ is the Jordan decomposition, then for every $k \in \{1,\ldots,d-1\}$ the eigenvalues of $B_h^{\wedge k}$ have pairwise distinct moduli. We show first that $B$ is pinching and then that the same holds for every exterior power $B^{\wedge k}$.

  Let $\beta_1,\ldots,\beta_d$ be the eigenvalues of $B_h$, and let $E_i$ be the eigenspace corresponding to $\beta_i$. Since $B_h$ is diagonalizable over $\R$ with positive eigenvalues, pairwise distinct moduli mean that the numbers $\beta_1,\ldots,\beta_d$ are distinct. Hence each $E_i$ is one-dimensional and
  \begin{equation*}
    \R^d = E_1 \oplus \cdots \oplus E_d.
  \end{equation*}
  Because $B_e$ and $B_u$ commute with $B_h$, each $E_i$ is invariant under both maps. The restriction of the unipotent map $B_u$ to the line $E_i$ is the identity, so $B_u = I$. Likewise, $B_e|_{E_i}$ is multiplication by a real scalar of modulus $1$, hence by some $\eps_i \in \{-1,1\}$. Choosing a non-zero vector $v_i \in E_i$ for each $i$ and letting $P \in \GL(d,\R)$ be the matrix with columns $v_1,\ldots,v_d$, we obtain
  \begin{equation*}
    B = P\diag(\eps_1\beta_1,\ldots,\eps_d\beta_d)P^{-1}.
  \end{equation*}
  Therefore $B$ is diagonalizable and its eigenvalues have pairwise distinct absolute values, so $B$ is pinching.

  Fix $k \in \{1,\ldots,d-1\}$. Since $B$ is diagonalizable, so is $B^{\wedge k}$. In the basis $v_{i_1} \wedge \cdots \wedge v_{i_k}$ with $1 \le i_1 < \cdots < i_k \le d$, the eigenvalues of $B^{\wedge k}$ are
  \begin{equation*}
    \prod_{r=1}^k \eps_{i_r}\beta_{i_r}.
  \end{equation*}
  Hence the eigenvalues of $B^{\wedge k}$ have the same absolute values as those of $B_h^{\wedge k}$. By \cref{lem:zariski-nonresonant}, these absolute values are pairwise distinct. Therefore $B^{\wedge k}$ is a pinching matrix on $\wedge^k \R^d$. Since the same matrix $B$ works for every $k \in \{1,\ldots,d-1\}$, the tuple $\A$ is strongly pinching.
\end{proof}

\subsection{Genericity of Zariski dense tuples} \label{sec:zariski-dense-generic}

The genericity of Zariski density rests on a criterion that detects, from projective and determinant data, whether an algebraic subgroup is all of $\GL(d,\R)$. The determinant records the scalar information that the quotient $\PGL(d,\R)$ discards, so the two pieces together suffice. 

\begin{lemma} \label{lem:projective-determinant}
  Let $H < \GL(d,\R)$ be an algebraic subgroup. If $\det(H)$ is Zariski dense in $\R \setminus \{0\}$ and $\Pi(H)$ is Zariski dense in $\PGL(d,\R)$, then $H = \GL(d,\R)$.
\end{lemma}

\begin{proof}
  Although the statement is real, the algebraic-group argument below is cleaner after complexification. Over $\C$, we can use the standard Lie-algebra correspondence for algebraic group homomorphisms and the fact that a Zariski closed subgroup of $\GL(d,\C)$ with full dimension is all of $\GL(d,\C)$. We therefore let $H_\C$ be the complex Zariski closure of $H$, prove that $H_\C = \GL(d,\C)$, and then descend to real points using that $H$ is cut out inside $\GL(d,\R)$ by real polynomial equations.

  Let $H_\C$ be the Zariski closure of $H$ in $\GL(d,\C)$. Since $H$ is a subgroup and multiplication and inversion are Zariski continuous, $H_\C$ is an algebraic subgroup of $\GL(d,\C)$. Let $\Pi_\C \colon \GL(d,\C) \to \PGL(d,\C)$ be the quotient map. As for the real points, we endow $\PGL(d,\C)$ with the quotient Zariski topology, in which a subset $Z \subseteq \PGL(d,\C)$ is Zariski closed if and only if $\Pi_\C^{-1}(Z)$ is Zariski closed in $\GL(d,\C)$. We first claim that $\det(H_\C)$ is Zariski dense in $\C \setminus \{0\}$ and $\Pi_\C(H_\C)$ is Zariski dense in $\PGL(d,\C)$. Since $\det(H_\C)$ contains $\det(H)$ and every proper Zariski closed subset of $\C \setminus \{0\}$ is finite, the first claim follows from the assumption that $\det(H)$ is Zariski dense in $\R \setminus \{0\}$. For the second, note that $\GL(d,\R)$ is Zariski dense in $\GL(d,\C)$. Since $\Pi_\C$ is Zariski continuous, it follows that $\PGL(d,\R) = \Pi_\C(\GL(d,\R))$ is Zariski dense in $\PGL(d,\C)$. As $\Pi(H)$ is Zariski dense in $\PGL(d,\R)$ and $\PGL(d,\R)$ is Zariski dense in $\PGL(d,\C)$, the set $\Pi(H)$ is also Zariski dense in $\PGL(d,\C)$. Indeed, a complex Zariski closed subset of $\PGL(d,\C)$ containing $\Pi(H)$ traces on the real points a real Zariski closed subset of $\PGL(d,\R)$, which is all of $\PGL(d,\R)$ by the density of $\Pi(H)$; hence the subset contains $\PGL(d,\R)$ and, since $\PGL(d,\R)$ is Zariski dense in $\PGL(d,\C)$, equals $\PGL(d,\C)$. Therefore $\Pi_\C(H_\C)$ contains a Zariski dense subset of $\PGL(d,\C)$ and is itself Zariski dense.

  We recall the basic facts from the Lie theory of complex algebraic groups used below; see, for example, \cite[\S 3.1 and \S 3.8]{Kirillov2008} and \cite{Humphreys1975}. For an algebraic subgroup $G < \GL(d,\C)$, the \emph{Lie algebra} of $G$ is the tangent space $\mathfrak{g} \subset \mathfrak{gl}_d(\C) = \Mat_d(\C)$ to $G$ at the identity matrix $I$; equivalently, it is the set of $X \in \mathfrak{gl}_d(\C)$ such that $\exp(tX) \in G$ for all sufficiently small $t \in \C$. By \cite[\S 9.1]{Humphreys1975}, the dimension of $\mathfrak{g}$ as a complex vector space equals the dimension of $G$ as an algebraic variety, and by \cite[Theorem~9.1]{Humphreys1975}, the space $\mathfrak{g}$ is closed under the \emph{Lie bracket} $[X,Y] = XY-YX$ on $\mathfrak{gl}_d(\C)$. The Lie algebra of the quotient $\PGL(d,\C) = \GL(d,\C)/\{tI : t \in \C \setminus \{0\}\}$ is $\mathfrak{pgl}_d(\C) = \mathfrak{gl}_d(\C)/\C I$ with the induced bracket by \cite[\S 11.5]{Humphreys1975}, and the Lie algebra of $\C \setminus \{0\}$ is $\C$ with the trivial bracket. On a direct product of algebraic groups, the bracket on the Lie algebra is componentwise. If $\varphi \colon G \to G'$ is a homomorphism of algebraic groups, then its differential $(\mathrm{d}\varphi)_I \colon \mathfrak{g} \to \mathfrak{g}'$ at the identity is a homomorphism of Lie algebras by \cite[Theorem~9.1]{Humphreys1975}, and the Lie algebra of the image $\varphi(G)$ is $(\mathrm{d}\varphi)_I(\mathfrak{g})$. Indeed, the image $\varphi(G)$ is an algebraic subgroup of $G'$ satisfying $\dim G = \dim \Ker(\varphi) + \dim \varphi(G)$ by \cite[Proposition~7.4B]{Humphreys1975}, and since the ground field has characteristic zero, the kernel of $(\mathrm{d}\varphi)_I$ is the Lie algebra of $\Ker(\varphi)$ by \cite[Theorem~12.5]{Humphreys1975}, so the subalgebra $(\mathrm{d}\varphi)_I(\mathfrak{g})$ of the Lie algebra of $\varphi(G)$ has full dimension and the two coincide.

  Our strategy is to compute the Lie algebra of $H_\C$ and show that it is all of $\mathfrak{gl}_d(\C)$; the comparison of dimensions then forces $H_\C$ to equal $\GL(d,\C)$. Let $\mathfrak{h}_\C$ be the Lie algebra of $H_\C$ and define
  \begin{equation*}
    \Psi \colon \GL(d,\C) \to \PGL(d,\C) \times (\C \setminus \{0\}), \qquad \Psi(A) = (\Pi_\C(A),\det(A)).
  \end{equation*}
  Since $\Psi$ is a homomorphism of algebraic groups, its differential $(\mathrm{d}\Psi)_I$ at the identity matrix $I$ is a homomorphism of Lie algebras, and therefore
  \begin{equation*}
    \mathfrak{g} = (\mathrm{d}\Psi)_I(\mathfrak{h}_\C)
  \end{equation*}
  is a Lie subalgebra of $\mathfrak{pgl}_d(\C) \oplus \C$. The restrictions $\Pi_\C|_{H_\C}$ and $\det|_{H_\C}$ are homomorphisms of algebraic groups. Thus their images, equivalently their Zariski closures, have Lie algebras equal to the images of the corresponding differentials at $I$. Since the previous paragraph shows that these images are Zariski dense in their respective targets, their Zariski closures are the whole targets. Therefore the coordinate projections of $\mathfrak{g}$ onto $\mathfrak{pgl}_d(\C)$ and $\C$ are surjective.

  As recalled above, the Lie bracket on $\mathfrak{pgl}_d(\C) \oplus \C$ is componentwise and the second factor is abelian, so every commutator in $\mathfrak{g}$ has zero second component. Therefore $[\mathfrak{g},\mathfrak{g}]$ is contained in $\mathfrak{pgl}_d(\C) \oplus \{0\}$. On the other hand, the first projection $\mathfrak{g} \to \mathfrak{pgl}_d(\C)$ is surjective, so the first projection of $[\mathfrak{g},\mathfrak{g}]$ equals $[\mathfrak{pgl}_d(\C),\mathfrak{pgl}_d(\C)]$. The Lie algebra $\mathfrak{pgl}_d(\C)$ is isomorphic to $\mathfrak{sl}_d(\C)$ and is therefore semisimple by \cite[Theorem 5.49]{Kirillov2008}, which gives $[\mathfrak{pgl}_d(\C),\mathfrak{pgl}_d(\C)] = \mathfrak{pgl}_d(\C)$ by \cite[Corollary 6.5]{Kirillov2008}. Hence
  \begin{equation*}
    [\mathfrak{g},\mathfrak{g}] = \mathfrak{pgl}_d(\C) \oplus \{0\}.
  \end{equation*}
  Since the second projection $\mathfrak{g} \to \C$ is surjective, we may choose $X = (u,c) \in \mathfrak{g}$ with $c \ne 0$. As $(u,0) \in \mathfrak{pgl}_d(\C) \oplus \{0\} \subset \mathfrak{g}$, subtraction gives $(0,c) \in \mathfrak{g}$. Thus $\{0\} \oplus \C \subset \mathfrak{g}$, and together with $[\mathfrak{g},\mathfrak{g}] = \mathfrak{pgl}_d(\C) \oplus \{0\}$ this gives
  \begin{equation*}
    \mathfrak{g} = \mathfrak{pgl}_d(\C) \oplus \C.
  \end{equation*}
  We next show that $(\mathrm{d}\Psi)_I$ is an isomorphism from $\mathfrak{gl}_d(\C)$ onto $\mathfrak{pgl}_d(\C) \oplus \C$. If $(\mathrm{d}\Psi)_I(Y) = 0$, then $(\mathrm{d}\Pi_\C)_I(Y) = 0$, so $Y = \lambda I$ for some $\lambda \in \C$. Also $(\mathrm{d}\det)_I(Y) = \tr(Y) = d\lambda$, and therefore $\lambda = 0$. Thus $(\mathrm{d}\Psi)_I$ is injective. Since both $\mathfrak{gl}_d(\C)$ and $\mathfrak{pgl}_d(\C) \oplus \C$ have dimension $d^2$, it follows that $(\mathrm{d}\Psi)_I$ is an isomorphism. Because $(\mathrm{d}\Psi)_I(\mathfrak{h}_\C) = \mathfrak{g} = \mathfrak{pgl}_d(\C) \oplus \C$, we conclude that $\mathfrak{h}_\C = \mathfrak{gl}_d(\C)$.

  Hence $\dim(H_\C) = d^2 = \dim(\GL(d,\C))$. Since $\GL(d,\C)$ is a non-empty Zariski open subset of the affine space $\C^{d^2}$, which is Zariski irreducible as an algebraic variety, every proper Zariski closed subset of $\GL(d,\C)$ has strictly smaller dimension. It follows that $H_\C = \GL(d,\C)$. Finally, choose real polynomial equations defining $H$ inside $\GL(d,\R)$. Their complexifications vanish on $H_\C = \GL(d,\C)$, so the original real equations vanish on $\GL(d,\R)$ as well. Therefore $H = \GL(d,\R)$.
\end{proof}

Zariski density is a generic condition: arbitrarily small perturbations of a tuple can produce a Zariski dense one, and the property is stable under further perturbations. Intuitively, a Zariski dense tuple generates a semigroup that is algebraically as large as possible: the matrices do not collectively satisfy any non-trivial polynomial identity. A tuple fails to be Zariski dense only when its semigroup is confined to a proper algebraic subgroup of $\GL(d,\R)$, such as the group of matrices preserving a quadratic form.

For pairs of generators in a semisimple algebraic group, the genericity of Zariski density is known. Guralnick \cite[Theorem~3.3]{Guralnick1998} proved that the pairs generating a Zariski dense subgroup of a simple algebraic group over an algebraically closed field of characteristic zero form a Zariski open set, and Breuillard, Green, Guralnick, and Tao \cite[Theorem~4.1]{BreuillardGreenGuralnickTao2012} extended this to semisimple algebraic groups over an arbitrary field $k$ of characteristic zero: for non-commuting words $w$ and $w'$ in the free group on two generators, the pairs $(a,b)$ for which $w(a,b)$ and $w'(a,b)$ generate a Zariski dense subgroup form a Zariski open subvariety defined over $k$ whose set of $k$-points is non-empty.

The group $\GL(d,\R)$ is reductive rather than semisimple, and its determinant carries scalar information that no semisimple quotient can see. The proof therefore again separates the projective and scalar aspects of invertible matrices: the quotient $\PGL(d,\R)$ retains only the action on lines through the origin, where the theorem of Breuillard, Green, Guralnick, and Tao applies, while the determinant records the missing scalar information. By \cref{lem:projective-determinant}, these two pieces together are enough to recover Zariski density in $\GL(d,\R)$. We are now ready to prove \cref{prop:tuples-open-dense}.

\begin{proof}[Proof of \cref{prop:tuples-open-dense}]
  Let $Z \subset \GL(d,\R)^N$ be the set of Zariski dense tuples. The key point is to separate the projective information from the scalar information: the quotient $\PGL(d,\R)$ forgets scalar multiples, while the determinant records the missing scalar factor.

  We begin with an elementary observation. If a subsemigroup of $\GL(d,\R)$ contains the identity, then its Zariski closure, meaning the smallest Zariski closed set containing it, is automatically an algebraic subgroup, that is, a subgroup which is Zariski closed. Indeed, let $S \subset \GL(d,\R)$ be such a semigroup, and let $H$ be its Zariski closure. For each $s \in S$, the map $x \mapsto sx$ is Zariski continuous, and since $sS \subset S \subset H$, the Zariski density of $S$ in $H$ implies $sH \subset H$. Now fix $h \in H$. The map $x \mapsto xh$ is Zariski continuous, and since $Sh \subset H$, again by the Zariski density of $S$ in $H$ we obtain $Hh \subset H$. Hence $HH \subset H$. Next fix $s \in S$. The chain of Zariski closed sets
  \begin{equation*}
    H \supset sH \supset s^2H \supset \cdots
  \end{equation*}
  stabilizes by the descending chain condition for Zariski closed sets. Hence $s^nH = s^{n+1}H$ for some $n \geq 0$. Multiplying by $s^{-n}$ on the left yields $H = sH$. Since the identity belongs to $H$, there exists $h \in H$ such that $sh = I$, and therefore $s^{-1} \in H$. This holds for every $s \in S$, so $S^{-1} \subset H$. As inversion is Zariski continuous, we obtain
  \begin{equation*}
    H^{-1} = \overline{S}^{-1} = \overline{S^{-1}} \subset H.
  \end{equation*}
  Thus $H$ is closed under multiplication and inversion, so $H$ is an algebraic subgroup of $\GL(d,\R)$.

  Even though our standing assumption is $d \ge 2$, let us also cover the case $d = 1$. Then $\GL(d,\R) = \R \setminus \{0\}$ and a tuple $(A_1,\ldots,A_N) \in \GL(d,\R)^N$ is Zariski dense if and only if $A_i^2 \neq 1$ for some $i \in \{1,\ldots,N\}$. Since every proper Zariski closed subset of $\R \setminus \{0\}$ is finite, the complement of
  \begin{equation*}
    \{(A_1,\ldots,A_N) \in \GL(d,\R)^N : A_i^2 = 1 \text{ for all } i \in \{1,\ldots,N\}\}
  \end{equation*}
  is a dense open subset of $\GL(d,\R)^N$, and it is exactly $Z$. Therefore the proposition follows in this case, and in the remainder of the proof we use $d \ge 2$.

  We now characterize the tuples in $Z$. The idea is to detect Zariski density in $\PGL(d,\R)$ by looking for two words in the generators whose projective images already generate a Zariski dense subgroup. By the characteristic-zero case of \cite[Theorem~4.1]{BreuillardGreenGuralnickTao2012} applied to the semisimple algebraic group $\PGL(d,\R)$ and the non-commuting words $x_1,x_2 \in F_2$, the set
  \begin{align*}
    U = \{(g,h) \in \PGL(d,\R)^2 : \;&\text{the subgroup generated by } g \\
    &\text{and } h \text{ is Zariski dense in } \PGL(d,\R)\}
  \end{align*}
  is a non-empty Zariski open subset of $\PGL(d,\R)^2$. Equivalently, the exceptional pairs lie in a proper algebraic subset, so from the Zariski point of view a typical pair already generates as large a subgroup as possible. Let $F_N$ be the free group on generators $x_1,\ldots,x_N$, so its elements are words in the letters $x_i^{\pm 1}$. For a word $w \in F_N$ and a tuple $\A = (A_1,\ldots,A_N) \in \GL(d,\R)^N$, we write $w(\A)$ for the image of $w$ under the homomorphism $F_N \to \GL(d,\R)$ determined by $x_i \mapsto A_i$; concretely, $w(\A)$ is the matrix obtained by substituting $A_i^{\pm 1}$ for each letter $x_i^{\pm 1}$, and it belongs to the group generated by $A_1,\ldots,A_N$. For every pair of words $u,v \in F_N$, define
  \begin{equation*}
    U_{u,v} = \{(A_1,\ldots,A_N) \in \GL(d,\R)^N : (\Pi(u(\A)),\Pi(v(\A))) \in U\}.
  \end{equation*}
  Thus $U_{u,v}$ consists of those tuples for which the two words $u(\A)$ and $v(\A)$ already witness projective Zariski density. Since multiplication and inversion on $\GL(d,\R)$ are algebraic maps, and $\Pi \colon \GL(d,\R) \to \PGL(d,\R)$ is the algebraic quotient map, the map
  \begin{equation*}
    \A \mapsto (\Pi(u(\A)),\Pi(v(\A)))
  \end{equation*}
  is algebraic. As $U$ is Zariski open, each $U_{u,v}$ is a Zariski open, and hence open, subset of $\GL(d,\R)^N$.

  We next identify exactly what these sets detect. We claim that
  \begin{equation*}
    \bigcup_{u,v \in F_N} U_{u,v}
  \end{equation*}
  is exactly the set of tuples $\A \in \GL(d,\R)^N$ for which the subgroup generated by $\Pi(A_1),\ldots,\Pi(A_N)$ is Zariski dense in $\PGL(d,\R)$. If $\A \in U_{u,v}$ for some $u,v \in F_N$, then the subgroup generated by $\Pi(A_1),\ldots,\Pi(A_N)$ contains the Zariski-dense subgroup generated by $\Pi(u(\A))$ and $\Pi(v(\A))$, and is therefore Zariski dense. Conversely, let $\Gamma$ be the subgroup generated by $\Pi(A_1),\ldots,\Pi(A_N)$ and suppose that $\Gamma$ is Zariski dense in $\PGL(d,\R)$. Then $\Gamma \times \Gamma$ is Zariski dense in $\PGL(d,\R)^2$: if a Zariski closed set $F \subseteq \PGL(d,\R)^2$ contains $\Gamma \times \Gamma$, then for each $g \in \Gamma$ the slice $\{k \in \PGL(d,\R) : (g,k) \in F\}$ is Zariski closed and contains $\Gamma$, hence equals $\PGL(d,\R)$, so that $\Gamma \times \PGL(d,\R) \subseteq F$; repeating the argument in the first coordinate for each fixed $k \in \PGL(d,\R)$ gives $F = \PGL(d,\R)^2$. Since $U$ is a non-empty Zariski open subset of $\PGL(d,\R)^2$, some pair of elements of $\Gamma$ must lie in $U$. In other words, once the whole projective subgroup is Zariski dense, two suitable words already detect this. More formally, we have
  \begin{equation*}
    U \cap (\Gamma \times \Gamma) \neq \emptyset.
  \end{equation*}
  Hence there exist $g,h \in \Gamma$ such that $(g,h) \in U$. Writing $g = \Pi(u(\A))$ and $h = \Pi(v(\A))$ for suitable words $u,v \in F_N$, we conclude that $\A \in U_{u,v}$.

  The projective image alone does not determine whether $\A$ is Zariski dense in $\GL(d,\R)$, because scalar matrices disappear in $\PGL(d,\R)$. We therefore also record the determinant, which detects the missing scalar direction. Let
  \begin{equation*}
    D = \{(A_1,\ldots,A_N) \in \GL(d,\R)^N : \det(A_i)^2 \neq 1 \text{ for some } i \in \{1,\ldots,N\}\}.
  \end{equation*}
  This is a Zariski open, and hence open, subset of $\GL(d,\R)^N$, since its complement is the common zero set of the polynomials $\A \mapsto \det(A_i)^2 - 1$ for $i \in \{1,\ldots,N\}$. We claim that
  \begin{equation*}
    Z = D \cap \bigcup_{u,v \in F_N} U_{u,v}.
  \end{equation*}
  If $\A \in Z$, then the subgroup generated by $\Pi(A_1),\ldots,\Pi(A_N)$ is Zariski dense in $\PGL(d,\R)$, so $\A \in U_{u,v}$ for some $u,v \in F_N$ by the preceding paragraph. Moreover, since $\overline{\mathcal{S}(\A)} = \GL(d,\R)$ and the determinant map $\det \colon \GL(d,\R) \to \R \setminus \{0\}$ is Zariski continuous, we have
  \begin{equation*}
    \R \setminus \{0\} = \det(\GL(d,\R)) = \det(\overline{\mathcal{S}(\A)}) \subset \overline{\det(\mathcal{S}(\A))}.
  \end{equation*}
  Thus $\det(\mathcal{S}(\A))$ is Zariski dense in $\R \setminus \{0\}$, so it cannot be contained in the finite set $\{1,-1\}$. Hence $\det(A_i)^2 \neq 1$ for some $i$, and thus $\A \in D$.

  Conversely, suppose that
  \begin{equation*}
    \A \in D \cap \bigcup_{u,v \in F_N} U_{u,v},
  \end{equation*}
  and let $H$ be the Zariski closure of $\mathcal{S}(\A) \cup \{I\}$ in $\GL(d,\R)$. By the opening observation, $H$ is an algebraic subgroup of $\GL(d,\R)$. Since $H$ contains $\mathcal{S}(\A)$ and is a subgroup, it also contains the group generated by $A_1,\ldots,A_N$. The condition $\A \in U_{u,v}$ tells us that the projective image of this group is Zariski dense in $\PGL(d,\R)$. Because this group is contained in $H$, it follows that $\Pi(H)$ is Zariski dense in $\PGL(d,\R)$. The condition $\A \in D$ supplies the missing scalar information: there exists $i \in \{1,\ldots,N\}$ with $\det(A_i)^2 \neq 1$. The set $\{\det(A_i)^n : n \geq 1\}$ is then infinite and contained in $\det(H)$. As every proper Zariski closed subset of $\R \setminus \{0\}$ is finite, it follows that $\det(H)$ is Zariski dense in $\R \setminus \{0\}$. Therefore \cref{lem:projective-determinant} gives $H = \GL(d,\R)$. Since $H = \overline{\mathcal{S}(\A) \cup \{I\}} = \overline{\mathcal{S}(\A)} \cup \{I\}$ and $\GL(d,\R)$ is Zariski irreducible, being a non-empty Zariski open subset of the affine space $\R^{d^2}$, while $\{I\}$ is a proper Zariski closed subset, we conclude that $\overline{\mathcal{S}(\A)} = \GL(d,\R)$; that is, $\A \in Z$.

  The characterization $Z = D \cap \bigcup_{u,v \in F_N} U_{u,v}$ shows that $Z$ is Zariski open, and hence open: each $U_{u,v}$ is Zariski open by the algebraicity above, and by the descending chain condition the intersection of the complementary Zariski closed sets $\GL(d,\R)^N \setminus U_{u,v}$ reduces to a finite subintersection, so it is Zariski closed and the union $\bigcup_{u,v} U_{u,v}$ is Zariski open; intersecting with the Zariski open set $D$ preserves Zariski openness. To prove that $Z$ is dense, it is enough to exhibit one dense open subset contained in $Z$. Consider the set
  \begin{equation*}
    \{(A_1,\ldots,A_N) \in \GL(d,\R)^N : (\Pi(A_1),\Pi(A_2)) \in U \text{ and } \det(A_1)^2 \neq 1\}.
  \end{equation*}
  This set is contained in $Z$ by the characterization above, taking $u = x_1$ and $v = x_2$. It is dense open because, by the standing assumption $N \geq 2$, the map $\Pi \times \Pi \colon \GL(d,\R)^2 \to \PGL(d,\R)^2$ is a surjective algebraic map, so $(\Pi \times \Pi)^{-1}(U)$ is a non-empty Zariski open, and hence dense open, subset of $\GL(d,\R)^2$. The condition $\det(A_1)^2 \neq 1$ is dense open in $\GL(d,\R)$, and taking the product with $\GL(d,\R)^{N-2}$ preserves density and openness. Thus $Z$ is dense. 
\end{proof}

\subsection{The two-dimensional case} \label{sec:planar-converse}

In dimension two, the decomposition provided by \cref{lem:projective-determinant} yields a complete description of the contractive strongly irreducible tuples: they are Zariski dense exactly when proximal, and are otherwise conjugate to a system of contracting similarities. The algebraic subgroups of $\PGL(2,\R)$ are few and well understood: the following dichotomy isolates the only proper possibility compatible with strong irreducibility, namely a group conjugate into $\Pi(\GO(2))$. We record it first, as both the converse below and the non-proximal structural \cref{lem:planar-nonproximal-conformal} rely on it: strong irreducibility narrows the projective image to either $\PGL(2,\R)$ or the conformal group $\Pi(\GO(2))$, proximality distinguishes the two, and contractivity makes the determinants Zariski dense in $\R \setminus \{0\}$.

\begin{lemma} \label{lem:planar-projective-dichotomy}
  Let $G$ be an algebraic subgroup of $\PGL(2,\R)$ that preserves no non-empty finite subset of $\RP^1$. Then either $G = \PGL(2,\R)$, or there exists $g \in \PGL(2,\R)$ such that the identity component of $gGg^{-1}$ is $\Pi(\SO(2))$ and $gGg^{-1} \subseteq \Pi(\GO(2))$.
\end{lemma}

\begin{proof}
  Let $L$ be the connected component of $G$ containing the identity element; throughout the proof, topological notions refer to the Euclidean topology unless the Zariski topology is mentioned explicitly. Being Zariski closed, $G$ is closed also in the Euclidean topology, since $\Pi^{-1}(G)$ is the common zero set of finitely many polynomials, and is thus a closed subgroup of the Lie group $\PGL(2,\R)$, whose Lie algebra is $\mathfrak{pgl}_2(\R) \cong \mathfrak{sl}_2(\R)$. By \cite[Theorem~2.9]{Kirillov2008}, $G$ is then a Lie subgroup; let $\mathfrak{g} \subseteq \mathfrak{sl}_2(\R)$ be its Lie algebra, consisting of the trace-zero matrices $X$ with $\Pi(\exp(tX)) \in G$ for all $t \in \R$. By \cite[Theorem~2.6]{Kirillov2008}, $L$ is a normal subgroup of $G$, and $L$ is generated by $\{\Pi(\exp(X)) : X \in \mathfrak{g}\}$, since this set contains a neighborhood of the identity in $G$ by \cite[Theorem~3.7]{Kirillov2008} applied to the Lie group $G$, and a neighborhood of the identity generates the connected Lie group $L$ by \cite[Corollary~2.10]{Kirillov2008}. Moreover, $G$ has finitely many connected components. Indeed, if the real polynomials $p_1,\ldots,p_k$ cut $\Pi^{-1}(G)$ out of $\GL(2,\R)$, then $\{(A,s) \in \Mat_2(\R) \times \R : s\det(A) = 1 \text{ and } p_1(A) = \cdots = p_k(A) = 0\}$, in which the auxiliary coordinate $s = \det(A)^{-1}$ encodes the invertibility of $A$ as a polynomial condition, is a real algebraic subset of $\R^5$ and hence has finitely many connected components by \cite[\S 2.4]{BochnakCosteRoy1998}, and $G$ is the image of this set under the continuous map $(A,s) \mapsto \Pi(A)$.

  We record two observations. First, a non-identity element of $\PGL(2,\R)$ fixes at most two points of $\RP^1$: its lifts are non-scalar, two distinct eigendirections of a matrix are linearly independent, and a third eigendirection, being a combination of the first two with non-zero coefficients, would force the corresponding eigenvalues to coincide and the matrix to be scalar. In particular, if $L$ is non-trivial, then the set $F \subseteq \RP^1$ of points fixed by every element of $L$ has at most two elements. Second, $F$ is invariant under $G$: if $x \in F$, $h \in G$, and $l \in L$, then the normality of $L$ gives $l(hx) = h(h^{-1}lh)x = hx$. If $G = \PGL(2,\R)$, there is nothing to prove, so assume $G \neq \PGL(2,\R)$. If $L$ is trivial, then $G$ is discrete and, having finitely many connected components, finite, so any $G$-orbit in $\RP^1$ is a non-empty finite $G$-invariant set, contradicting the hypothesis; hence $\dim(\mathfrak{g}) \geq 1$. If $\dim(\mathfrak{g}) = 3$, then $\mathfrak{g} = \mathfrak{sl}_2(\R)$, so $G$ contains a neighborhood of the identity in $\PGL(2,\R)$ and, being a subgroup, is open in $\PGL(2,\R)$; this is impossible, since $\Pi^{-1}(G)$ would then contain a non-empty Euclidean open subset of $\GL(2,\R)$, so each $p_j$ would vanish on a non-empty Euclidean open subset of $\R^4$ and hence identically, forcing $\Pi^{-1}(G) = \GL(2,\R)$ and $G = \PGL(2,\R)$, contrary to our assumption. Hence $\dim(\mathfrak{g}) \in \{1,2\}$.

  Suppose that $\dim(\mathfrak{g}) = 2$. A non-zero element of $\mathfrak{sl}_2(\R)$ is either nilpotent or has two distinct eigenvalues, and in both cases the matrices commuting with it are the linear combinations of it and the identity, so its centralizer in $\mathfrak{sl}_2(\R)$ is the line it spans. Therefore $\mathfrak{g}$ is not abelian, since an abelian $\mathfrak{g}$ would be contained in the centralizer of each of its non-zero elements, a line. By bilinearity and antisymmetry, the bracket of any two elements of $\mathfrak{g}$ is a scalar multiple of the bracket $Y$ of a fixed basis of $\mathfrak{g}$, and $Y \neq 0$ since $\mathfrak{g}$ is not abelian. As the centralizer of $Y$ cannot contain the two-dimensional $\mathfrak{g}$, there exists $X \in \mathfrak{g}$ with $[X,Y] = \lambda Y$ for some $\lambda \neq 0$, and after replacing $X$ by $\lambda^{-1}X$ we have $[X,Y] = Y$. Then $\{X,Y\}$ is a basis of $\mathfrak{g}$, since $X \in \R Y$ would force $[X,Y] = 0$. Multiplying $[X,Y] = Y$ by $Y$ from the right and taking traces gives $\tr(Y^2) = \tr(XY^2) - \tr(YXY) = 0$ by the cyclic invariance of the trace, and since the trace-zero matrix $Y$ satisfies $Y^2 = -\det(Y)I$ by the Cayley--Hamilton theorem, we conclude that $\det(Y) = 0$; thus $Y$ is nilpotent and non-zero, and its kernel is a line $\ell \subset \R^2$. If $v \in \ell$ and $Z = aX + bY \in \mathfrak{g}$, then $YZv = ZYv - [Z,Y]v = -aYv = 0$, so every element of $\mathfrak{g}$ maps $\ell$ into itself. Consequently, every $\Pi(\exp(Z))$ with $Z \in \mathfrak{g}$ fixes the point $\ell \in \RP^1$, and hence so does every element of $L$. Thus $F$ is a non-empty finite $G$-invariant subset of $\RP^1$, contradicting the hypothesis.

  Suppose then that $\dim(\mathfrak{g}) = 1$ and write $\mathfrak{g} = \R X$ for a non-zero $X \in \mathfrak{sl}_2(\R)$, so that $L = \{\Pi(\exp(tX)) : t \in \R\}$, as this set is a subgroup and generates $L$. The set $F$ consists exactly of the real eigendirections of $X$: every eigendirection of $X$ is fixed by every $\exp(tX)$, and conversely, if $\exp(tX)v \in \R v$ for all $t \in \R$, then differentiating at $t = 0$ shows that $Xv \in \R v$. Since the characteristic polynomial of the trace-zero matrix $X$ is $\lambda^2 + \det(X)$, there are three possibilities. If $\det(X) < 0$, then $X$ has two real eigendirections, and if $\det(X) = 0$, then $X$ is nilpotent and non-zero and has exactly one; in both cases $F$ is a non-empty finite $G$-invariant subset of $\RP^1$, contradicting the hypothesis. Hence $\det(X) > 0$, and after rescaling $X$ we may assume $\det(X) = 1$, so that $X$ has the eigenvalues $\pm i$. The real canonical form provides $P \in \GL(2,\R)$ with $PXP^{-1} = J$, where $J$ denotes the rotation by the angle $\pi/2$, the standard complex structure on $\R^2$. Since $\exp(tJ)$ is the rotation by the angle $t$, conjugating by $\Pi(P)$ gives $\Pi(P)L\Pi(P)^{-1} = \{\Pi(\exp(tJ)) : t \in \R\} = \Pi(\SO(2))$. Hence $L$ is conjugate, by an element of $\PGL(2,\R)$, to $\Pi(\SO(2))$.

  Choose $g \in \PGL(2,\R)$ with $gLg^{-1} = \Pi(\SO(2))$. Since $L$ is normal in $G$, the group $gGg^{-1}$ normalizes $\Pi(\SO(2))$. The normalizer of $\Pi(\SO(2))$ in $\PGL(2,\R)$ is $\Pi(\GO(2))$: the group $\Pi(\GO(2))$ normalizes $\Pi(\SO(2))$ since $\SO(2)$ is a normal subgroup of $\GO(2)$. Conversely, let $A \in \GL(2,\R)$ be a lift of a projective class normalizing $\Pi(\SO(2))$, and recall that $J$ has trace zero and spectrum $\{\pm i\}$. For every $t \in \R$ the class of $A\exp(tJ)A^{-1} = \exp(tAJA^{-1})$ belongs to $\Pi(\SO(2))$, so the matrix $\exp(tAJA^{-1})$ is a non-zero real scalar multiple of a rotation and lies in the two-dimensional linear space $\R I + \R J$; differentiating at $t = 0$ gives $AJA^{-1} \in \R I + \R J$, and since conjugation preserves the trace, $AJA^{-1}$ lies in the trace-zero part $\R J$. Hence $AJA^{-1} = cJ$ for some $c \in \R \setminus \{0\}$; since $AJA^{-1}$ has the same spectrum $\{\pm i\}$ as $J$ while $cJ$ has spectrum $\{\pm ci\}$, we conclude that $c = \pm 1$. Thus $A$ commutes or anticommutes with $J$, hence is complex-linear or complex-antilinear with respect to $J$, and is therefore a non-zero scalar multiple of an orthogonal matrix: identifying $(\R^2,J)$ with $\C$, a complex-linear $A$ acts as $z \mapsto \alpha z$ for some $\alpha \neq 0$, which is $|\alpha|$ times a rotation, and a complex-antilinear $A$ acts as $z \mapsto \alpha \overline{z}$, which is $|\alpha|$ times a reflection. Therefore $gGg^{-1} \subseteq \Pi(\GO(2))$, and since conjugation by $g$ is a homeomorphism of $\PGL(2,\R)$, the identity component of $gGg^{-1}$ is $gLg^{-1} = \Pi(\SO(2))$.
\end{proof}

In any dimension, a Zariski dense tuple is strongly $1$-irreducible and strongly pinching by \cref{lem:zariski-irreducible,lem:zariski-strongly-pinching}, and in the plane strong pinching yields $1$-proximality. With \cref{lem:planar-projective-dichotomy} available, the following proposition establishes the reverse implication for contractive tuples, so that $1$-proximality and strong $1$-irreducibility together characterize Zariski density. The forward implication is used in the proof of \cref{lem:planar-nonproximal-conformal}, to conclude that a strongly irreducible non-proximal tuple is not Zariski dense.

\begin{proposition} \label{prop:zariski-irreducible-converse-planar}
  Let $\A = (A_1,\ldots,A_N) \in \GL(2,\R)^N$ satisfy $\|A_i\| < 1$ for all $i \in \{1,\ldots,N\}$. Then $\A$ is Zariski dense if and only if $\A$ is $1$-proximal and strongly $1$-irreducible.
\end{proposition}

\begin{proof}
  If $\A$ is Zariski dense, then \cref{lem:zariski-irreducible,lem:zariski-strongly-pinching} show that $\A$ is strongly $1$-irreducible and strongly pinching. Since $d = 2$, strong pinching means that there exists $B \in \mathcal{S}(\A)$ such that $B$ is pinching, that is, $B$ is diagonalizable and its eigenvalues have distinct absolute values. The eigenvalues $\lambda_1$ and $\lambda_2$ of $B$ are both real, since a complex-conjugate pair would have equal moduli. Ordering them so that $|\lambda_1| > |\lambda_2|$, we see that $\lambda_1^{-n}B^n$ converges to a rank-one linear map as $n \to \infty$. Hence $\A$ is $1$-proximal.

  Assume now that $\A$ is $1$-proximal and strongly $1$-irreducible. Let $S = \mathcal{S}(\A) \cup \{I\}$ be the monoid generated by $\A$, and let $H$ be the Zariski closure of $S$ in $\GL(2,\R)$. We first check that $H$ is an algebraic subgroup of $\GL(2,\R)$. For each $B \in S$, the map $C \mapsto BC$ is Zariski continuous, and since $BS \subset S \subset H$, the Zariski density of $S$ in $H$ implies $BH \subset H$. Now fix $C \in H$. The map $B \mapsto BC$ is again Zariski continuous, and since $SC \subset H$, the Zariski density of $S$ in $H$ gives $HC \subset H$. Hence $HH \subset H$. Next fix $B \in S$. The descending chain $H \supset BH \supset B^2H \supset \cdots$ stabilizes, so $B^nH = B^{n+1}H$ for some $n \geq 0$. Multiplying by $B^{-n}$ on the left yields $H = BH$. Since $I \in H$, there exists $C \in H$ such that $BC = I$, and therefore $B^{-1} \in H$. This holds for every $B \in S$, so $S^{-1} \subset H$. As inversion is Zariski continuous, we have $H^{-1} = \overline{S}^{-1} = \overline{S^{-1}} \subset H$. Thus $H$ is closed under multiplication and inversion, and therefore $H$ is an algebraic subgroup of $\GL(2,\R)$.

  We next show that the projective image of $H$ is all of $\PGL(2,\R)$. Since $\A$ is strongly $1$-irreducible, it is irreducible. As $d = 2$, the assumption that $\A$ is $1$-proximal means that $\A$ is proximal. Therefore, since a proximal irreducible tuple in $\GL(2,\R)$ is pinching, $\A$ is pinching, and there exists $\iii \in \II^* \setminus \{\varnothing\}$ such that $A_\iii$ is pinching. Let $M$ be the Zariski closure of $\Pi(H)$ in $\PGL(2,\R)$. Since $\Pi(H)$ is a subgroup and multiplication and inversion on $\PGL(2,\R)$ are Zariski continuous, $M$ is an algebraic subgroup of $\PGL(2,\R)$. If $M$ preserved a non-empty finite subset of $\RP^1$, then, since each $\Pi(A_i)$ belongs to $\Pi(H) \subseteq M$, the corresponding finite union of lines in $\R^2$ would be invariant under every $A_i$, contradicting strong $1$-irreducibility. Hence \cref{lem:planar-projective-dichotomy} applies: either $M = \PGL(2,\R)$, or $M$ is conjugate into $\Pi(\GO(2))$. In the latter case every element of $M$ is the class of a scalar multiple of a conjugate of an orthogonal matrix, and hence every representative has eigenvalues of equal absolute value; but $A_\iii$ is pinching, so $\Pi(A_\iii) \in M$ has a representative with eigenvalues of distinct absolute value, a contradiction. Therefore $M = \PGL(2,\R)$, and $\Pi(H)$ is Zariski dense in $\PGL(2,\R)$.

  Finally, contractivity gives the determinant information. Since $\|A_1\| < 1$ and $A_1$ is invertible, the singular values of $A_1$ satisfy $0 < \alpha_2(A_1) \leq \alpha_1(A_1) = \|A_1\| < 1$, and therefore $0 < |\det(A_1)| = \alpha_1(A_1)\alpha_2(A_1) < 1$. Therefore the set $\{\det(A_1)^n : n \geq 1\}$ is infinite and contained in $\det(H)$. As every proper Zariski closed subset of $\R \setminus \{0\}$ is finite, it follows that $\det(H)$ is Zariski dense in $\R \setminus \{0\}$. Hence \cref{lem:projective-determinant} gives $H = \GL(2,\R)$. The group $\GL(2,\R)$ is Zariski irreducible, being a non-empty Zariski open subset of the affine space $\R^4$, and therefore cannot be written as a union of two proper Zariski closed subsets. Since
  \begin{equation*}
    \GL(2,\R) = H = \overline{\mathcal{S}(\A) \cup \{I\}} = \overline{\mathcal{S}(\A)} \cup \{I\}
  \end{equation*}
  and $\{I\}$ is a proper Zariski closed subset, it follows that $\overline{\mathcal{S}(\A)} = \GL(2,\R)$. Therefore $\A$ is Zariski dense.
\end{proof}

The hypotheses of \cref{prop:zariski-irreducible-converse-planar} are sharp. Without contractivity, the determinants of all semigroup elements may equal $1$, confining the semigroup to $\SL(2,\R)$. In higher dimensions, even contractive tuples satisfying strong pinching and strong irreducibility can fail to be Zariski dense when the semigroup preserves additional algebraic structure, such as a quadratic form.

\begin{example} \label{ex:zariski-examples}
  (1) The contractive hypothesis in \cref{prop:zariski-irreducible-converse-planar} cannot be omitted. Take
  \begin{equation*}
    A_1 =
    \begin{pmatrix}
      2 & 0 \\
      0 & 1/2
    \end{pmatrix}
    \qquad \text{and} \qquad
    A_2 =
    \begin{pmatrix}
      \cos \theta & -\sin \theta \\
      \sin \theta & \cos \theta
    \end{pmatrix},
  \end{equation*}
  where $\theta/\pi \notin \Q$. Then $A_1,A_2 \in \SL(2,\R)$, so $\A = (A_1,A_2)$ is contained in the proper algebraic subgroup $\SL(2,\R)$ of $\GL(2,\R)$ and hence is not Zariski dense in $\GL(2,\R)$. The matrix $A_1$ is pinching, so $\A$ is strongly pinching. Moreover, $A_2$ has infinite order and no finite orbit on $\RP^1$, so $\A$ is strongly $1$-irreducible. In fact, the semigroup generated by $A_1$ and $A_2$ is Zariski dense in $\SL(2,\R)$.

  (2) The converse fails in dimension three, even for contractive tuples which are strongly pinching and strongly $k$-irreducible for all $k \in \{1,2\}$. Let
  \begin{equation*}
    Q =
    \begin{pmatrix}
      1 & 0 & 0 \\
      0 & 1 & 0 \\
      0 & 0 & -1
    \end{pmatrix},
  \end{equation*}
  fix $t > 0$, and choose $\theta \in \R$ such that $\theta/\pi \notin \Q$. Set
  \begin{equation*}
    B_1 =
    \begin{pmatrix}
      \cosh t & 0 & \sinh t \\
      0 & 1 & 0 \\
      \sinh t & 0 & \cosh t
    \end{pmatrix}
    \qquad \text{and} \qquad
    B_2 =
    \begin{pmatrix}
      \cos \theta & -\sin \theta & 0 \\
      \sin \theta & \cos \theta & 0 \\
      0 & 0 & 1
    \end{pmatrix}.
  \end{equation*}
  Choose $0 < c < 1$ so small that $\|cB_i\| < 1$ for both $i \in \{1,2\}$, and set $A_i = cB_i$. Then $\A = (A_1,A_2)$ is contractive. If $\iii = i_1 \cdots i_n \in \II^n$ and $B_\iii = B_{i_1} \cdots B_{i_n}$, then $A_\iii = c^nB_\iii$. Since $B_1^\top QB_1 = Q$ and $B_2^\top QB_2 = Q$, we have $B_\iii^\top QB_\iii = Q$, and therefore $A_\iii^\top QA_\iii = c^{2n}Q$. Hence every element of $\mathcal{S}(\A)$ belongs to the subset
  \begin{equation*}
    E = \{A \in \GL(3,\R) : A^\top Q A = sQ \text{ for some } s \in \R \setminus \{0\}\}.
  \end{equation*}
  Since $A^\top QA$ is symmetric, the condition $A \in E$ is equivalent to
  \begin{align*}
    (A^\top QA)_{12} &= (A^\top QA)_{13} = (A^\top QA)_{23} = 0, \\
    (A^\top QA)_{11} &= (A^\top QA)_{22} = -(A^\top QA)_{33}.
  \end{align*}
  Indeed, the displayed conditions state that $A^\top QA = sQ$ with $s = (A^\top QA)_{11}$, and $s$ is automatically non-zero, since $A^\top QA$ is invertible for every $A \in \GL(3,\R)$. Thus $E$ is defined by polynomial equations. It is a proper subset of $\GL(3,\R)$ because, for example,
  \begin{equation*}
    \begin{pmatrix}
      1 & 1 & 0 \\
      0 & 1 & 0 \\
      0 & 0 & 1
    \end{pmatrix}
    \notin E.
  \end{equation*}
  Hence $\A$ is not Zariski dense in $\GL(3,\R)$.

  The matrix $A_1$ has eigenvalues $ce^t$, $c$, and $ce^{-t}$, whose moduli are distinct. Thus $A_1$ is pinching, and since we are in dimension three, $\A$ is strongly pinching. To prove strong $1$-irreducibility, suppose that $\VV$ is a finite union of proper, non-trivial subspaces of $\R^3$ such that $A_i\VV = \VV$ for both $i \in \{1,2\}$. Since each $A_i$ is a non-zero scalar multiple of $B_i$, the same union $\VV$ is invariant under $B_1$ and $B_2$. Write $\VV = V_1 \cup \cdots \cup V_m$, where each $V_j$ is maximal among the subspaces contained in $\VV$. Since $B_2$ is invertible and $B_2\VV = \VV$, it permutes the subspaces $V_1,\ldots,V_m$. Hence some power $B_2^n$ fixes each $V_j$. Since $\theta/\pi \notin \Q$, the matrix $B_2^n$ is still an irrational rotation about the $x_3$-axis. The only proper, non-trivial subspaces of $\R^3$ invariant under such a rotation are $\linspan\{e_3\}$ and $\linspan\{e_1,e_2\}$. As $B_1$ likewise permutes the subspaces $V_1,\ldots,V_m$ and preserves their dimensions, it fixes each $V_j$; but $B_1$ preserves neither $\linspan\{e_3\}$ nor $\linspan\{e_1,e_2\}$, since $\sinh t \neq 0$. This contradiction shows that $\A$ is strongly $1$-irreducible.

  To prove strong $2$-irreducibility, suppose that $\VV$ is a finite union of proper, non-trivial subspaces of $\wedge^2\R^3$ such that $A_i^{\wedge 2}\VV = \VV$ for both $i \in \{1,2\}$. Then $\VV$ is also invariant under $B_1^{\wedge 2}$ and $B_2^{\wedge 2}$. The map $u \wedge v \mapsto [w \mapsto \det(u,v,w)]$ is an isomorphism from $\wedge^2\R^3$ to the dual $(\R^3)^*$, and because $\det(B_i) = 1$, it intertwines $B_i^{\wedge 2}$ with the dual action of $B_i$ for $i \in \{1,2\}$. Also, $Q$ identifies $(\R^3)^*$ with $\R^3$, and the identity $B_i^\top QB_i = Q$ shows that this identification commutes with the action of $B_i$ for $i \in \{1,2\}$. Therefore the action of $B_1,B_2$ on $\wedge^2\R^3$ is conjugate to their action on $\R^3$. The image of $\VV$ would thus be a finite union of proper, non-trivial subspaces of $\R^3$ invariant under $B_1$ and $B_2$, contradicting the previous paragraph. Hence $\A$ is strongly $2$-irreducible.
\end{example}

The same algebraic analysis applies in the non-proximal planar case. Here strong $1$-irreducibility forces a rigid structure: up to a linear conjugacy, the tuple consists of contracting similarities whose rotational parts generate an infinite group. This is the structural input used in the non-proximal case of \cref{thm:projection-two-dimensional}.

\begin{lemma} \label{lem:planar-nonproximal-conformal}
  Let $\A = (A_1,\ldots,A_N) \in \GL(2,\R)^N$ satisfy $\|A_i\| < 1$ for all $i \in \{1,\ldots,N\}$. If $\A$ is strongly $1$-irreducible but not $1$-proximal, then there exists $M \in \GL(2,\R)$ such that
  \begin{equation*}
    MA_iM^{-1} = \lambda_iO_i
  \end{equation*}
  for each $i \in \{1,\ldots,N\}$, where $\lambda_i \in \R \setminus \{0\}$ satisfies $0 < |\lambda_i| < 1$ and $O_i \in \GO(2)$. Moreover, if $\mathcal{T}$ is the group generated by $O_1,\ldots,O_N$, then $\mathcal{T}$ is infinite and $\mathcal{T}V$ is dense in $\RP^1$ for every $V \in \RP^1$.
\end{lemma}

\begin{proof}
  Since $\A$ is strongly $1$-irreducible but not $1$-proximal, \cref{prop:zariski-irreducible-converse-planar} shows that $\A$ is not Zariski dense. Let $H$ be the Zariski closure of $\mathcal{S}(\A) \cup \{I\}$ in $\GL(2,\R)$. As in the proof of \cref{prop:zariski-irreducible-converse-planar}, this Zariski closure is an algebraic subgroup of $\GL(2,\R)$. Moreover, $H = \overline{\mathcal{S}(\A)} \cup \{I\}$, so, as $\GL(2,\R)$ is Zariski irreducible and $\{I\}$ is a proper Zariski closed subset, the equality $H = \GL(2,\R)$ would force $\overline{\mathcal{S}(\A)} = \GL(2,\R)$, that is, the Zariski density of $\A$; hence $H \neq \GL(2,\R)$.

  Since $\|A_1\| < 1$, we have $0 < |\det(A_1)| < 1$, so $\{\det(A_1)^n : n \geq 1\}$ is infinite and contained in $\det(H)$; as every proper Zariski closed subset of $\R \setminus \{0\}$ is finite, $\det(H)$ is Zariski dense in $\R \setminus \{0\}$. Let $M_0$ be the Zariski closure of $\Pi(H)$ in $\PGL(2,\R)$, an algebraic subgroup. If $\Pi(H)$ were Zariski dense in $\PGL(2,\R)$, that is, if $M_0 = \PGL(2,\R)$, then \cref{lem:projective-determinant} together with the determinant density just established would give $H = \GL(2,\R)$, contradicting $H \neq \GL(2,\R)$.

  Since each $\Pi(A_i)$ belongs to $\Pi(H) \subseteq M_0$, any non-empty finite $M_0$-invariant subset of $\RP^1$ would, as $M_0$ is a group, be invariant under every $\Pi(A_i)$ and hence give a finite union of lines in $\R^2$ invariant under every $A_i$, contradicting strong $1$-irreducibility; thus $M_0$ preserves no non-empty finite subset of $\RP^1$. By \cref{lem:planar-projective-dichotomy}, either $M_0 = \PGL(2,\R)$, or there exists $g \in \PGL(2,\R)$ such that $gM_0g^{-1} \subseteq \Pi(\GO(2))$ with identity component $\Pi(\SO(2))$. The first case was excluded above, so the second holds.

  Let $M \in \GL(2,\R)$ be a lift of $g$, so that $\Pi(M) = g$. For each $i$ we have $\Pi(MA_iM^{-1}) = g\Pi(A_i)g^{-1} \in gM_0g^{-1} \subseteq \Pi(\GO(2))$, so $MA_iM^{-1}$ is a non-zero scalar multiple of an orthogonal matrix, that is, $MA_iM^{-1} = \lambda_iO_i$ with $\lambda_i \in \R \setminus \{0\}$ and $O_i \in \GO(2)$. Comparing determinants, $\lambda_i^2 = |\det(MA_iM^{-1})| = |\det(A_i)|$, and since $\|A_i\| < 1$ gives $|\det(A_i)| \le \|A_i\|^2 < 1$, we obtain $0 < |\lambda_i| < 1$.

  Finally, let $\mathcal{T}$ be the group generated by $O_1,\ldots,O_N$. If $\mathcal{T}$ were finite, then the orbit of any line under $\mathcal{T}$ would be a finite set of lines; since each $O_i$ lies in $\mathcal{T}$, this orbit is invariant under every $O_i$, and so the corresponding finite union of lines is invariant under every $MA_iM^{-1} = \lambda_iO_i$. Conjugating back by $M^{-1}$ gives a finite union of lines invariant under every $A_i$, contradicting strong $1$-irreducibility. Therefore $\mathcal{T}$ is infinite. Let $\mathcal{T}_0 = \mathcal{T} \cap \SO(2)$. Since the determinant map from $\mathcal{T}$ to $\{1,-1\}$ has kernel $\mathcal{T}_0$, the group $\mathcal{T}_0$ is infinite. The kernel of the action of $\SO(2)$ on $\RP^1$ is $\{I,-I\}$, so the image of $\mathcal{T}_0$ in the circle group $\SO(2)/\{I,-I\}$ is infinite. Its closure is therefore an infinite closed subgroup of the circle, and hence is the whole circle. Thus $\mathcal{T}_0V$ is dense in $\RP^1$ for every $V \in \RP^1$, and the same is true for $\mathcal{T}V$.
\end{proof}

\section{Subsystems} \label{sec:subsystems}

To prove the set projection theorem, \cref{thm:projection}, we apply the measure theorem, \cref{thm:projection-bernoulli}, to a subsystem obtained in two steps: we first pass to a dominated subsystem, and then, using a packing argument, refine it to a strongly separated one that meets the separation hypothesis of the measure theorem. The outcome of the two steps is the following proposition, which is proved in \cref{sec:strongly-separated-subsystems} after the dominated-subsystem construction of \cref{sec:dominated-subsystems}.

\begin{proposition} \label{prop:ssc-approx}
  Let $\Phi = (\fii_1,\ldots,\fii_N)$ be an affine iterated function system such that the associated tuple $\A = (A_1,\ldots,A_N) \in \GL(d,\R)^N$ is strongly pinching and strongly $k$-irreducible for all $k \in \{1,\ldots,d-1\}$, and let $X \subset \R^d$ be the self-affine set associated with $\Phi$. Then for every $\eps > 0$ there exists a finite set $\JJ \subset \II^*$ for which $\A' = (A_\iii)_{\iii \in \JJ}$ is $1$-dominated, and $k$-proximal and strongly $k$-irreducible for all $k \in \{1,\ldots,d-1\}$, satisfying
  \begin{equation*}
    \min\{1, \udimm(X)\} - \eps \le \dimaff(\A') < d,
  \end{equation*}
  and $\Phi' = (\fii_\iii)_{\iii \in \JJ}$ satisfies the strong separation condition.
\end{proposition}

\subsection{Dominated subsystems} \label{sec:dominated-subsystems}

The separation argument of \cref{sec:strongly-separated-subsystems} requires uniform projective control, which domination provides through a common strongly invariant multicone. In this subsection we show that, starting from a proximal strongly irreducible tuple, any finite family of words $\DD \subset \II^*$ can be sandwiched between prefix and suffix words of uniformly bounded length so that the resulting tuple is $1$-dominated. This extends the earlier $2 \times 2$ construction of B\'ar\'any, Jordan, K\"aenm\"aki, and Rams \cite[Lemma 5.2]{BaranyJordanKaenmakiRams2021}. 

A matrix $A \in \GL(d,\R)$ has a \emph{simple dominant eigenvalue} if it has an eigenvalue $\lambda_1$ of algebraic multiplicity one whose modulus strictly exceeds the modulus of every other eigenvalue. In that case $\lambda_1$ is real, as non-real eigenvalues of a real matrix occur in conjugate pairs of equal modulus, and its eigenspace, viewed as an element of $\RP^{d-1}$, is called the \emph{leading eigenspace} of $A$. The unique hyperplane $W \subset \R^d$ such that $W^\bot \in \RP^{d-1}$ is the eigenspace of $A^\top$ corresponding to $\lambda_1$ is called the \emph{repelling hyperplane} of $A$. Throughout, whenever a linear subspace $U \subseteq \R^d$ is intersected with or compared to subsets of $\RP^{d-1}$, it is identified with its projectivization $\{L \in \RP^{d-1} : L \subseteq U\}$; for instance, the conditions $K \cap W = \emptyset$ with $K \subseteq \RP^{d-1}$ and $V \subset W$ with $V \in \RP^{d-1}$ are read through this identification.

\begin{proposition} \label{prop:dominated-subsystem}
  Let $\A = (A_1,\ldots,A_N) \in \GL(d,\R)^N$ be $1$-proximal and strongly $1$-irreducible. For each $\iii \in \II^*$ such that $A_\iii$ has a simple dominant eigenvalue, there exist $K \geq 1$, $\rho > 0$, and functions $\overline{\hhh}, \overline{\kkk} \colon \II^* \to \bigcup_{m=0}^K \II^m$ with the following property: if $V \in \RP^{d-1}$ is the leading eigenspace of $A_\iii$ and $W \subset \R^d$ is the repelling hyperplane of $A_\iii$, then there exists a nonempty compact set $\CC_0 \subset B^o(V,\rho)$ with $V \in \CC_0^o$ such that
  \begin{equation*}
    B(V,\rho) \cap W = \emptyset
  \end{equation*}
  and
  \begin{equation*}
    A_{\overline{\hhh}(\jjj)} A_\jjj A_{\overline{\kkk}(\jjj)} B(V,\rho) \subset \CC_0^o
  \end{equation*}
  for all $\jjj \in \II^*$, where $B(V,\rho)$ and $B^o(V,\rho)$ are the closed and open balls in $\RP^{d-1}$ centered at $V$ with radius $\rho$. In particular, the hyperplane $W$ is transverse to every element of $B(V,\rho)$, and for every finite set $\DD \subset \II^*$ the tuple
  \begin{equation*}
    \A' = (A_{\overline{\hhh}(\jjj)} A_\jjj A_{\overline{\kkk}(\jjj)})_{\jjj \in \DD}
  \end{equation*}
  is $1$-dominated.
\end{proposition}

The next lemma isolates the covering obstruction used in the proof of \cref{prop:dominated-subsystem}: under strong $1$-irreducibility, an invariant subset of projective space cannot be covered by finitely many proper linear subspaces.

\begin{lemma} \label{lem:Z_F-infinite}
  Let $\A = (A_1,\ldots,A_N) \in \GL(d,\R)^N$ be strongly $1$-irreducible. If $\mathcal{F}$ is a collection of proper linear subspaces of $\R^d$ and $\emptyset \ne Z_F \subseteq \RP^{d-1}$ such that $Z_F = \bigcup_{i \in \II} A_i Z_F$ and
  \begin{equation*}
    Z_F \subseteq \bigcup_{V \in \mathcal{F}} V,
  \end{equation*}
  then $\mathcal{F}$ is infinite.
\end{lemma}

\begin{proof}
  Suppose to the contrary that $Z_F$ is contained in a finite union of proper linear subspaces of $\R^d$, viewed as a subset of $\RP^{d-1}$. Like the Zariski topology on $\GL(d,\R)$, the Zariski topology on $\RP^{d-1}$, whose closed sets are the common zero sets of finite collections of homogeneous polynomials, satisfies the descending chain condition, so every non-empty family of Zariski closed subsets of $\RP^{d-1}$ has a minimal element with respect to inclusion. Since such finite unions are Zariski closed subsets of $\RP^{d-1}$, there exists a set $\mathcal{W} \subseteq \RP^{d-1}$ that is minimal with respect to inclusion among the finite unions of proper linear subspaces satisfying $Z_F \subseteq \mathcal{W}$. Fix $i \in \II$. Since $A_iZ_F \subseteq Z_F \subseteq \mathcal{W}$, we have $Z_F \subseteq A_i^{-1}\mathcal{W}$, and hence $Z_F \subseteq \mathcal{W} \cap A_i^{-1}\mathcal{W}$. Now $A_i^{-1}\mathcal{W}$ is again a finite union of proper linear subspaces, and so is $\mathcal{W} \cap A_i^{-1}\mathcal{W}$, since the intersection of two linear subspaces is a linear subspace. By the minimality of $\mathcal{W}$, we obtain $\mathcal{W} = \mathcal{W} \cap A_i^{-1}\mathcal{W}$, and hence
  \begin{equation*}
    A_i\mathcal{W} \subseteq \mathcal{W}
  \end{equation*}
  for all $i \in \II$. Therefore, for each $i \in \II$, the sets $A_i^n\mathcal{W}$, $n \geq 0$, form a descending chain of finite unions of proper linear subspaces. By the descending chain condition, there exists $n_i \geq 0$ such that $A_i^{n_i}\mathcal{W} = A_i^{n_i+1}\mathcal{W}$. Since $A_i$ is invertible, applying $A_i^{-n_i}$ yields $\mathcal{W} = A_i\mathcal{W}$ for all $i \in \II$, which contradicts the strong $1$-irreducibility, since $\mathcal{W} \supseteq Z_F$ is non-empty and is therefore a finite union of proper and non-trivial linear subspaces of $\R^d$. This finishes the proof.
\end{proof}

We also record the following elementary contraction property of a matrix with a simple dominant eigenvalue. It will be used in the proof of \cref{prop:dominated-subsystem} and later in the construction of strongly separated subsystems.

\begin{lemma} \label{lem:pinching-projective}
  Let $A \in \GL(d,\R)$ have a simple dominant eigenvalue $\lambda_1$, let $V \in \RP^{d-1}$ be the eigenspace of $A$ corresponding to $\lambda_1$, and let $W \subset \R^d$ be the repelling hyperplane of $A$. Then $W$ is $A$-invariant and $V \not\subset W$. If $K \subset \RP^{d-1}$ is compact and $K \cap W = \emptyset$, then for every open neighborhood $U$ of $V$ there exists $n_0 \ge 1$ such that
  \begin{equation*}
    A^n K \subset U
  \end{equation*}
  for all $n \ge n_0$. In particular, if $\eps > 0$ and $B(V,\eps) \cap W = \emptyset$, then
  \begin{equation*}
    A^n B(V,\eps) \subset B^o(V,\eps)
  \end{equation*}
  for all $n \ge n_0$.
\end{lemma}

\begin{proof}
  Note that $\lambda_1$ is also a simple dominant eigenvalue of $A^\top$. Let $w^\bot \in W^\bot \setminus \{0\}$. Since
  \begin{equation*}
    \la Aw, w^\bot \ra = \la w, A^\top w^\bot \ra = \lambda_1 \la w, w^\bot \ra = 0
  \end{equation*}
  for all $w \in W$, we see that $AW \subseteq W$.

  We next prove that $V \not\subset W$. Let $v \in V \setminus \{0\}$. Since $Av = \lambda_1 v$, there exist $M \in \GL(d,\R)$, $B \in \GL(d-1,\R)$, and $a \in \R^{d-1}$ such that $M^{-1}(\mathbf{0},x) = v$ for some $x \neq 0$, and
  \begin{equation} \label{eq:block-form}
    MAM^{-1} =
    \begin{pmatrix}
      B & \mathbf{0} \\
      a^\top & \lambda_1
    \end{pmatrix},
  \end{equation}
  where $\mathbf{0} = (0,\ldots,0) \in \R^{d-1}$. Writing $(M^{-1})^\top w^\bot = (u,y)$ with $u \in \R^{d-1}$ and $y \in \R$, the identity $A^\top w^\bot = \lambda_1 w^\bot$ gives
  \begin{equation*}
    (MAM^{-1})^\top (u,y) = \lambda_1 (u,y).
  \end{equation*}
  If $y = 0$, then $(u,y) \neq 0$ implies that $u \neq 0$, and hence $(B^\top u,0) = (\lambda_1 u,0)$. Thus $\lambda_1$ is an eigenvalue of $B^\top$, and therefore also of $B$. By \cref{eq:block-form}, it follows that $\lambda_1$ has algebraic multiplicity at least $2$ as an eigenvalue of $A$, which contradicts the simplicity of $\lambda_1$. Hence $y \neq 0$, and therefore
  \begin{equation*}
    \la v, w^\bot \ra = \la Mv, (M^{-1})^\top w^\bot \ra = \la (\mathbf{0},x), (u,y) \ra = xy \neq 0.
  \end{equation*}
  Thus $v \notin W$, proving that $V \not\subset W$.

  Since $\dim(W) = d-1$ and $v \notin W$, we have $\R^d = \linspan\{v\} \oplus W$. Therefore every element of $\RP^{d-1} \setminus W$ can be written uniquely in the form $\linspan\{v + w\}$ with $w \in W$, and this identifies $\RP^{d-1} \setminus W$ with the affine space $W$. Write $\phi \colon \RP^{d-1} \setminus W \to W$ for the affine chart given by $\phi(\linspan\{v + w\}) = w$. Since $Av = \lambda_1 v$ and $AW \subseteq W$, the induced map in this affine chart is
  \begin{equation*}
    w \mapsto \lambda_1^{-1}Aw.
  \end{equation*}
  Because $W$ is $A$-invariant, the restriction $A|W$ is a linear endomorphism of $W$. After complexifying if necessary, every eigenvalue of $A|W$ is an eigenvalue of $A$. If $\lambda_1$ were an eigenvalue of $A|W$, there would be a non-zero $w \in W$ with $Aw = \lambda_1 w$; since $\lambda_1$ is simple, its eigenspace is $V$, so $\linspan{w} = V \subset W$, contradicting $V \not\subset W$. Therefore every eigenvalue of $A|W$ has modulus strictly smaller than $|\lambda_1|$. It follows that the linear map $T = \lambda_1^{-1}A|W$ has spectral radius strictly smaller than $1$, so $T^n$ converges to $0$ in operator norm. Since $\phi^{-1}$ is continuous and $U$ is an open neighborhood of $V = \phi^{-1}(0)$, there exists an open neighborhood $O \subset W$ of $0$ such that $\phi^{-1}(O) \subset U$. Since $\phi(K)$ is compact in $W$, we have $T^n\phi(K) \subset O$ for all sufficiently large $n$. Therefore $A^n K = \phi^{-1}(T^n\phi(K)) \subset U$ for all sufficiently large $n$. This proves the first claim. The second follows by applying the first claim with $K = B(V,\eps)$ and $U = B^o(V,\eps)$.
\end{proof}

Equipped with \cref{lem:Z_F-infinite,lem:pinching-projective}, we now construct the dominated subsystem.

\begin{proof}[Proof of \cref{prop:dominated-subsystem}]
  Fix $\iii \in \II^*$ such that $A_\iii$ has a simple dominant eigenvalue, and let $\lambda_1$ denote this eigenvalue. The $1$-proximality of $\A$ is assumed only to guarantee the existence of such $\iii$; recall \cref{sec:irred}. Let $V_1 \in \RP^{d-1}$ be the eigenspace of $A_\iii$ corresponding to $\lambda_1$, and let $W$ be the repelling hyperplane of $A_\iii$. By \cref{lem:pinching-projective}, the subspace $W$ is $A_\iii$-invariant and $V_1 \not\subset W$. Since $\A$ is $1$-irreducible, for each $V \in \RP^{d-1}$ there exists $\jjj_V \in \II^*$ such that $A_{\jjj_V} V \not\subset W$. Indeed, if $V \not\subset W$, then we may take $\jjj_V = \varnothing$. Otherwise, if there were $V \subset W$ such that $A_\jjj V \subset W$ for all $\jjj \in \II^*$, then $\linspan\{A_\jjj V : \jjj \in \II^*\} \subseteq W$ would be an invariant proper linear subspace. By continuity, we thus see that for each $V \in \RP^{d-1}$ there is $\eps_V > 0$ such that
  \begin{equation} \label{eq:dominated-subsystem3}
    A_{\jjj_V} B(V,\eps_V) \cap [W]_{\eps_V} = \emptyset,
  \end{equation}
  where the $\eps$-neighborhood $[W]_\eps$ of $W$ in $\RP^{d-1}$ is the closure of $\RP^{d-1} \setminus B(W^\bot,\sqrt{1-\eps^2})$. Equivalently, $[W]_\eps$ is the closed $\eps$-neighborhood of $\{U \in \RP^{d-1} : U \subseteq W\}$ in the metric $d$ on $\RP^{d-1}$ defined in \cref{eq:RP-metric}, so that the bracket notation $[\,\cdot\,]_\eps$ is consistent with the neighborhoods $[\partial Q_i]_\eps$ and $[L^i_j]_\eps$ of boundary subspaces introduced below. In local affine charts, open projective polygons form a basis for the topology of $\RP^{d-1}$. Therefore, for each $V \in \RP^{d-1}$ we may choose an open projective polygon $Q_V$ such that $V \subset Q_V \subset B(V,\eps_V)$. Since $\{Q_V : V \subset W\}$ is an open cover of $\{V \in \RP^{d-1} : V \subset W\}$ which is compact in $\RP^{d-1}$, there are finitely many lines $V_2,\ldots,V_m \in \RP^{d-1}$ such that $W \subset \bigcup_{i=2}^m Q_{V_i}$. To simplify notation, we write $Q_i = Q_{V_i}$ and $\jjj_i = \jjj_{V_i}$ for all $i \in \{2,\ldots,m\}$. It follows that there is $0 < \delta < \min_{i \in \{2,\ldots,m\}} \eps_{V_i}$ such that
  \begin{equation*} 
    [W]_\delta \subset \bigcup_{i = 2}^m Q_{i}
  \end{equation*}
  and, by \cref{eq:dominated-subsystem3},
  \begin{equation} \label{eq:dominated-subsystem5}
    A_{\jjj_{i}}Q_{i} \cap [W]_\delta = \emptyset
  \end{equation}
  for all $i \in \{2,\ldots,m\}$. We also write $\jjj_1 = \varnothing$, in which case $A_{\jjj_1}$ is the identity. Since $[W]_\delta$ is a neighborhood of $W$, the set $K = \overline{\RP^{d-1} \setminus [W]_\delta} \cup \{V_1\}$ is a compact subset of the affine chart $\RP^{d-1} \setminus W$. After identifying this chart with $\R^{d-1}$, we may choose an open Euclidean polytope whose closure is compact and contains $K$. Viewing this polytope in projective coordinates, we obtain a projective polygon $Q_1$ such that $K \subseteq Q_1$ and $\overline{Q_1} \cap W = \emptyset$. Then clearly $\{Q_i\}_{i=1}^m$ covers the whole projective space $\RP^{d-1}$. Since $V_1 \not\subset W$, we may choose numbers $0 < \rho < \eps_1 \le \eps_{V_1}$ so small that $B(V_1,\eps_1) \cap W = \emptyset$. Since $\overline{Q_1}$ is compact and disjoint from $W$, \cref{lem:pinching-projective} implies that, after replacing $\iii$ by a sufficiently long concatenation $\iii\iii\cdots\iii$ in the beginning of the proof if needed, which changes neither the property that $A_\iii$ has a simple dominant eigenvalue nor its leading eigenspace or repelling hyperplane, since $A_\iii^n$ has the simple dominant eigenvalue $\lambda_1^n$, with the same leading eigenspace $V_1$ and repelling hyperplane $W$, we may assume that
  \begin{equation*}
    A_\iii \overline{Q_1} \subseteq B^o(V_1,\rho).
  \end{equation*}
  Therefore, by \cref{eq:dominated-subsystem5}, we have
  \begin{equation} \label{eq:dominated-subsystem6}
    A_{\iii\jjj_i} Q_i \subseteq A_{\iii}(\RP^{d-1} \setminus [W]_\delta) \subseteq A_\iii Q_1 \subseteq B^o(V_1,\rho)
  \end{equation}
  for all $i \in \{2,\ldots,m\}$, while for $i = 1$ we have $A_{\iii\jjj_1} Q_1 = A_\iii Q_1 \subseteq A_\iii \overline{Q_1} \subseteq B^o(V_1,\rho)$ directly. Define
  \begin{equation*}
    \CC_0 = A_\iii \overline{Q_1}.
  \end{equation*}
  Then $\CC_0$ is compact, $\CC_0 \subset B^o(V_1,\rho)$, and $V_1 \in \CC_0^o$. Moreover,
  \begin{equation*}
    A_{\iii\jjj_i} Q_i \subset \CC_0^o
  \end{equation*}
  for all $i \in \{1,\ldots,m\}$. See \cref{fig:Proof} for an illustration in $\RP^2$.

  \definecolor{darkgreen}{rgb}{0.0, 0.4, 0.0}
  \begin{figure}[t]
    \begin{tikzpicture}[scale=3.8, line cap=round, line join=round, >=Latex]

      \def\r{1} 
      \def\ry{0.22} 
      \def\h{0.9} 

      \fill[gray!10] (-1.55,-0.3) -- (1.15,-0.3) -- (1.55,0.3) -- (-1.15,0.3) -- (-1.55,-0.3);

      \node at (-1.3,-0.2) {$W$};

      \draw[dashed, black!60] (0,0) ellipse ({\r} and {\ry});

      \filldraw[thick, draw=darkgreen, fill=darkgreen!30, opacity=0.8] (-\r,0) -- (-\r,0.09) arc (180:355:{\r} and {\ry}) -- (\r,0) arc (360:180:{\r} and {\ry});

      \node[darkgreen, right] at (\r,0) {$[W]_{\delta}$};

      \filldraw[thick, draw=blue!90, fill=blue!30, opacity=0.7] (-0.48\r,-0.19) -- (-0.48\r,-0.08) .. controls ({-0.52*\r}, {-0.0}) and ({-0.56*\r}, {0.05}) .. (-0.6\r, 0.1) .. controls ({-0.67*\r}, {0.12}) and ({-0.73*\r}, {0.12}) .. (-0.8\r,0.12) -- (-0.91\r,0.02) -- (-0.92\r,-0.08) .. controls ({-0.8*\r}, {-0.15}) and ({-0.6*\r}, {-0.18}) .. (-0.48\r,-0.19);

      \node[blue!90] at ({-0.72*\r}, {-0.0}) {$Q_i$};

      \filldraw[thick, draw=blue!90, fill=blue!30, opacity=0.7] (-0.25\r,0.4) -- (-0.24\r,0.48) .. controls ({-0.25*\r}, {0.53}) and ({-0.25*\r}, {0.55}) .. (-0.3\r, 0.6) .. controls ({-0.34*\r}, {0.62}) and ({-0.38*\r}, {0.62}) .. (-0.45\r,0.6) -- (-0.5\r,0.54) .. controls ({-0.52*\r}, {0.51}) and ({-0.52*\r}, {0.48}) .. (-0.5\r,0.45) .. controls ({-0.45*\r}, {0.4}) and ({-0.3*\r}, {0.38}) .. (-0.25\r,0.4);

      \draw[->,thick, blue!90] ({-0.55*\r}, {-0.02}) .. controls ({-0.45*\r}, {0.1}) and ({-0.35*\r}, {0.36}) .. ({-0.35*\r}, {0.48}) node[pos=0.6, right] {$A_{\mathtt{j}_{i}}$};

      \filldraw[thick, draw=blue!90, fill=blue!30, opacity=0.7] (0.5\r,0.5) -- (0.5\r,0.53) .. controls ({0.51*\r}, {0.55}) and ({0.48*\r}, {0.57}) .. (0.45\r, 0.57) .. controls ({0.42*\r}, {0.58}) and ({0.38*\r}, {0.57}) .. (0.35\r,0.55) -- (0.34\r,0.52) .. controls ({0.34*\r}, {0.5}) and ({0.35*\r}, {0.48}) .. (0.36\r,0.46) .. controls ({0.4*\r}, {0.46}) and ({0.45*\r}, {0.46}) .. (0.5\r,0.5);

      \draw[->,thick, blue!90] ({-0.3*\r}, {0.5}) .. controls ({-0.15*\r}, {0.55}) and ({0.3*\r}, {0.55}) .. ({0.43*\r}, {0.52}) node[pos=0.65, below left] {$A_{\mathtt{i}}$};

      \draw[thick,domain=0:180,smooth,variable=\t]
      plot ({\r*cos(\t)}, {\h*sin(\t)});

      \coordinate (O) at (0,0);

      \coordinate (Wperp) at (0,\h);

      \coordinate (V1) at ({0.55*\r}, {0.55*\h});

      \coordinate (V2) at ({0.8*\r}, {-0.13});

      \coordinate (Vm) at ({0.3*\r}, {-0.21});

      \coordinate (Vi) at ({-0.8*\r}, {-0.13});

      \draw[->,thick] (O) -- (Wperp) node[pos=0.98, above=2] {$W^{\perp}$};

      \draw[->,thick] (O) -- (V1) node[pos=0.5, above] {$V_1$\phantom{x}};

      \draw[->,thick] (O) -- (V2) node[pos=0.95,below right] {$V_2$};

      \draw[->,thick] (O) -- (Vm) node[pos=0.98,below] {$V_m$};

      \draw[->,thick] (O) -- (Vi) node[pos=0.93, below left] {$V_i$};

      \node[red] at (-0.7,0.22) {$B(V_i, \varepsilon_{V_i})$};
      \draw[red, thick]
      ({-0.95*\r}, {-0.06}) .. controls ({-0.88*\r}, {0.25}) and ({-0.5*\r}, {0.22}) .. ({-0.43*\r}, {-0.2}) ;
      \draw[red, thick]
      ({-0.55*\r}, {-0.18}) .. controls ({-0.43*\r}, {0.25}) and ({0.05*\r}, {0.23}) .. ({0.15*\r}, {-0.22}) ;
      \draw[red, thick]
      ({-0.05*\r}, {-0.22}) .. controls ({0.08*\r}, {0.2}) and ({0.47*\r}, {0.2}) .. ({0.6*\r}, {-0.17}) ;
      \draw[red, thick]
      ({0.5*\r}, {-0.19}) .. controls ({0.6*\r}, {0.3}) and ({0.84*\r}, {0.25}) .. ({0.94*\r}, {-0.07}) ;
      \draw[red, thick]
      ({0.88*\r}, {-0.1}) .. controls ({0.88*\r}, {0.2}) and ({0.95*\r}, {0.22}) .. ({0.97*\r}, {0.2}) ;
      \draw[red, thick]
      ({-0.98*\r}, {0.15}) .. controls ({-0.96*\r}, {0.23}) and ({-0.92*\r}, {0.23}) .. ({-0.9*\r}, {-0.09}) ;
      \begin{scope}[xshift=14.6,yshift=13.8]
        \draw[red, thick, rotate=60] (0, 0) ellipse (0.15 and 0.2);
      \end{scope}
      \node[red] at (0.35,0.71) {$B(V_1, \varepsilon_{1})$\phantom{x}};

      \draw[thick] (-\r,0) arc (180:360:{\r} and {\ry});
    \end{tikzpicture}
    \caption{Illustration for the proof of \cref{prop:dominated-subsystem} in $\RP^2$.}\label{fig:Proof}
  \end{figure}

  Since the sets $Q_i$ are projective polygons, their sides are formed by $(d-1)$-dimensional linear subspaces $L^i_1,\ldots,L^i_{M_i}$ of $\R^d$. The $\eps$-neighborhood of the boundary subspaces of $Q_i$ is
  \begin{equation*} 
    [\partial Q_i]_\eps \subseteq \bigcup_{j = 1}^{M_i} [L^i_j]_\eps.
  \end{equation*}
  If $D \in \GL(d,\R)$, then we write
  \begin{equation*} 
    Q_i^D = D^{-1}Q_i \cap B(V_1,\rho)
  \end{equation*}
  for all $i \in \{1,\ldots,m\}$. Since the sets $D^{-1}Q_i$ cover $\RP^{d-1}$ and are open, the sets $Q_i^D$ cover $B(V_1,\rho)$ and are open as subsets of $B(V_1,\rho)$. By \cref{eq:dominated-subsystem6},
  \begin{equation} \label{eq:dominated-subsystem9}
    A_{\iii\jjj_i}DQ_i^D \subseteq A_{\iii\jjj_i}Q_i \subseteq B^o(V_1,\rho)
  \end{equation}
  for all $i \in \{1,\ldots,m\}$.

  Define
  \begin{equation*}
    Z_F = \overline{\{A_\jjj V_1 : \jjj \in \II^*\}}
  \end{equation*}
  and notice that $Z_F = \bigcup_{i \in \II} A_i Z_F$. Indeed, since $A_\ppp Z_F \subseteq Z_F$ for every word $\ppp \in \II^*$, we have $\bigcup_{i \in \II} A_i Z_F \subseteq Z_F$. Conversely, every point of $\{A_\jjj V_1 : \jjj \in \II^*\}$ belongs to $\bigcup_{i \in \II} A_i Z_F$. If $\jjj \ne \varnothing$, then $\jjj = i\ppp$ for some $i \in \II$ and $\ppp \in \II^*$, so $A_\jjj V_1 = A_i(A_\ppp V_1) \in A_i Z_F$. If $\jjj = \varnothing$, then $A_\jjj$ is the identity by convention. Writing $\iii = i\ppp$ with $i \in \II$ and $\ppp \in \II^*$, we obtain $A_\jjj V_1 = V_1 = A_\iii V_1 = A_i(A_\ppp V_1) \in A_i Z_F$. Since $\bigcup_{i \in \II} A_i Z_F$ is closed, the claimed equality follows. 
  
  Let us next show that there exists $\eta > 0$ such that for every $D \in \GL(d,\R)$ there is $i \in \{1,\ldots,m\}$ for which
  \begin{equation} \label{eq:dominated-subsystem10}
    Z_F \cap Q_i^D \setminus [\partial D^{-1}Q_i]_\eta \ne \emptyset.
  \end{equation}
  Here $[\partial D^{-1}Q_i]_\eta$ denotes the $\eta$-neighborhood of $\partial D^{-1}Q_i$ in $\RP^{d-1}$. Suppose this is not the case. Then for every $\eta > 0$ there is $D_\eta \in \GL(d,\R)$ such that
  \begin{align*} 
    Z_F \cap B(V_1,\rho) &\subseteq \bigcup_{i=1}^m Q_i^{D_\eta} \cap [\partial D_\eta^{-1}Q_i]_\eta \\
    &\subseteq \bigcup_{i=1}^m \bigcup_{j=1}^{M_i} Q_i^{D_\eta} \cap [D_\eta^{-1}L_j^i]_\eta \subseteq \bigcup_{i=1}^m \bigcup_{j=1}^{M_i} [D_\eta^{-1}L_j^i]_\eta.
  \end{align*}
  By letting $\eta \downarrow 0$ along a countable sequence and relying on the compactness of the Grassmannian manifold, we obtain a convergent subsequence along which $D_\eta^{-1}L_j^i \to \hat L_j^i$ for all $i \in \{1,\ldots,m\}$ and $j \in \{1,\ldots,M_i\}$. Fix $V \in Z_F \cap B(V_1,\rho)$. For every $\eta$ in this subsequence there are $i_\eta \in \{1,\ldots,m\}$ and $j_\eta \in \{1,\ldots,M_{i_\eta}\}$ such that $V \in [D_\eta^{-1}L_{j_\eta}^{i_\eta}]_\eta$. Passing to a further subsequence, we may assume that $i_\eta = i$ and $j_\eta = j$ are constant. Since $D_\eta^{-1}L_j^i \to \hat L_j^i$ in the Grassmannian and $\eta \downarrow 0$, it follows that $\dist(V,\hat L_j^i) = 0$, and hence $V \in \hat L_j^i$. Therefore
  \begin{equation} \label{eq:dominated-subsystem12}
    Z_F \cap B(V_1,\rho) \subseteq \bigcup_{i = 1}^m \bigcup_{j = 1}^{M_i} \hat L_j^i.
  \end{equation}
  Fix $V \in Z_F$. Since $\RP^{d-1} \subseteq \bigcup_{i=1}^m Q_i$, there exists $i \in \{1,\ldots,m\}$ such that $V \in Q_i$. Since $A_\ppp Z_F \subseteq Z_F$ for every word $\ppp \in \II^*$ and $V \in Z_F$, we have $A_{\iii\jjj_i} V \in Z_F$. From \cref{eq:dominated-subsystem6} it follows that $A_{\iii\jjj_i} V \in B^o(V_1,\rho)$ and hence, by \cref{eq:dominated-subsystem12}, that $A_{\iii\jjj_i} V \in \bigcup_{i=1}^m \bigcup_{j=1}^{M_i} \hat L_j^i$. Therefore
  \begin{equation*} 
    Z_F \subseteq \bigcup_{k = 1}^m A_{\iii\jjj_k}^{-1} \biggl( \bigcup_{i = 1}^m \bigcup_{j = 1}^{M_i} \hat L_j^i \biggr).
  \end{equation*}
  Set $\mathcal{F} = \{A_{\iii\jjj_k}^{-1}\hat L_j^i : k,i \in \{1,\ldots,m\}$ and $j \in \{1,\ldots,M_i\}\}$. Since each $A_{\iii\jjj_k}$ is invertible, every element of $\mathcal{F}$ is a proper linear subspace, and the previous inclusion shows that
  \begin{equation*}
    Z_F \subseteq \bigcup_{U \in \mathcal{F}} U.
  \end{equation*}
  This contradicts \cref{lem:Z_F-infinite} and establishes \cref{eq:dominated-subsystem10}.

  Since $B(V_1,\rho)$ is compact, there are points $U_1,\ldots,U_M \in B(V_1,\rho)$ such that
  \begin{equation*}
    B(V_1,\rho) \subseteq \bigcup_{\ell = 1}^M B^o(U_\ell,\eta/2).
  \end{equation*}
  For each $\ell \in \{1,\ldots,M\}$, write $G_\ell = B^o(U_\ell,\eta/2)$ and set $\mathcal{G} = \{G_\ell : \ell \in \{1,\ldots,M\}\}$. If $G \in \mathcal{G}$ and $V \in G$, then $G \subseteq B^o(V,\eta)$ by the triangle inequality. Write $\mathcal{G}_F = \{G \in \mathcal{G} : G \cap Z_F \ne \emptyset\}$. For each $G \in \mathcal{G}_F$, the set $G$ is open and intersects $Z_F = \overline{\{A_\jjj V_1 : \jjj \in \II^*\}}$, so there exists $\jjj_G \in \II^*$ such that $A_{\jjj_G}V_1 \in G$. Applying \cref{lem:pinching-projective} to the compact set $B(V_1,\rho)$ and the open neighborhood $A_{\jjj_G}^{-1}G$ of $V_1$, we obtain $n_G \ge 1$ such that the word $\iii_G = \iii\iii\cdots\iii$, where $\iii$ is concatenated $n_G$ times, satisfies
  \begin{equation} \label{eq:dominated-subsystem14}
    A_{\jjj_G\iii_G} B(V_1,\rho) \subseteq G.
  \end{equation}
  By \cref{eq:dominated-subsystem10}, for every $D \in \GL(d,\R)$ there are $i \in \{1,\ldots,m\}$ and $V \in Z_F \cap Q_i^D \setminus [\partial D^{-1}Q_i]_\eta$. Since $V \notin [\partial D^{-1}Q_i]_\eta$, the open ball $B^o(V,\eta)$ in the metric $d$ of \cref{eq:RP-metric} does not meet $\partial D^{-1}Q_i$; being connected and containing the point $V \in Q_i^D \subseteq D^{-1}Q_i$, it therefore satisfies
  \begin{equation*}
    B^o(V,\eta) \subset D^{-1}Q_i.
  \end{equation*}
  Choose $G \in \mathcal{G}$ so that $V \in G$. Since $G \subseteq B^o(V,\eta)$, we have $G \subseteq D^{-1}Q_i$. As $V \in Z_F \cap G$, we have $G \in \mathcal{G}_F$. Therefore, by \cref{eq:dominated-subsystem14} and \cref{eq:dominated-subsystem9},
  \begin{equation*}
    A_{\iii\jjj_i}DA_{\jjj_G\iii_G} B(V_1,\rho) \subset A_{\iii\jjj_i}DG \subset A_{\iii\jjj_i}Q_i \subset \CC_0^o.
  \end{equation*}
  Since $B(V_1,\rho) \subset B(V_1,\eps_1)$ and $B(V_1,\eps_1) \cap W = \emptyset$, no element of $B(V_1,\rho)$ is contained in $W$, and hence the hyperplane $W$ is transverse to every element of $B(V_1,\rho)$. Define $\KK = \{\iii\jjj_i : i \in \{1,\ldots,m\}\} \cup \{\jjj_G\iii_G : G \in \mathcal{G}_F\}$. Since $\mathcal{G}_F$ is finite, the set $\KK$ is finite. Let $K = \max_{\kkk \in \KK} |\kkk|$. For each $\jjj \in \II^*$, choose words $\overline{\hhh}(\jjj), \overline{\kkk}(\jjj) \in \KK$ such that
  \begin{equation*}
    A_{\overline{\hhh}(\jjj)} A_\jjj A_{\overline{\kkk}(\jjj)} B(V_1,\rho) \subset \CC_0^o.
  \end{equation*}
  Consequently, for every finite set $\DD \subset \II^*$ the ball $B(V_1,\rho)$ is a strongly invariant multicone for the tuple $\A' = (A_{\overline{\hhh}(\jjj)} A_\jjj A_{\overline{\kkk}(\jjj)})_{\jjj \in \DD}$, and hence \cite[Theorem~B]{BochiGourmelon2009} gives the $1$-dominated property. This proves the proposition.
\end{proof}

\subsection{Strongly separated subsystems} \label{sec:strongly-separated-subsystems}

We now prove \cref{prop:ssc-approx}. The dominated-subsystem machinery gives uniform projective control after bounded prefix and suffix words, the packing argument supplies enough disjoint cylinders to keep the affinity dimension close to $\udimm(X)$, and the strong pinching provides a single word whose associated matrix is pinching in every exterior power, for use in the finite replacement step to preserve proximality and force strong irreducibility.

The next lemma applies \cref{prop:dominated-subsystem} to a prescribed word whose associated matrix has a simple dominant eigenvalue in order to produce a strongly separated dominated subsystem with quantitative size control. Roughly speaking, \cref{lem:ssc-dominated} says that, at arbitrarily small scales $r$, one can find at least a constant multiple of $r^{-\udimm(X)+\eps}$ many cylinder maps of comparable size whose images are pairwise disjoint and whose linear parts remain $1$-dominated.

\begin{lemma} \label{lem:ssc-dominated}
  Let $\Phi = (\fii_1,\ldots,\fii_N)$ be an affine iterated function system such that the associated tuple $\A = (A_1,\ldots,A_N) \in \GL(d,\R)^N$ of matrices is $1$-proximal and strongly $1$-irreducible. If $\ppp \in \II^*$ is such that $A_\ppp$ has a simple dominant eigenvalue, if $V_1 \in \RP^{d-1}$ is the leading eigenspace of $A_\ppp$, and if $W \subset \R^d$ is the repelling hyperplane of $A_\ppp$, then there exist $K \ge 1$, $\rho > 0$, a nonempty compact set $\CC_0 \subset B^o(V_1,\rho)$ with $V_1 \in \CC_0^o$, functions $\overline{\hhh}, \overline{\kkk} \colon \II^* \to \bigcup_{m=0}^K \II^m$, and constants $C_\ppp > 0$ and $c_\ppp > 0$ such that
  \begin{equation*}
    B(V_1,\rho) \cap W = \emptyset
  \end{equation*}
  and
  \begin{equation*}
    A_{\overline{\hhh}(\iii)} A_\iii A_{\overline{\kkk}(\iii)} B(V_1,\rho) \subset \CC_0^o
  \end{equation*}
  for all $\iii \in \II^*$. Moreover, for every $\eps > 0$ and every $0 < \tau_0 < \min\{\eps,\frac{1}{e}\}$ there exist $0<r<\tau_0$ and a finite set $\DD \subset \II^*$ such that
  \begin{equation*}
    \#\DD \ge C_\ppp r^{-\udimm(X)+\eps}
  \end{equation*}
  and $\|A_\jjj\| \ge c_\ppp r$ for all $\jjj \in \DD$, the compact sets $\fii_\jjj(X)$, $\jjj \in \DD$, are pairwise disjoint, and
  \begin{equation*}
    A_\jjj B(V_1,\rho) \subset \CC_0^o
  \end{equation*}
  for all $\jjj \in \DD$. In particular, the iterated function system $\Phi' = (\fii_\jjj)_{\jjj \in \DD}$ satisfies the strong separation condition and the tuple $\A' = (A_\jjj)_{\jjj \in \DD}$ is $1$-dominated.
\end{lemma}

\begin{proof}
  By \cref{prop:dominated-subsystem}, there exist $K \ge 1$, $\rho > 0$, a nonempty compact set $\CC_0 \subset B^o(V_1,\rho)$ with $V_1 \in \CC_0^o$, and functions $\overline{\hhh}, \overline{\kkk} \colon \II^* \to \bigcup_{m=0}^K \II^m$ such that $B(V_1,\rho) \cap W = \emptyset$ and
  \begin{equation*}
    A_{\overline{\hhh}(\iii)} A_\iii A_{\overline{\kkk}(\iii)} B(V_1,\rho) \subset \CC_0^o
  \end{equation*}
  for all $\iii \in \II^*$. Set
  \begin{equation*}
    C_\ppp = \biggl(\sum_{m = 0}^K \#\II^m\biggr)^{-2} \qquad \text{and} \qquad c_\ppp = \Bigl(\min_{i \in \{1,\ldots,N\}}\alpha_d(A_i)\Bigr)^{2K+1} \diam(X)^{-1}.
  \end{equation*}
  Let $\eps > 0$ and $0 < \tau_0 < \min\{\eps,\frac{1}{e}\}$. By the definition of upper Minkowski dimension, there exists $0 < r < \min\{\tau_0,\diam(X)\}$ such that
  \begin{equation} \label{eq:ssc-dominated}
    N_r(X) > r^{-\udimm(X)+\eps}.
  \end{equation}
  Let $\{B_j\}_{j=1}^{N_r(X)}$ be a maximal $r$-packing of $X$. For each $j \in \{1,\ldots,N_r(X)\}$, let $x_j \in X$ be the center of $B_j$. Choose $\iii^{(j)} \in \II^\N$ such that $\pi(\iii^{(j)}) = x_j$. Since $\diam(\fii_{\iii^{(j)}|_n}(X)) \to 0$ as $n \to \infty$, there exists a least integer $n_j \ge 1$ for which $\diam(\fii_{\iii^{(j)}|_{n_j}}(X)) < r$. Write $\qqq_j = \iii^{(j)}|_{n_j}$. Then $x_j \in \fii_{\qqq_j}(X)$, and since $x_j$ is the center of $B_j$ and $\diam(\fii_{\qqq_j}(X)) < r$, we have
  \begin{equation} \label{eq:ssc-dominated0}
    x_j \in \fii_{\qqq_j}(X) \subseteq B_j,
  \end{equation}
  and, by the minimality of $n_j$ and the choice of $r < \diam(X)$,
  \begin{equation} \label{eq:ssc-dominated1}
    \diam(\fii_{\qqq_j}(X)) < r \le \diam(\fii_{\qqq_j^-}(X)).
  \end{equation}
  Let $\PP_r = \{\qqq_j \in \II^* : j \in \{1,\ldots,N_r(X)\}\}$ be the collection of such finite words and note that, since $\{B_j\}_{j=1}^{N_r(X)}$ is an $r$-packing of $X$, \cref{eq:ssc-dominated0} gives
  \begin{equation} \label{eq:ssc-dominated2}
    \fii_\iii(X) \cap \fii_\jjj(X) = \emptyset
  \end{equation}
  for all $\iii, \jjj \in \PP_r$ with $\iii \ne \jjj$. In particular, if $\qqq_j = \qqq_{j'}$, then \cref{eq:ssc-dominated0} implies that $x_j$ and $x_{j'}$ belong to the same set $\fii_{\qqq_j}(X)$, which is contained in both $B_j$ and $B_{j'}$. Since the packing balls are pairwise disjoint, this is possible only when $j = j'$, and hence, $\#\PP_r = N_r(X)$.

  Write $\PP_r(\hhh,\kkk) = \{\qqq \in \PP_r : \overline{\hhh}(\qqq) = \hhh \text{ and } \overline{\kkk}(\qqq) = \kkk\}$ for all $\hhh, \kkk \in \bigcup_{m=0}^K \II^m$ and observe that there are $\hhh, \kkk \in \bigcup_{m=0}^K \II^m$ such that
  \begin{equation} \label{eq:ssc-dominated3}
    \#\PP_r(\hhh,\kkk) \ge C_\ppp\#\PP_r = C_\ppp N_r(X).
  \end{equation}
  Defining $\DD_r = \{\hhh\qqq\kkk \in \II^* : \qqq \in \PP_r(\hhh,\kkk)\}$, the map $\qqq \mapsto \hhh\qqq\kkk$ is injective, and therefore $\#\DD_r = \#\PP_r(\hhh,\kkk)$. Hence, by \cref{eq:ssc-dominated3} and \cref{eq:ssc-dominated}, the desired cardinality estimate holds. Since $\diam(\fii_\iii(X)) \le \|A_\iii\|\diam(X)$ for all $\iii \in \II^*$, \cref{eq:ssc-dominated1} implies that
  \begin{equation*}
    \|A_{\qqq^-}\| \ge \diam(X)^{-1} r
  \end{equation*}
  for all $\qqq \in \PP_r$. Moreover, since $\|A\| = \|ABB^{-1}\| \le \|AB\|\|B^{-1}\|$, we have $\|AB\| \ge \|A\|\|B^{-1}\|^{-1} = \|A\|\alpha_d(B)$ for all $A,B \in \GL(d,\R)$. Applying this to transposes also gives $\|AB\| = \|B^\top A^\top\| \ge \|B^\top\|\alpha_d(A^\top) = \|B\|\alpha_d(A)$ for all $A,B \in \GL(d,\R)$. Write $\alpha = \min_{i \in \{1,\ldots,N\}}\alpha_d(A_i)$, so that $c_\ppp = \alpha^{2K+1}\diam(X)^{-1}$; note that $\alpha < 1$, since $\alpha_d(A_i) \le \|A_i\| < 1$ for every $i$. By supermultiplicativity of $\alpha_d$, together with $\alpha < 1$ and $|\hhh|,|\kkk| \le K$, we have $\alpha_d(A_\hhh)\alpha_d(A_\kkk) \ge \alpha^{|\hhh|+|\kkk|} \ge \alpha^{2K}$; and, writing $\qqq = \qqq^- q$ with $q \in \II$ the last letter of $\qqq$, we have $\|A_\qqq\| = \|A_{\qqq^-}A_q\| \ge \|A_{\qqq^-}\|\alpha_d(A_q) \ge \alpha\|A_{\qqq^-}\|$. Therefore, each $\jjj \in \DD_r$ satisfies
  \begin{align*}
    \|A_\jjj\| = \|A_\hhh A_\qqq A_\kkk\| &\ge \alpha_d(A_\hhh)\|A_\qqq\|\alpha_d(A_\kkk) \ge \alpha^{2K}\|A_\qqq\| \\
    &\ge \alpha^{2K+1}\|A_{\qqq^-}\| \ge c_\ppp r.
  \end{align*}
  If $\jjj_1 = \hhh\qqq_1\kkk$ and $\jjj_2 = \hhh\qqq_2\kkk$ are distinct words in $\DD_r$, then $\qqq_1 \ne \qqq_2$, and therefore \cref{eq:ssc-dominated2} gives $\fii_{\qqq_1}(X) \cap \fii_{\qqq_2}(X) = \emptyset$. Since $\fii_\kkk(X) \subset X$ and $\fii_\hhh$ is injective, it follows that
  \begin{align*}
    \fii_{\jjj_1}(X) &= \fii_\hhh(\fii_{\qqq_1}(\fii_\kkk(X))) \subset \fii_\hhh(\fii_{\qqq_1}(X)), \\
    \fii_{\jjj_2}(X) &= \fii_\hhh(\fii_{\qqq_2}(\fii_\kkk(X))) \subset \fii_\hhh(\fii_{\qqq_2}(X)),
  \end{align*}
  and hence, $\fii_{\jjj_1}(X) \cap \fii_{\jjj_2}(X) = \emptyset$. Let $X' \subseteq X$ be the self-affine set of $\Phi' = (\fii_\jjj)_{\jjj \in \DD_r}$. Since $\fii_{\jjj_1}(X') \subseteq \fii_{\jjj_1}(X)$ and $\fii_{\jjj_2}(X') \subseteq \fii_{\jjj_2}(X)$, it follows that $\fii_{\jjj_1}(X') \cap \fii_{\jjj_2}(X') = \emptyset$. Thus the iterated function system $\Phi'$ satisfies the strong separation condition. Finally, every $\jjj \in \DD_r$ has the form $\jjj = \overline{\hhh}(\qqq)\qqq\overline{\kkk}(\qqq)$ for some $\qqq \in \PP_r$, and therefore
  \begin{equation*}
    A_\jjj B(V_1,\rho) \subset \CC_0^o \subset B^o(V_1,\rho)
  \end{equation*}
  for all $\jjj \in \DD_r$. Since $B(V_1,\rho) \cap W = \emptyset$, the hyperplane $W$ is transverse to every element of $B(V_1,\rho)$. Hence $B(V_1,\rho)$ is a strongly invariant multicone for $\A'$, and therefore $\A' = (A_\jjj)_{\jjj \in \DD_r}$ is $1$-dominated by \cite[Theorem~B]{BochiGourmelon2009}. This proves the lemma.
\end{proof}

With \cref{lem:ssc-dominated} at hand, we complete the proof of \cref{prop:ssc-approx} by refining the construction to enforce $k$-proximality and strong $k$-irreducibility for all $k$ while keeping the dimensional bound.

\begin{proof}[Proof of \cref{prop:ssc-approx}]
  Since $\A$ is strongly pinching, there is $\ppp \in \II^*$ such that $A_\ppp^{\land k}$ is $1$-pinching for all $k \in \{1,\ldots,d-1\}$. In particular, $A_\ppp$ has a simple dominant eigenvalue, and hence $\A$ is $1$-proximal. Let $V_1 \in \RP^{d-1}$ be the leading eigenspace of $A_\ppp$ and let $W \subset \R^d$ be the repelling hyperplane of $A_\ppp$. Since $A_\ppp$ is diagonalizable, the subspace $W$ is spanned by the remaining eigenspaces of $A_\ppp$. Since $\A$ is $1$-proximal and strongly $1$-irreducible, \cref{lem:ssc-dominated} yields $K \ge 1$, $\rho > 0$, a nonempty compact set $\CC_0 \subset B^o(V_1,\rho)$ with $V_1 \in \CC_0^o$, functions $\overline{\hhh}, \overline{\kkk} \colon \II^* \to \bigcup_{m=0}^K \II^m$, and constants $C_\ppp > 0$ and $c_\ppp > 0$ such that
  \begin{equation} \label{eq:dominated-ball-inclusion}
    A_{\overline{\hhh}(\iii)} A_\iii A_{\overline{\kkk}(\iii)} B(V_1,\rho) \subset \CC_0^o
  \end{equation}
  for all $\iii \in \II^*$, and $B(V_1,\rho) \cap W = \emptyset$. By \cref{lem:domination-pressure}, applied with $\CC = B(V_1,\rho)$ and $\CC_0$, there exists a constant $0 < \kappa_\rho \le 1$, depending only on $B(V_1,\rho)$ and $\CC_0$, with the following property: if $\B = (B_1,\ldots,B_M) \in \GL(d,\R)^M$ is a finite tuple satisfying
  \begin{equation*}
    B_i B(V_1,\rho) \subset \CC_0
  \end{equation*}
  for all $i \in \{1,\ldots,M\}$, then
  \begin{equation} \label{eq:pressure-lower-bound}
    P(\B,s) \ge \log\kappa_\rho^{2s} + \log\sum_{i \in \{1,\ldots,M\}} \|B_i\|^s
  \end{equation}
  for all $0 \le s \le 1$. Since $X$ is not a singleton by our standing assumption, \cite[Corollary 1.1]{XieYinSun2003} gives $\dimh(X) > 0$. Therefore $\udimm(X) \ge \dimh(X) > 0$. Fix $\eps > 0$. If $\eps \ge \min\{1,\udimm(X)\}$, then proving the claim for any $0 < \eps' < \min\{1,\udimm(X)\}$ yields the desired conclusion for $\eps$ as well. Therefore we may assume $\eps < \min\{1,\udimm(X)\}$. Fix $0 < \delta < \eps$ and an exponent $s_0$ with
  \begin{equation*}
    \min\{1,\udimm(X)\} - \eps < s_0 < \min\{1,\udimm(X)\} - \delta.
  \end{equation*}
  The construction below produces a strongly separated, $1$-dominated subsystem $\A'$ with $P(\A',s_0) > 0$; since $\dimaff(\A')$ is the zero of the strictly decreasing pressure $s\mapsto P(\A',s)$, this will give $\dimaff(\A') > s_0 > \min\{1,\udimm(X)\} - \eps$.

  Let $v_1^k,\ldots,v_{\binom{d}{k}}^k$ be the linearly independent eigenvectors of $A_{\ppp}^{\land k}$ and
  \begin{align*}
    \VV_k = \{\linspan\{v_{i_1}^k,\ldots,v_{i_\ell}^k\} \subset \wedge^k\R^d : \; &\{i_1,\ldots,i_\ell\} \subset \{1,\ldots,\tbinom{d}{k}\} \\
    &\text{for some }\ell \in \{1,\ldots,\tbinom{d}{k}-1\}\}
  \end{align*}
  be the collection of all non-trivial proper linear subspaces spanned by the eigenvectors. Observe that for each non-trivial proper linear subspace $V \notin \VV_k$ we have
  \begin{equation} \label{eq:V-images-not-same}
    (A_\ppp^{\land k})^m V \ne (A_\ppp^{\land k})^{m'} V
  \end{equation}
  whenever $m \ne m'$. Indeed, if $(A_\ppp^{\land k})^m V = (A_\ppp^{\land k})^{m'} V$ for some $m \ne m'$, then $V$ is invariant under $(A_\ppp^{\land k})^{|m-m'|}$. Since $A_\ppp^{\land k}$ is $1$-pinching, the map $A_\ppp^{\land k}$ is diagonalizable and its eigenvalues have pairwise distinct absolute values. It follows that every invariant subspace of $(A_\ppp^{\land k})^{|m-m'|}$ is spanned by eigenvectors of $A_\ppp^{\land k}$, contradicting the assumption that $V \notin \VV_k$. Let
  \begin{equation*}
    M = \sum_{k=1}^{d-1} \#\VV_k.
  \end{equation*}
  Since each $\VV_k$ consists of subspaces of the exterior power $\wedge^k\R^d$, the collections $\VV_1,\ldots,\VV_{d-1}$ are pairwise disjoint, so the union $\bigcup_{k=1}^{d-1}\VV_k$ has exactly $M$ elements, and each of its members $V$ lies in a unique $\VV_k$; in the conditions imposed on such a $V$ below, $k$ always denotes this uniquely determined index. Choose $0 < \tau_0 < \min\{\delta,\tfrac{1}{e}\}$ so small that $C_\ppp\tau_0^{-\udimm(X)+\delta} > M+1$. Since $s_0 - \udimm(X) + \delta < 0$, we may also assume that
  \begin{equation*}
    \log\kappa_\rho^{2s_0} + \log(C_\ppp r^{-\udimm(X)+\delta}-M-1) + \log c_\ppp^{s_0}r^{s_0} > 0
  \end{equation*}
  for all $0 < r < \tau_0$. Applying \cref{lem:ssc-dominated} with the parameters $\delta$ and $\tau_0$, we obtain a number $0 < r < \tau_0$ and a finite set $\DD' \subset \II^*$ such that
  \begin{equation}
    \#\DD' \ge C_\ppp r^{-\udimm(X)+\delta},
  \end{equation}
  $\|A_\iii\| \ge c_\ppp r$ for all $\iii \in \DD'$, the compact sets $\fii_\iii(X)$, $\iii \in \DD'$, are pairwise disjoint, and
  \begin{equation*}
    A_\iii B(V_1,\rho) \subset \CC_0^o
  \end{equation*}
  for all $\iii \in \DD'$. Since $r < \tau_0$, the choice of $\tau_0$ gives $\#\DD' > M+1$, and hence $\#\DD' \ge M+2$. Moreover, every $\iii \in \DD'$ satisfies $A_\iii B(V_1,\rho) \subset \CC_0^o$, and $B(V_1,\rho) \cap W = \emptyset$. Therefore $(A_\iii)_{\iii \in \DD'}$ is $1$-dominated by the multicone characterization. Since the family $\{\fii_\iii(X) : \iii \in \DD'\}$ is finite and consists of pairwise disjoint compact sets, write
  \begin{equation*}
    \eta = \min\{\dist(\fii_\iii(X),\fii_\jjj(X)) : \iii, \jjj \in \DD' \text{ with } \iii \ne \jjj\} > 0.
  \end{equation*}
  By \cref{lem:pinching-projective}, applied to the compact set $B(V_1,\rho)$ and the open neighborhood $\CC_0^o$ of $V_1$, there exists $n_0 \ge 1$ such that
  \begin{equation*}
    A_{\ppp^n} B(V_1,\rho) \subset \CC_0^o
  \end{equation*}
  for all $n \ge n_0$. Since $A_\ppp$ is contractive, $\|A_{\ppp^n}\| \to 0$ as $n \to \infty$, so by increasing $n_0$ if necessary we may also assume that $\|A_{\ppp^{n_0}}\|\diam(X) < \min\{\eta,r\}$. If there exists $\jjj \in \DD'$ such that $\fii_{\ppp^{n_0}}(X) \cap \fii_\jjj(X) \ne \emptyset$, then this word is unique, because otherwise two distinct sets in the family $\{\fii_\iii(X) : \iii \in \DD'\}$ would both meet $\fii_{\ppp^{n_0}}(X)$, whose diameter is smaller than $\eta$. Otherwise choose any $\jjj \in \DD'$. Defining
  \begin{equation*}
    \DD'' = (\DD' \setminus \{\jjj\}) \cup \{\ppp^{n_0}\},
  \end{equation*}
  we have $\#\DD'' = \#\DD'$. Indeed, if $\ppp^{n_0} \in \DD'$, then the uniqueness clause forces $\jjj = \ppp^{n_0}$ and hence $\DD'' = \DD'$, while if $\ppp^{n_0} \notin \DD'$, then removing $\jjj$ and adjoining the new word $\ppp^{n_0}$ preserves the cardinality. We see from the choice of $n_0$ that $(\fii_\iii)_{\iii \in \DD''}$ satisfies the strong separation condition, and from the inclusion $A_{\ppp^{n_0}} B(V_1,\rho) \subset \CC_0^o$ that
  \begin{equation*}
    A_\iii B(V_1,\rho) \subset \CC_0^o
  \end{equation*}
  for all $\iii \in \DD''$. Define
  \begin{align*}
    \VV_k' &= \biggl\{A_\jjj^{\land k} V : \jjj \in \bigcup_{m=0}^K \II^m \text{ and }V \in \VV_k\biggr\}, \\
    \WW_k &= \biggl\{(A_\jjj^{\land k})^{-1} V : \jjj \in \bigcup_{m=0}^K \II^m \text{ and }V \in \VV_k\biggr\}.
  \end{align*}
  Since $\A$ is strongly $k$-irreducible for all $k \in \{1,\ldots,d-1\}$, \cref{prop:simultaneous-escape} shows that $\A$ has the simultaneous escape property. Applying it to the finite family indexed by all pairs $(k,V)$ with $k \in \{1,\ldots,d-1\}$ and $V \in \VV_k'$, taking the forbidden family to be $\WW_k$ for the corresponding value of $k$, we obtain $\iii \in \II^*$ such that
  \begin{equation*}
    A_\iii^{\wedge k}V \notin \WW_k
  \end{equation*}
  for every $k \in \{1,\ldots,d-1\}$ and every $V \in \VV_k'$. Writing $\qqq = \overline{\hhh}(\iii) \iii \overline{\kkk}(\iii)$, we claim that
  \begin{equation*}
    A_\qqq^{\wedge k}V \notin \VV_k
  \end{equation*}
  for all $V \in \VV_k$. Indeed, if $A_\qqq^{\wedge k}V \in \VV_k$, then $A_\iii^{\wedge k} A_{\overline{\kkk}(\iii)}^{\wedge k} V \in (A_{\overline{\hhh}(\iii)}^{\wedge k})^{-1}\VV_k \subset \WW_k$. Since $A_{\overline{\kkk}(\iii)}^{\wedge k}V \in \VV_k'$, this contradicts the choice of $\iii$. Moreover, \cref{eq:dominated-ball-inclusion} gives $A_\qqq B(V_1,\rho) \subset \CC_0^o$.

  For each $V \in \bigcup_{k=1}^{d-1} \VV_k$, choose the word $\jjj_V \in \DD' \setminus \{\jjj\}$ injectively. This is possible because $\bigcup_{k=1}^{d-1} \VV_k$ has $M$ elements while $\#(\DD' \setminus \{\jjj\}) \ge M+1$. Since the family is finite, each map $A_{\jjj_V}$ is continuous, and $A_{\jjj_V}V_1 \in \CC_0^o$ for every $V \in \bigcup_{k=1}^{d-1} \VV_k$, there is an open neighborhood $U \subset B(V_1,\rho)$ of $V_1$ such that
  \begin{equation*}
    A_{\jjj_V} U \subset \CC_0^o
  \end{equation*}
  for all $V \in \bigcup_{k=1}^{d-1} \VV_k$. For each such $V$, choose $n_V \ge 1$ so large that
  \begin{equation*}
    A_{\jjj_V}^{\land k} (A_{\ppp}^{\land k})^{n_V} A_{\qqq}^{\land k} V \notin \VV_k
  \end{equation*}
  and
  \begin{equation*}
    A_{\ppp^{n_V}} A_\qqq B(V_1,\rho) \subset U.
  \end{equation*}
  Indeed, the second inclusion holds for all sufficiently large $n$ by \cref{lem:pinching-projective}. If the first inclusion failed for all sufficiently large $n$, then, as $\#\VV_k < \infty$, there would exist $m \ne m'$ such that
  \begin{equation*}
    A_{\jjj_V}^{\land k} (A_{\ppp}^{\land k})^m A_{\qqq}^{\land k} V = A_{\jjj_V}^{\land k} (A_{\ppp}^{\land k})^{m'} A_{\qqq}^{\land k} V,
  \end{equation*}
  contradicting \cref{eq:V-images-not-same} because $A_\qqq^{\land k}V \notin \VV_k$ and $A_{\jjj_V}^{\land k}$ is invertible. Define
  \begin{equation*}
    \JJ = \biggl(\DD'' \setminus \biggl\{\jjj_V : V \in \bigcup_{k=1}^{d-1} \VV_k\biggr\}\biggr) \cup \biggl\{\jjj_V\ppp^{n_V}\qqq : V \in \bigcup_{k=1}^{d-1} \VV_k\biggr\}.
  \end{equation*}
  The word $\ppp^{n_0}$ belongs to $\JJ$: it lies in $\DD''$, and it is not among the removed words $\jjj_V$, which are chosen from $\DD' \setminus \{\jjj\}$. Indeed, if $\ppp^{n_0} \notin \DD'$, this is immediate, while if $\ppp^{n_0} \in \DD'$, then the uniqueness clause forces $\jjj = \ppp^{n_0}$, so that $\ppp^{n_0} \notin \DD' \setminus \{\jjj\}$. Every word in $\JJ$ maps $B(V_1,\rho)$ into $\CC_0^o$: this is true for words in $\DD''$ by construction, and for the replacement words it follows from the choice of $U$. Consequently, $\A' = (A_\iii)_{\iii \in \JJ}$ is $1$-dominated. By the characterization of strong $k$-irreducibility from the notation section, it is enough to rule out finite families $\mathcal{F}$ of proper non-trivial subspaces of $\wedge^k\R^d$ such that
  \begin{equation*}
    A_\iii^{\land k} \mathcal{F} = \mathcal{F}
  \end{equation*}
  for all $\iii \in \JJ$. Since $\ppp^{n_0} \in \JJ$, the orbit $\{(A_\ppp^{\land k})^{mn_0}V : m \ge 0\}$ is finite for every $V \in \mathcal{F}$, and hence \cref{eq:V-images-not-same} implies that every $V \in \mathcal{F}$ belongs to $\VV_k$. Choosing the replacement word attached to such a subspace $V$, we obtain
  \begin{equation*}
    A_{\jjj_V\ppp^{n_V}\qqq}^{\land k} V \notin \VV_k,
  \end{equation*}
  contradicting the invariance of $\mathcal{F}$. Hence $\A'$ is strongly $k$-irreducible for all $k \in \{1,\ldots,d-1\}$. Since $\ppp^{n_0} \in \JJ$ and $A_\ppp^{\land k}$ is $1$-pinching, the tuple $\A'$ is also $k$-proximal for all $k \in \{1,\ldots,d-1\}$. Finally, since the sets $\fii_\iii(X)$ with $\iii \in \DD''$ are pairwise disjoint and $\fii_{\jjj_V\ppp^{n_V}\qqq}(X) \subseteq \fii_{\jjj_V}(X)$ for all $V \in \bigcup_{k=1}^{d-1} \VV_k$, the iterated function system $\Phi' = (\fii_\iii)_{\iii \in \JJ}$ satisfies the strong separation condition.

  Define
  \begin{equation*}
    \DD = \DD'' \setminus \biggl(\biggl\{\jjj_V : V \in \bigcup_{k=1}^{d-1} \VV_k\biggr\} \cup \{\ppp^{n_0}\}\biggr) \subset \JJ.
  \end{equation*}
  Then $\A'' = (A_\iii)_{\iii \in \DD}$ is a finite tuple whose generators are indexed by $\DD$, and every generator of $\A''$ maps $B(V_1,\rho)$ into $\CC_0^o$. Since $\DD \subset \JJ$, we have $P(\A',s_0) \ge P(\A'',s_0)$, and since $\varphi^{s_0}(A) = \|A\|^{s_0}$, the pressure lower bound \cref{eq:pressure-lower-bound} gives
  \begin{equation*}
    P(\A',s_0) \ge P(\A'',s_0) \ge \log\kappa_\rho^{2s_0} + \log\sum_{\iii \in \DD} \|A_\iii\|^{s_0}.
  \end{equation*}
  Recalling that $\|A_\iii\| \ge c_\ppp r$ for all $\iii \in \DD$, we obtain
  \begin{align*}
    P(\A',s_0) &\ge \log\kappa_\rho^{2s_0} + \log \#\DD + \log c_\ppp^{s_0}r^{s_0} \\
    &\ge \log\kappa_\rho^{2s_0} + \log(\#\DD''-M-1) + \log c_\ppp^{s_0}r^{s_0} \\
    &\ge \log\kappa_\rho^{2s_0} + \log(C_\ppp r^{-\udimm(X)+\delta}-M-1) + \log c_\ppp^{s_0}r^{s_0} > 0.
  \end{align*}
  Moreover, every $\iii \in \JJ$ is a finite word over the original alphabet, so $A_\iii$ is a product of contractive matrices and hence $\|A_\iii\| < 1$. Therefore the pressure is strictly decreasing, and thus
  \begin{equation*}
    \dimaff(\A') > s_0 > \min\{1,\udimm(X)\} - \eps.
  \end{equation*}
  To prove that $\dimaff(\A') < d$, let $X' \subseteq X$ be the self-affine set of $\Phi'$. Since $\Phi'$ satisfies the strong separation condition, we have
  \begin{equation*}
    \eta' = \min\{\dist(\fii_\iii(X'),\fii_\jjj(X')) : \iii, \jjj \in \JJ \text{ with } \iii \ne \jjj\} > 0.
  \end{equation*}
  Choose $0 < \tau < \eta'/3$ and write $O = \{x \in \R^d : \dist(x,X') < \tau\}$ and $O_\iii = \{x \in \R^d : \dist(x,\fii_\iii(X')) < \tau\}$ for all $\iii \in \JJ$. Since $X' = \bigcup_{\iii \in \JJ} \fii_\iii(X')$, we have
  \begin{equation*}
    O = \bigcup_{\iii \in \JJ} O_\iii,
  \end{equation*}
  and the sets $O_\iii$ are pairwise disjoint by the choice of $\tau$. Moreover, if $\lambda = \max_{\iii \in \JJ}\|A_\iii\| < 1$, then $\fii_\iii(O) \subset \{x \in \R^d : \dist(x,\fii_\iii(X')) < \lambda\tau\} \subset O_\iii$ for all $\iii \in \JJ$. Since $\fii_\iii(X')$ is non-empty and compact, the distance function to $\fii_\iii(X')$ takes every value in $[0,\tau]$ along a segment joining a point of $\fii_\iii(X')$ to a point outside $O_\iii$. Hence
  \begin{equation*}
    S_\iii = \{x \in O_\iii : \lambda\tau < \dist(x,\fii_\iii(X')) < \tau\}
  \end{equation*}
  is a non-empty open subset of $O_\iii \setminus \fii_\iii(O)$ for every $\iii \in \JJ$. In particular, $\LL^d(\fii_\iii(O)) < \LL^d(O_\iii)$ for all $\iii \in \JJ$. Since the sets $O_\iii$ are pairwise disjoint and $\fii_\iii(O) \subset O_\iii$, the sets $\fii_\iii(O)$ are pairwise disjoint, and therefore
  \begin{equation*}
    \sum_{\iii \in \JJ} |\det(A_\iii)| \LL^d(O) = \sum_{\iii \in \JJ} \LL^d(\fii_\iii(O)) < \sum_{\iii \in \JJ} \LL^d(O_\iii) = \LL^d(O).
  \end{equation*}
  Hence $\sum_{\iii \in \JJ} |\det(A_\iii)| < 1$. Since $\varphi^d(A) = |\det(A)|$ for all $A \in \GL(d,\R)$, it follows that
  \begin{equation*}
    P(\A',d) = \log\sum_{\iii \in \JJ} |\det(A_\iii)| < 0.
  \end{equation*}
  Therefore $\dimaff(\A') < d$. This completes the proof.
\end{proof}

\section{Projections of self-affine sets} \label{sec:proof-main-result}

It remains to deduce the two set projection theorems stated in the introduction from the subsystem construction and the measure theorem. We first prove \cref{thm:projection}: we pass to a strongly separated subsystem given by \cref{prop:ssc-approx}, choose on it a Bernoulli measure whose Lyapunov dimension is close to $\udimm(X)$, and then apply \cref{thm:projection-bernoulli}. We then specialize to the plane and prove \cref{thm:projection-two-dimensional}, where the algebraic input from \cref{sec:planar-converse} lets us dispense with Zariski density.

\begin{theorem} \label{thm:proj}
  Let $\Phi = (\fii_1,\ldots,\fii_N)$ be an affine iterated function system such that the associated tuple $\A = (A_1,\ldots,A_N) \in \GL(d,\R)^N$ of matrices is strongly pinching and strongly $k$-irreducible for all $k \in \{1,\ldots,d-1\}$, and let $X \subset \R^d$ be the self-affine set associated with $\Phi$. Then
  \begin{equation*}
    \dimh(\proj_V(X)) = \min\{1,\dimh(X)\} = \min\{1,\udimm(X)\}
  \end{equation*}
  for all $V \in \RP^{d-1}$.
\end{theorem}

\begin{proof}
  By \cref{eq:hausdorff-projection-lipschitz-estimate} and the general bound $\dimh(X) \le \udimm(X)$, it suffices to show that
  \begin{equation} \label{eq:main-claim}
    \dimh(\proj_V(X)) \ge \min\{1, \udimm(X)\}
  \end{equation}
  for all $V \in \RP^{d-1}$. There is nothing to prove if $\min\{1,\udimm(X)\} = 0$, so fix $0 < \eps < \frac{1}{2}\min\{1,\udimm(X)\}$. By \cref{prop:ssc-approx}, there exists a finite set $\JJ \subset \II^*$ for which $\A' = (A_\iii)_{\iii \in \JJ}$ is $1$-dominated, and $k$-proximal and strongly $k$-irreducible for all $k \in \{1,\ldots,d-1\}$ with
  \begin{equation} \label{eq:proof-main1}
    0 < \min\{1, \udimm(X)\} - \eps \le \dimaff(\A') < d,
  \end{equation}
  and $\Phi' = (\fii_\iii)_{\iii \in \JJ}$ satisfies the strong separation condition. Let $X' \subseteq X$ be the self-affine set of $\Phi'$ and set
  \begin{equation*}
    s = \min\{1,\dimaff(\A')\} - \eps.
  \end{equation*}
  Then $0 < s < 1$. Since every generator of $\A'$ is contractive and $s < \dimaff(\A')$, the pressure is strictly decreasing, and hence $P(\A',s) > 0$. By \cref{lem:domination-pressure}, there exists $0 < \kappa \le 1$ such that $\|A_{\iii\jjj}\| \ge \kappa^2\|A_\iii\|\|A_\jjj\|$ for all $\iii,\jjj \in \JJ^*$. Therefore, by the definition of the pressure, there exists $n \ge 1$ such that
  \begin{equation*}
    \kappa^{2s}\sum_{\iii \in \JJ^n} \|A_\iii\|^s > 1.
  \end{equation*}
  Let $\KK = \JJ^n$, write $\A'' = (A_\iii)_{\iii \in \KK}$ and $\Phi'' = (\fii_\iii)_{\iii \in \KK}$, and define
  \begin{equation*}
    p_\iii = \biggl(\sum_{\jjj \in \KK} \|A_\jjj\|^s\biggr)^{-1}\|A_\iii\|^s
  \end{equation*}
  for all $\iii \in \KK$. Let $\nu$ be the associated fully supported Bernoulli measure and let $\mu$ be the self-affine measure of $\Phi''$ associated with $\nu$. Since $\KK = \JJ^n$, the self-affine set of $\Phi''$ is $X'$, and $\Phi''$ satisfies the strong separation condition. Furthermore, $\A''$ is $k$-proximal and strongly $k$-irreducible for all $k \in \{1,\ldots,d-1\}$. Indeed, since $\A'$ is $k$-proximal and $k$-irreducible, there exists a word $\www$ over $\JJ$ such that $A_\www^{\wedge k}$ has a simple dominant eigenvalue; see \cite[Lemma~4.1]{BenoistQuint2016}. The $n$-fold concatenation $\www^n$ is a word over $\KK$, and $A_{\www^n}^{\wedge k} = (A_\www^{\wedge k})^n$ again has a simple dominant eigenvalue, so $\A''$ is $k$-proximal. If a finite union $\mathcal{F}$ of proper subspaces of $\wedge^k\R^d$ were invariant under all words in $\JJ^n$, then, writing
  \begin{equation*}
    \mathcal{G} = \bigcup_{\iii \in \bigcup_{m=0}^{n-1}\JJ^m} A_\iii^{\wedge k}\mathcal{F},
  \end{equation*}
  we would have $A_\jjj^{\wedge k}\mathcal{G} \subseteq \mathcal{G}$ for all $\jjj \in \JJ$. As each $A_\jjj^{\wedge k}$ is invertible, the descending chain $\mathcal{G} \supseteq A_\jjj^{\wedge k}\mathcal{G} \supseteq (A_\jjj^{\wedge k})^2\mathcal{G} \supseteq \cdots$ of finite unions of proper subspaces stabilizes, so that $A_\jjj^{\wedge k}\mathcal{G} = \mathcal{G}$ for every $\jjj \in \JJ$, contradicting the strong $k$-irreducibility of $\A'$. Since all probabilities $p_\iii$ are positive, the Bernoulli measure $\nu$ is fully supported, and therefore the self-affine measure $\mu$ is fully supported on $X'$. Since $0 < s < 1$, we have $\varphi^s(A) = \|A\|^s$ for all $A \in \GL(d,\R)$. We claim that
  \begin{equation} \label{eq:proof-main2}
    \diml(\nu) \ge s.
  \end{equation}
  Indeed, for all $m \ge 1$ and $\iii_1,\ldots,\iii_m \in \KK$,
  \begin{equation*}
    \|A_{\iii_1\cdots\iii_m}\| \ge \kappa^{2(m-1)}\|A_{\iii_1}\| \cdots \|A_{\iii_m}\|,
  \end{equation*}
  and therefore
  \begin{align*}
    h(\nu) + \Lambda(\A'',\nu,s) &\ge -\sum_{\iii \in \KK} p_\iii \log p_\iii + \log \kappa^{2s} + \sum_{\iii \in \KK} p_\iii \log \|A_\iii\|^s \\
    &= \log\biggl(\kappa^{2s}\sum_{\iii \in \KK} \|A_\iii\|^s\biggr) > 0.
  \end{align*}
  By the definition of the Lyapunov dimension, this proves \cref{eq:proof-main2}. Since $\dimaff(\A') > 0$, the set $\JJ$ has at least two elements, and therefore so does $\KK$. Thus the strong separation condition for $\Phi''$ implies that the maps in $\Phi''$ do not have a common fixed point, and it also implies the exponential separation condition. Hence \cref{thm:proj-dim} gives
  \begin{equation*}
    \dim((\proj_V)_*\mu) = \min\{1,\diml(\nu)\}
  \end{equation*}
  for all $V \in \RP^{d-1}$. Since $(\proj_V)_*\mu$ is supported on $\proj_V(X')$, the monotonicity of the Hausdorff dimension and \cref{eq:proof-main2} and \cref{eq:proof-main1} imply
  \begin{align*}
    \dimh(\proj_V(X)) &\ge \dimh(\proj_V(X')) \ge \dim((\proj_V)_*\mu) \\
    &= \min\{1,\diml(\nu)\} \ge \min\{1,\dimaff(\A')\}-\eps \\
    &\ge \min\{1,\udimm(X)\}-2\eps.
  \end{align*}
  As $\eps > 0$ is arbitrary, \cref{eq:main-claim} follows. This proves the theorem.
\end{proof}

The proof of \cref{thm:projection-two-dimensional} splits according to whether the tuple is $1$-proximal, and the two-dimensional structure theory of \cref{sec:planar-converse} reduces each case to a projection theorem already at our disposal. In the proximal case, a proximal strongly irreducible planar tuple is strongly pinching, so the conclusion follows from \cref{thm:projection}. In the non-proximal case, \cref{lem:planar-nonproximal-conformal} conjugates the system to one generated by contracting similarities whose orthogonal parts generate an infinite group, and a projection theorem for self-similar sets due to Farkas \cite{Farkas2016} yields the result.

\begin{theorem} \label{thm:proj-two-dimensional}
  Let $\Phi = (\fii_1,\ldots,\fii_N)$ be an affine iterated function system such that the associated tuple $\A = (A_1,\ldots,A_N) \in \GL(2,\R)^N$ of matrices is strongly irreducible, and let $X \subset \R^2$ be the self-affine set associated with $\Phi$. Then
  \begin{equation*}
    \dimh(\proj_V(X)) = \min\{1, \dimh(X)\} = \min\{1, \udimm(X)\}
  \end{equation*}
  for all $V \in \RP^1$.
\end{theorem}

\begin{proof}
  Since $d = 2$, strong irreducibility coincides with strong $1$-irreducibility. If $\A$ is $1$-proximal, then, being also strongly $1$-irreducible, it is strongly pinching by \cref{sec:irred} and strongly $k$-irreducible for all $k \in \{1,\ldots,d-1\}$, and the claim follows from \cref{thm:proj}.

  Suppose now that $\A$ is not $1$-proximal. By \cref{lem:planar-nonproximal-conformal} there is $M \in \GL(2,\R)$ such that $MA_iM^{-1} = \lambda_iO_i$ with $\lambda_i \in \R \setminus \{0\}$, $0 < |\lambda_i| < 1$, and $O_i \in \GO(2)$, and the group $\mathcal{T}$ generated by $O_1,\ldots,O_N$ is infinite. Recall that a \emph{similarity} of $\R^2$ is a map $y \mapsto \lambda Oy + b$ for some $\lambda \in \R \setminus \{0\}$ and $O \in \GO(2)$, and that the attractor of an iterated function system of contracting similarities is a \emph{self-similar set}. Each $\lambda_iO_i$ has $\|\lambda_iO_i\| = |\lambda_i| < 1$, so the maps $\psi_i(y) = \lambda_iO_iy + Mt_i$ are contracting similarities. Since $\psi_i = M \circ \fii_i \circ M^{-1}$, the attractor of $(\psi_1,\ldots,\psi_N)$ is $Y = M(X)$, a self-similar set whose group of orthogonal parts agrees with the infinite group $\mathcal{T}$ up to the signs of the $\lambda_i$, which are invisible on $\RP^1$.

  By the final assertion of \cref{lem:planar-nonproximal-conformal}, the orbit of every line under $\mathcal{T}$ is dense in $\RP^1$. Fix $V \in \RP^1$ and let $Q_V = \proj_V \circ M^{-1}$, a rank-one linear map, which we regard as a map into $\R$ by identifying the line $V$ isometrically with $\R$. By \cite[Theorem~1.6]{Farkas2016},
  \begin{equation*}
    \dimh(Q_V(Y)) = \min\{1,\dimh(Y)\}.
  \end{equation*}
  The linear isomorphism $M$ is bi-Lipschitz, so it preserves both Hausdorff and upper Minkowski dimension; in particular $\dimh(Y) = \dimh(X)$ and $\udimm(Y) = \udimm(X)$. We thus have
  \begin{equation*}
    \dimh(\proj_V(X)) = \dimh(Q_V(Y)) = \min\{1,\dimh(Y)\} = \min\{1,\dimh(X)\}
  \end{equation*}
  for all $V \in \RP^1$. As $Y$ is self-similar, Falconer \cite[Theorem~4]{Falconer1989} gives $\dimh(Y) = \udimm(Y)$, and therefore $\dimh(X) = \dimh(Y) = \udimm(Y) = \udimm(X)$. The desired equality follows.
\end{proof}

The strong irreducibility assumption cannot be dropped from \cref{thm:projection-two-dimensional}. For a counterexample, one may take $X$ to be a Bedford--McMullen carpet; see \cite{Bedford1984,McMullen1984}. Here the two coordinate axes are invariant, and the projection onto a coordinate axis is a self-similar subset of that axis whose dimension may be strictly smaller than $\min\{1,\dimh(X)\}$; see \cite{BaranyKaenmakiPyoralaWu2023-preprint,Pyorala2025} for a description of the set of exceptional directions for planar carpets. More generally, by the proof of \cite[Lemma~2.2]{BaranyKaenmakiRossi2021}, a proximal planar tuple that fails to be strongly irreducible is, after a change of basis, either reducible, that is, simultaneously upper triangular, or composed of diagonal and antidiagonal matrices. In each of these two classes there are self-affine sets possessing an \emph{exceptional direction}, namely a line $V \in \RP^1$ with $\dimh(\proj_V(X)) < \min\{1,\dimh(X)\}$. The following example exhibits such a set that is irreducible but not strongly irreducible.

\begin{example} \label{ex:sharpness}
  The irreducible but not strongly irreducible case is more delicate, because the antidiagonal maps interchange the two coordinate axes and thereby force the two coordinate projections of $X$ to have a common dimension, so that no exceptional direction can arise from a mere asymmetry between the axes. We exhibit an example in which this common value is nonetheless deficient. Consider the affine maps $\fii_1,\ldots,\fii_4 \colon \R^2 \to \R^2$,
  \begin{align*}
    \fii_1(x,y) &= (\tfrac18 x, \tfrac14 y + \tfrac14), & \fii_2(x,y) &= (\tfrac14 x + \tfrac14, \tfrac18 y), \\
    \fii_3(x,y) &= (\tfrac18 y, \tfrac14 x + \tfrac34), & \fii_4(x,y) &= (\tfrac14 y + \tfrac34, \tfrac18 x),
  \end{align*}
  and let $X \subset \R^2$ be the associated self-affine set. The linear parts of $\fii_1,\fii_2$ are diagonal and those of $\fii_3,\fii_4$ are antidiagonal. The diagonal matrix of $\fii_1$ has distinct entries, so it is pinching and its only invariant lines are the coordinate axes, while the antidiagonal part of $\fii_3$ interchanges these axes. Hence $\A$ is proximal, no line is invariant under all four matrices, and the pair of coordinate axes is invariant; that is, $\A$ is irreducible but not strongly irreducible. A direct computation shows that the four images of the square $Q = [0,1]^2$ under $\fii_1,\ldots,\fii_4$ are the pairwise disjoint subsets
  \begin{equation*}
    [0,\tfrac18] \times [\tfrac14,\tfrac12], \quad [\tfrac14,\tfrac12] \times [0,\tfrac18], \quad [0,\tfrac18] \times [\tfrac34,1], \quad [\tfrac34,1] \times [0,\tfrac18]
  \end{equation*}
  of $Q$. Hence $X \subset Q$, and the sets $\fii_1(X),\ldots,\fii_4(X)$ are pairwise disjoint, so $\Phi$ satisfies the strong separation condition.

  The system is invariant under the reflection $\iota(x,y) = (y,x)$, in the sense that $\iota\fii_1\iota = \fii_2$ and $\iota\fii_3\iota = \fii_4$; consequently $\iota(X) = X$ and the two coordinate projections coincide, $\proj_{V_1}(X) = \proj_{V_2}(X) = C$, where $V_1$ and $V_2$ denote the $x$-axis and the $y$-axis. Projecting $X = \bigcup_{i=1}^4 \fii_i(X)$ onto the $x$-axis and using $\proj_{V_1}(X) = \proj_{V_2}(X) = C$ gives
  \begin{equation} \label{eq:sharp-overlap}
    C = \tfrac18 C \cup (\tfrac14 C + \tfrac14) \cup \tfrac18 C \cup (\tfrac14 C + \tfrac34).
  \end{equation}
  The contributions of $\fii_1$ and $\fii_3$ are one and the same similarity $u \mapsto \tfrac18 u$, that is, an \emph{exact overlap}, and in particular the self-similar system generating $C$ does not satisfy the exponential separation condition. After discarding the repetition in \cref{eq:sharp-overlap}, the set $C$ is the attractor of the three maps $u \mapsto \tfrac18 u$, $u \mapsto \tfrac14 u + \tfrac14$, and $u \mapsto \tfrac14 u + \tfrac34$, whose images $[0,\tfrac18]$, $[\tfrac14,\tfrac12]$, and $[\tfrac34,1]$ are disjoint. This system satisfies the strong separation condition, and therefore
  \begin{equation} \label{eq:sharp-proj-dim}
    \dimh(\proj_{V_1}(X)) = \dimh(\proj_{V_2}(X)) = t,
  \end{equation}
  where $8^{-t} + 2 \cdot 4^{-t} = 1$, so that $t = \log_2(\frac{1+\sqrt{5}}{2}) = 0.6942\ldots < 1$.

  It remains to bound $\dimh(X)$ from below. For a word $\iii \in \{1,2,3,4\}^n$, the product $A_\iii$ is either diagonal or antidiagonal. Write its two non-zero entries as $2^{-m_1(\iii)}$ and $2^{-m_2(\iii)}$, where the first is the non-zero entry in the first row and the second is the non-zero entry in the second row. Since $|\det A_\iii| = 2^{-5n}$, we have $m_1(\iii) + m_2(\iii) = 5n$. We call $\iii$ balanced if $m_1(\iii) = m_2(\iii) = 5n/2$, which can occur only when $n$ is even. For a balanced word, $A_\iii$ is a similarity with contraction ratio $2^{-5n/2}$.

  It remains to count balanced words. Let $\Delta(\iii) = m_1(\iii)-m_2(\iii)$, and let $b_n(q)$ be the number of words $\iii \in \{1,2,3,4\}^n$ such that $\Delta(\iii) = q$. The reflection interchanging the two coordinate axes conjugates $A_1,A_3$ to $A_2,A_4$, respectively, and hence $b_n(q)=b_n(-q)$. If $\Delta(\iii)=q$, then prefixing $\iii$ by the symbols $1,2,3,4$ gives exponent differences $q+1,q-1,1-q,-1-q$, respectively. Therefore, using the symmetry $b_n(q)=b_n(-q)$,
  \begin{align*}
    b_{n+1}(q) &= b_n(q-1) + b_n(q+1) + b_n(1-q) + b_n(-1-q) \\
    &= 2b_n(q-1) + 2b_n(q+1).
  \end{align*}
  Since $b_0(0)=1$, it follows by induction that $b_n(q)=2^n s_n(q)$, where $s_n(q)$ is the number of nearest-neighbour walks on $\Z$ from $0$ to $q$ in $n$ steps; explicitly, $s_n(q)=\binom{n}{(n+q)/2}$. A word is balanced precisely when $\Delta(\iii)=0$, so for even $n$ the set $\BB_n$ of balanced words of length $n$ satisfies
  \begin{equation*}
    \# \BB_n = b_n(0) = 2^n s_n(0) = 2^n \binom{n}{n/2}.
  \end{equation*}
  The subsystem formed by the maps $\fii_\iii$ with $\iii \in \BB_n$ satisfies the strong separation condition, since its cylinder sets are contained in distinct level-$n$ cylinder sets of the original system. Its attractor $X_n$ is contained in $X$, and all its maps are similarities with contraction ratio $2^{-5n/2}$. Hence the Moran formula gives
  \begin{equation*}
    \dimh(X) \ge \dimh(X_n) = \frac{\log \#\BB_n}{(5n/2)\log 2} = \frac{\log\bigl(2^n \binom{n}{n/2}\bigr)}{(5n/2)\log 2}.
  \end{equation*}
  The central binomial coefficient $\binom{n}{n/2}$ is the largest of the $n+1$ binomial coefficients whose sum is $2^n$, and hence $2^n/(n+1) \le \binom{n}{n/2} \le 2^n$. Thus $\binom{n}{n/2}$ differs from $2^n$ only by a factor that grows at most polynomially in $n$, so $\log\binom{n}{n/2} = (1+o(1))n\log 2$. Consequently, the right-hand side above equals
  \begin{equation*}
    \frac{n\log 2 + \log\binom{n}{n/2}}{(5n/2)\log 2} = \frac{(2+o(1))n\log 2}{(5n/2)\log 2} = \frac45 + o(1),
  \end{equation*}
  and letting $n \to \infty$ through even integers gives $\dimh(X) \ge \tfrac45$.

  Since $8^{-4/5} + 2 \cdot 4^{-4/5} < 1$ and $t \mapsto 8^{-t} + 2 \cdot 4^{-t}$ is strictly decreasing, the number $t$ in \cref{eq:sharp-proj-dim} satisfies $t < \tfrac45$, and hence
  \begin{equation*}
    \dimh(\proj_{V_1}(X)) = \dimh(\proj_{V_2}(X)) = t < \tfrac45 \le \min\{1,\dimh(X)\}.
  \end{equation*}
  Therefore both coordinate axes are exceptional directions and the conclusion of \cref{thm:projection-two-dimensional} fails for $X$. In particular, strong irreducibility cannot be weakened to irreducibility, even under the strong separation condition.
\end{example}

\bibliographystyle{abbrv}
\bibliography{Bibliography}

\end{document}